\documentclass[reqno,english]{amsart}
\usepackage{amsmath}
\usepackage{amsthm}
\usepackage{amssymb}
\usepackage{mathrsfs}
\usepackage{mathtools}
\usepackage{comment}
\usepackage{tikz-cd}
\usepackage[colorlinks=true]{hyperref}
\hypersetup{urlcolor=blue, citecolor=red}

\usepackage[left=1.5in,right=1.5in,top=1.5in,bottom=1.5in,footskip=.25in]{geometry}

\usepackage[square,numbers]{natbib}

\newtheorem{theorem}{Theorem}[section]
\newtheorem{corollary}[theorem]{Corollary}

\newtheorem{lemma}[theorem]{Lemma}
\newtheorem{proposition}[theorem]{Proposition}

\theoremstyle{definition}

\newtheorem{remark}[theorem]{Remark}

\newtheorem*{remark*}{Remark}

\newcommand{\RR}{\mathbb{R}}

\newcommand{\SSS}{\mathbb{S}}
\newcommand{\ca}[1]{\mathcal{#1}}
\newcommand{\mrm}[1]{\mathrm{#1}}

\newcommand{\dd}{\, {\rm d}}

\newcommand{\de}{\delta}
\newcommand{\De}{\Delta}
\newcommand{\ga}{\gamma}

\newcommand{\p}{\partial}

\newcommand{\na}{\nabla}

\newcommand{\abs}[1]{\left |#1\right |}
\newcommand{\norm}[1]{\left\|#1\right\|}
\newcommand{\br}[1]{\left \langle#1\right\rangle}
\newcommand{\set}[1]{\left \{#1\right \} }

\newcommand{\PP}{\mathbb{P}}
\newcommand{\UU}{\mathbb{U}}
\newcommand{\VV}{\mathbb{V}} 
\newcommand{\HH}{{\mathcal{H}}} 

\newcommand{\RE}{\operatorname{Re}}

\newcommand{\Mref}{\underline{M}}
\newcommand{\Lref}{\underline{L}}

\newcommand{\uref}{{\mathfrak{u}_-}} 
\newcommand{\NSG}{{\mathbb{N}^{s,\gamma}}}

\newcommand{\bfu}{u} 
\newcommand{\uns}{\bar{\mathfrak{u}}_{\mathrm{NS}}} 

\newcommand{\sfs}{\mathsf{s}}
\newcommand{\sfc}{\mathsf{c}}

\title[Shock profiles for non-cutoff Boltzmann]{Shock profiles for the non-cutoff Boltzmann equation with hard potentials}

\author{Dominic Wynter}

\date{\today}

\begin{document}

\maketitle

	\begin{abstract}
		The Boltzmann equation models gas dynamics in the low density or high Mach number regime, using a statistical description of molecular interactions. Planar shock wave solutions have been constructed for the Boltzmann equation for hard potentials with angular cutoff, and more recently for the Landau equation of plasma dynamics. In this work, we construct shock profile solutions for the Boltzmann equation where the molecular interactions are long-range, and we show these solutions to be smooth and well approximated by compressible Navier Stokes shock profiles. Our proof procedes by standard energy estimates and a quantitative Chapman-Enskog approximation.
	\end{abstract}

	\tableofcontents
	
\section{Introduction}

Shock wave phenomena are of great importance in the theory of gas dynamics, and can be modelled in various ways. The theory of steady planar shock waves is particularly well-developed, and has been studied using the compressible Euler equations, the compressible Navier-Stokes equations, and using kinetic models such as the Boltzmann and Landau equations. For one-dimensional flows, we consider the compressible Euler equation for a polytropic gas,
\begin{align}
	\label{Euler_equation}
	\begin{split}
		&\p_t\rho +\p_x(\rho\bfu) = 0
		\\
		&
		\p_t(\rho\bfu) + \p_x(\rho\bfu^2 + \rho T) = 0
		\\
		&
		\p_t\left( \frac12\rho\bfu^2 + \frac32\rho T\right)
		+
		\p_x\left( \frac12\rho\bfu^3 + \frac52\rho\bfu T\right) 
		= 0,
	\end{split}
\end{align}
expressed here in the planar setting for the pressure law $p=\rho T$. Steady shock waves for this equation take the form of jump discontinuities moving at speed $\sfs$, with constant data $(\rho_-,\bfu_-,T_-)$ and $(\rho_+,\bfu_+,T_+)$ to the left and right of the discontinuity respecively \cite{CourantFriedrichs48}. Such shock solutions have been constructed for boundary data and shock speed satisfying the Rankine-Hugoniot and Lax entropy conditions \cite{Lax73}. Here $(\rho,\bfu,T)\in \RR_{>0}\times\RR^1\times\RR_{>0}$ represent the density, velocity, and temperature, where the velocity $(u,0,0)\in\RR^3$ is taken parallel to the $x$ axis.

When viscous and thermal conduction effects are considered, as in the compressible Navier-Stokes equations \eqref{Compressible_Navier-Stokes_Equations}, shock solutions $(\rho,\bfu,T)(t,x) = (\rho,\bfu,T)(x-\sfs t)$ are known to exist with the same speed $\sfs$ and boundary conditions $(\rho_\pm,\bfu_\pm,T_\pm)$ \cite{Gilbarg51}. For small-amplitude shocks, the compressible Navier Stokes equations provide solutions close to experimental observations, however, discrepancies are observed inside the shock layer for strong shocks, where more accurate results have been provided by the Boltzmann equation, which models gas dynamics using a statistical description of molecular interactions. Shock profiles for the Boltzmann equation, which travel with speed $\sfs$, interpolate smoothly between the velocity distributions
\begin{align}
	\label{shock_profile_boundary_conditions}
	F_\pm(\xi)=\lim_{x\to\pm\infty} F(x,\xi) = \frac{\rho_\pm}{(2\pi T_\pm)^{3/2}} e^{-\abs{\xi - \bfu_\pm e_1}^2/T_\pm}
\end{align}
taking limits in $L^1(1+\abs\xi^2)$,
for $(\rho_\pm,\bfu_\pm,T_\pm)$ and $\sfs$ satisfying the Rankine-Hugoniot and Lax entropy conditions, where $F_\pm(\xi)$ represent equilibrium solutions to the Boltzmann equation. Such solutions have been modelled numerically \cite{Bird70}, providing results close to experiment \cite{LiepmannNarasimhaChahine62} \cite{Alsmeyer76} \cite{FiszdonHerczynskiWalenta76}. In this work, we construct small-amplitude shock profiles to the non-cutoff Boltzmann equation with hard potentials, corresponding to power-law intermolecular interactions. The non-cutoff Boltzmann equation
is a formal statistical model developed by Maxwell for collisional gases interacting by long-range interaction potentials \cite{Maxwell67}, which is considered a more realistic assumption on the microscopic dynamics. More precisely, we construct solutions to the traveling wave non-cutoff Boltzmann equation
\begin{align}
	\label{traveling_wave_Boltzmann}
	(\xi_1 - \sfs)\p_x F = Q(F,F),\qquad (x,\xi)\in\RR\times\RR^3.
\end{align}
The collision term $Q$ is a bilinear operator defined by
\begin{align}
	\label{collision_operator_definition}
	\begin{split}
		Q(G,F)(\xi) &= \int_{\RR^3_{\xi_*}}\!\int_{\SSS^2_\sigma}\!
		B\left( \abs{\xi-\xi_*},\sigma\cdot\frac{\xi-\xi_*}{\abs{\xi-\xi_*}}\right)
		\left[ G(\xi_*')F(\xi') - G(\xi_*)F(\xi)\right]\dd\sigma\dd\xi_*
	\end{split}
\end{align}
where
\begin{align*}
	\xi' &= \frac{\xi + \xi_*}{2} + \frac{\xi-\xi_*}{2}\sigma,
	\quad
	\xi_*' = \frac{\xi + \xi_*}{2} + \frac{\xi-\xi_*}{2}\sigma,
\end{align*}
and where we define the angle
\begin{align}
	\label{angle_definition}
	\cos\theta = \sigma\cdot\frac{\xi-\xi_*}{\abs{\xi-\xi_*}},
\end{align}
for some nonnegative collision kernel $B\ge0$. We remark that when $F=G$, we can symmetrize the collision kernel $B$ as 
\begin{align*}
	&
	\bar{B}\left(\abs{\xi-\xi_*},\cos\theta\right) 
	= 
	\left[
	B\left(\abs{\xi-\xi_*},\cos\theta\right) 
	+
	B\left(\abs{\xi-\xi_*},-\cos\theta\right) 
	\right] 1_{\cos\theta\ge0},
\end{align*}
using the integral transformation $\sigma\mapsto-\sigma$ and interchanging $\xi'$ and $\xi_*'$, so we may assume that $B$ is supported on $\theta\in(0,\pi/2]$. We make the following further hypotheses.

\bigskip
\noindent\textbf{The collision kernel.} 
We impose the following assumptions on the collision kernel\linebreak $B(\abs{\xi-\xi_*},\cos\theta)$.
\begin{itemize}
	\item We assume $B(\abs{\xi-\xi_*},\cos\theta)$ can be decomposed in the form
	\begin{align}
		\label{Hypothesis1}
		B(\abs{\xi-\xi_*},\cos\theta)
		=
		\abs{\xi-\xi_*}^\ga b_s(\cos\theta)
	\end{align}
	for some $\ga\in\RR$ and some measurable function $b_s\ge0$.
	
	\item We assume the angular term has bounds
	\begin{align}
		\label{Hypothesis2}
		\begin{split}
			\frac{1}{c_b\theta^{1+2s}}\le b_s(\cos\theta)\sin\theta \le \frac{c_b}{\theta^{1+2s}},
			\qquad&\theta\in \left(0,\frac\pi2\right],
			\\
			b_s(\cos\theta) = 0,\qquad &\theta\in\left(\frac\pi2,\pi\right],
		\end{split}
	\end{align}
	for some $c_b>0$ and some $s\in(0,1)$.
	
	\item We assume \emph{hard potential} interactions, or more precisely, that
	\begin{align}
		\label{Hypothesis3}
		s\in(0,1/2),\qquad\ga\in(0,1).
	\end{align}
\end{itemize}
We will assume that the conditions \eqref{Hypothesis1}, \eqref{Hypothesis2}, and \eqref{Hypothesis3} hold for the collision kernel $B$ and the associated collision operator $Q$ for the rest of this article.
We note that hypothesis \eqref{Hypothesis2} ensures that the collision operator \eqref{collision_operator_definition} is a singular integral. The main motivation for these assumptions comes from particles interacting according to the radial power law potential 
\begin{align*}
	\phi(r) = \frac{1}{r^{p-1}},\qquad p\in(5,\infty).
\end{align*}
The collision kernel for such potentials was first computed in $\RR^3$ by Maxwell \cite{Maxwell67} in 1867, who showed assumptions \eqref{Hypothesis1}, \eqref{Hypothesis2}, and \eqref{Hypothesis3} are satisfied for the parameters $s=1/(p-1)$ and $\ga = (p-5)/(p-1)$ (see Chapter II, Section 5 of \cite{Cercignani88}). Moreover, simulations of Boltzmann shock profiles for the power $p=10$ have shown particularly good agreement with experiment for monoatomic gases \cite{FisckoChapman89}, further motivating our work. 

We now recall basic facts about the collision operator. We have the identities
\begin{align}
	\label{Q_is_microscopic}
	\int_{\RR^3}\! Q(F,F)\begin{pmatrix} 1 \\ \xi \\ \frac12\abs\xi^2 \end{pmatrix} \dd\xi
	=
	\begin{pmatrix}
		0 \\ 0 \\ 0
	\end{pmatrix}
\end{align}
corresponding to conservation of the number of particles, and conservation of momentum and energy in elastic collisions. If we define the reflection operators
\begin{align}
	\label{reflection_operators}
	R_2 F(\xi_1,\xi_2,\xi_3) = F(\xi_1,-\xi_2,\xi_3),\qquad
	R_3 F(\xi_1,\xi_2,\xi_3) = F(\xi_1,\xi_2,-\xi_3),
\end{align}
we note that by the commutator identities $[R_2,\xi_1\p_x] = [R_3,\xi_1\p_x]=0$ and invariance of the collision operator $Q$ under Galilean transformations, $R_2 F$ and $R_3 F$ will also be solutions to \eqref{traveling_wave_Boltzmann}. By the choice of boundary conditions \eqref{shock_profile_boundary_conditions}, we have $R_2 F_\pm = R_3 F_\pm = F_\pm$, and therefore, it is reasonable to assume that any unique solution $F$ to \eqref{traveling_wave_Boltzmann} with these boundary conditions will satisfy $$R_2 F = R_3 F=F,$$ so we require this condition on $F$ for the rest of this section. The main advantage of this condition will be to simplify the 1D viscous shock stability analysis in Section \ref{Macro_section}, however it could be removed at the cost of more intricate fluid stability estimates. If we further assume that $F>0$ and $F\in L^\infty_x L^1_\xi(1+\abs\xi^2)$, then we can define the associated macroscopic quantities
\begin{align}
	\label{hydrodynamic_quantities_definition_0}
	I[F]:=
	\begin{pmatrix}
		\rho \\ m \\ E
	\end{pmatrix}
	=
	\int_{\RR^3}\! F(\xi) \begin{pmatrix} 1 \\ \xi _1\\ \frac12\abs\xi^2 \end{pmatrix} \dd\xi \in\RR_{>0}\times\RR\times\RR_{>0}
\end{align}
which are preserved by collisions,
where $m$ denotes the momentum, and $E$ the energy, which can re-express as
\begin{align}
	\label{hydrodynamic_quantities_definition}
	\mathsf{f}_0(\rho,\bfu,T):=
	\begin{pmatrix} 
		\rho 
		\\
		\rho\bfu 
		\\
		\frac12 \rho\bfu^2 + \frac32\rho T
	\end{pmatrix}
	=
	\begin{pmatrix}
		\rho \\ m \\ E
	\end{pmatrix}
\end{align}
and we see that
\begin{equation}
	\label{hydrodynamic_quantities_range}
	(\rho,\bfu,T)\in\RR_{>0}\times\RR^1\times\RR_{>0}
\end{equation}
holds whenever the integral in \eqref{hydrodynamic_quantities_definition_0} is well-defined. For any hydrodynamic quantities $(\rho,\bfu,T)$ satisfying \eqref{hydrodynamic_quantities_range}, we can define the associated Maxwellian distribution
\begin{align}
	\label{Maxwellian_distribution_definition}
	M_{(\rho,\bfu,T)}(\xi) = \frac{\rho}{(2\pi T)^{3/2}}e^{-\abs{\xi - \bfu e_1}^2/T}
\end{align}
satisfying
\begin{align}
	\label{Maxwellian_hydrodynamic_quantities}
	\int_{\RR^3}\! M_{(\rho,\bfu,T)}(\xi) \begin{pmatrix} 1 \\ \xi _1\\ \frac12\abs\xi^2 \end{pmatrix} \dd\xi
	=
	\mathsf{f}_0(\rho,\bfu,T),
\end{align}
and which is an equilibrium for the collision operator in the sense that
\begin{equation*}
	\label{Maxwellians_are_equilibria}
	Q\left(M_{(\rho,\bfu,T)},M_{(\rho,\bfu,T)}\right)=0
\end{equation*}
holds, and it can be shown from the $H$-theorem that all sufficiently regular and decaying functions $F>0$ satisfying $Q(F,F)=0$ are of this form (see Chapter 3 of \cite{CercignaniIllnerPulvirenti94}). For small-amplitude shocks, it is reasonable to assume that solutions will be close to equilibrium, so computing the moment identities
\begin{align}
	\label{Maxwellian_moment_identities}
	\begin{split}
		\int_{\RR^3}\! M_{(\rho,\bfu,T)}(\xi)\xi_1 
		\begin{pmatrix} 1 \\ \xi _1\\ \frac12\abs\xi^2 \end{pmatrix}
		\dd\xi
		&=
		\begin{pmatrix}
			\rho\bfu \\ \rho\bfu^2 + \rho T \\ \frac12 \rho\bfu^3 + \frac52 \rho\bfu T
		\end{pmatrix}
		\\[3pt]
		&
		=:
		\mathsf{f}_1(\rho,\bfu,T),
	\end{split}
\end{align}
we see that solutions to the Boltzmann equation \eqref{traveling_wave_Boltzmann} can be formally approximated by $M_{(\rho,\bfu,T)}$ for solutions $(\rho,\bfu,T)(t,x) = (\rho,\bfu,T)(x-\sfs t)$ to the Euler equation \eqref{Euler_equation}.
Integrating \eqref{traveling_wave_Boltzmann} in space with boundary conditions \eqref{shock_profile_boundary_conditions}, we get the Rankine-Hugoniot conditions
\begin{align}
	\label{Rankine-Hugoniot_conditions}
	\begin{split}
		\mathsf{f}_1(\rho_-,\bfu_-,T_-) - \mathsf{f}_1(\rho_+,\bfu_+,T_+) 
		=
		\sfs\big[\mathsf{f}_0(\rho_-,\bfu_-,T_-) - \mathsf{f}_0(\rho_+,\bfu_+,T_+)\big].
	\end{split}
\end{align}
For any left reference state $(\rho_-,\bfu_-,T_-)$ satisfying \eqref{hydrodynamic_quantities_range}, we define the \emph{Rankine-Hugoniot locus} as the set of $(\rho_+,\bfu_+,T_+)$ satisfying \eqref{Rankine-Hugoniot_conditions} for some $\sfs\in\RR$. To solve for $(\rho_+,\bfu_+,T_+)$ sufficiently close to $(\rho_-,\bfu_-,T_-)$, we rewrite the Rankine-Hugoniot conditions in terms of the conservative variables $(\rho,m,E)$ defined in \eqref{hydrodynamic_quantities_definition}, giving
\begin{align*}
	\left(\mathsf{f}_1\circ\mathsf{f}_0^{-1}\right)(\rho_-,m_-,E_-)
	-
	\left(\mathsf{f}_1\circ\mathsf{f}_0^{-1}\right)(\rho_+,m_+,E_+)
	=
	\sfs\left[(\rho_-,m_-,E_-)-(\rho_+,m_+,E_+)\right],
\end{align*}
which we see is an approximate eigenvalue problem for $\left(\mathsf{f}_1\circ\mathsf{f}_0^{-1}\right)^{-1}(\rho_-,m_-,E_-)$. If we compute the derivatives
\begin{align*}
	\mathsf{f}_0'(\rho,u,T)=
	\begin{pmatrix}
		1 & 0 & 0 \\
		u & \rho & 0 \\
		\frac12 u^2 + \frac32 T & \rho u & \frac32 \rho
	\end{pmatrix}
\end{align*}
and
\begin{align*}
	\mathsf{f}_1'(\rho,u,T)=
	\begin{pmatrix}
		\bfu & \rho & 0 \\
		\bfu^2 + T & 2\rho\bfu & \rho \\
		\frac12\abs{\bfu}^2\bfu + \frac52\bfu T
		&
		\frac32\rho\bfu^2 + \frac52\rho T
		&
		\frac52\rho\bfu
	\end{pmatrix},
\end{align*}
we can then compute that the matrix $\mathsf{f}_1'(\mathsf{f}_0')^{-1}(\rho,m,E)$ has the distinct eigenvalues 
\begin{align*}
	\lambda_1 = \bfu - \sfc,\quad\lambda_2 = \bfu,\quad\lambda_3= \bfu + \sfc
\end{align*}
where
\begin{align*}
	\sfc = \sqrt{\frac{5T}{3}}
\end{align*}
is the speed of sound. By standard perturbation theory for simple eigenvalues, this implies that the shock speed $\sfs$ will be close to $\lambda_i$ for some $i=1,2,3$. It will be useful to compute the right eigenvectors in $(\rho,u,T)$ coordinates, namely we have vectors
\begin{align}
	\label{fluid_eigenvectors}
	\mathbf{r}_1 = \frac{1}{4\sfc} \begin{pmatrix} -3\rho \\ 3 \sfc \\ -2T \end{pmatrix}
	,
	\quad
	\mathbf{r}_2 = \begin{pmatrix} -3\rho \\ 0 \\ \bfu^2 +3T\end{pmatrix},
	\quad
	\mathbf{r}_3 = \frac{1}{4\sfc}\begin{pmatrix} 3\rho \\ 3\sfc  \\ 2T
	\end{pmatrix},
\end{align}
which satisfy the generalized eigenvalue problem
\begin{align}
	\label{three_fluid_eigenvalues}
	\left[ \lambda_i \mathsf{f}_0' - \mathsf{f}_1'\right] \mathbf{r}_i = 0,\qquad i=1,2,3.
\end{align}

We can now characterize the Rankine-Hugoniot locus more precisely. Writing $V=(\rho,u,T)$ for the hydrodynamic variables, in a neighbourhood of $V_-$, by an implicit function theory argument, we see that the Rankine-Hugoniot locus then consists of three smooth curves intersecting at $V_-$ and tangent to $\mathbf{r}_{i}(V_-)$ at $V_-$, for $i=1,2,3$ (cf. Chapter 8 of \cite{Dafermos_book}). To construct shocks in the kinetic or fluid setting, we must also impose the \emph{Lax entropy condition} 
\begin{align}
	\label{Lax_entropy}
	\lambda_i(V_+)\le \sfs \le \lambda_i(V_-),
\end{align}
first introduced in \cite{Lax57}, which guarantees that the end states $(V_-,V_{+,i}(\epsilon))$  correspond to admissible jump solutions to the Euler equations \eqref{Euler_equation} travelling with speed $\sfs_i(\epsilon)$. The eigenvalues $\lambda_1,\lambda_3$ are genuinely nonlinear, satisfying $\mathbf{r}_1(V)\cdot\na\lambda_1(V)=\mathbf{r}_3(V)\cdot\na\lambda_3(V)=1$ by our normalization, and correspond to leftward or rightward compressive shocks, while the middle eigenvalue $\lambda_2$ is linearly degenerate in the sense that $\mathbf{r}_2(V)\cdot\na\lambda_2(V)=0$, and corresponds to a contact discontinuity. For the genuinely nonlinear eigenvalues $i=1,3$, the Rankine-Hugoniot curve $V_{+,i}(\epsilon)$ can be parametrized such that
\begin{align*}
	\lambda_i\big(V_{+,i}(\epsilon)\big) = \lambda_i(V_-) - \epsilon
\end{align*}
with asymptotics
\begin{align}
	\label{Rankine-Hugoniot_parametrization}
	V_{+,i}(\epsilon) = V_- - \epsilon\mathbf{r}_i(V_-) + O(\epsilon^2)
\end{align}
and shock speed
\begin{align}
	\label{shock_speed_perturbation}
	\sfs_i(\epsilon) = \lambda_i(V_-) - \frac\epsilon2 + O(\epsilon^2)
\end{align}
(see Theorem 5.14 of \cite{HoldenRisebro_book}), such that the Lax entropy condition \eqref{Lax_entropy} is satisfied for $\epsilon>0$. For kinetic shocks with small amplitude $\epsilon$, we make the assumption of a length scale $1/\epsilon$, so that the Knuden number, measuring the ratio of the mean free path to the length scale, is of order $\epsilon$. In this scaling, we can formally derive the compressible Navier-Stokes equations 
\begin{align}
	\label{Compressible_Navier-Stokes_Equations}
	\p_t\begin{pmatrix} 
		\rho 
		\\
		\rho\bfu 
		\\
		\frac12 \rho\bfu^2 + \frac32\rho T
	\end{pmatrix}
	+
	\p_x
	\begin{pmatrix}
		\rho\bfu \\ \rho\bfu^2 + \rho T \\ \frac12 \rho\bfu^3 + \frac52 \rho\bfu T
	\end{pmatrix}
	=
	\p_x
	\left[
	\begin{pmatrix}
		0 & 0 & 0 \\
		0 & \mu & 0 \\
		0 & \mu\bfu & \varkappa
	\end{pmatrix}
	\p_x
	\begin{pmatrix}
		\rho \\ \bfu \\ T
	\end{pmatrix}
	\right]
\end{align}
as a second order Chapman-Enskog expansion for the Boltzmann equation (cf. \cite{KawashimaMatsumuraNishida79} \cite{BardosGolseLevermore91}),
where the viscosity and thermal conductivity $\mu,\varkappa$ are defined by
\begin{align*}
	&
	\mu = -T^{1-\ga/2}\int_{\RR^3}\!\left(\xi_1^2 - \frac13\abs\xi^2\right) \Lref^{-1}
	\left[\left(\xi_1^2 - \frac13\abs\xi^2\right)\Mref\right]
	\dd\xi>0,
	\\[0.6em]
	&
	\varkappa = -T^{1-\ga/2}\int_{\RR^3}\! \xi_1\left(\abs\xi^2 -5\right) \Lref^{-1}
	\left[\xi_1\left(\abs\xi^2 -5\right)\Mref \right]
	\dd\xi>0,
\end{align*}
where we have defined the reference Maxwellian
\begin{align*}
	\Mref(\xi) = M_{(1,0,1)}(\xi) = \frac{1}{(2\pi)^{3/2}}e^{-\abs\xi^2/2},
\end{align*}
and the reference linearized collision operator
\begin{align*}
	\Lref f = Q(\Mref,f) + Q(f,\Mref).
\end{align*}

We fix the left reference state $V_- = (\rho_-,\bfu_-,T_-) = (1,0,1)$, and let $V_+ = V_{+,3}(\epsilon)$ be the associated right hand state on the third Rankine-Hugoniot curve, uniquely determined such that $\lambda_3(V_+) = \lambda_3(V_-) - \epsilon$ for $\epsilon>0$ sufficiently small, and we let $\sfs = \sfs_3(\epsilon)$ be the associated shock speed. For such boundary data, we let $(\bar\rho,\bar u,\bar T)(t,x) = (\bar\rho,\bar u,\bar T)(x-\sfs t)$ be a smooth solution to \eqref{Compressible_Navier-Stokes_Equations}, which we know exists uniquely up to translation \cite{Gilbarg51}. 

We now define the spaces in which we will construct solutions. We first define Hilbert space norm
\begin{align*}
	\norm{F}_{\HH^0} = \norm{\Mref^{-1/2}F}_{L^2(\RR^3)}
\end{align*}
and the associated Hilbert space
\begin{align}
	\label{definition_of_HH}
	\HH^0 = \set{F\in L^2(\RR^3)\,\middle|\, \norm{F}_{\HH^0}<\infty,R_2 F = R_3 F = F}.
\end{align}
A perturbative global-well posedness theory for solutions $F = \Mref + \Mref^{1/2}f$ to the non-cutoff Boltzmann equation was developed by the group of Alexandre, Morimoto, Ukai, Xu, and Yang in the papers \cite{AMUXY11_Maxwellian} \cite{AMUXY11_hard_potentials} \cite{AMUXY12_soft_potentials},
and simultaneously by Gressman and Strain \cite{GressmanStrain11}. For our result, we work in the norm defined in \cite{GressmanStrain11},
\begin{align}
	\label{GS11norm}
	\norm{f}_{N^{s,\gamma}}
	=
	\norm{f}_{L^2_{\gamma+2s}}
	+
	\int_{\RR^3\times\RR^3}\! 
	\left(\br\xi \br{\xi'}\right)^{\frac{\gamma+2s+1}{2}}
	\frac{(f' - f)^2}{d(\xi,\xi')^{3+2s}} 1_{d(\xi,\xi')\le1} \dd\xi'\dd\xi,
\end{align}
for the bracket $\br\xi = (1+\abs\xi^2)^{1/2}$, where the metric
\begin{align*}
	d(\xi,\xi') = \sqrt{\abs{\xi - \xi'}^2 + \left( \abs{\xi}^2 - \abs{\xi'}^2\right)^2/4}
\end{align*}
is the distance on a lifted paraboloid. We additionally have the norm bounds
\begin{align}
	\label{NSG_Sobolev_bounds}
	\norm{\br\xi^{\gamma/2+s}f}_{L^2} + \norm{\br\xi^{\gamma/2}f}_{H^s}
	\lesssim
	\norm{f}_{N^{s,\gamma}}
	\lesssim
	\norm{\br\xi^{\gamma/2+s}f}_{H^s}
\end{align}
proved in Section 2 of \cite{GressmanStrain11}, where the  Sobolev norms $H^s$ are defined by $\norm{f}_{H^s_\xi} = {\|\!\br\zeta^s\! \hat f(\zeta)\|}_{L^2_\zeta}$ where $\hat f(\zeta)$ represents the Fourier transform of $f$. The well-posedness theory in \cite{GressmanStrain11} depends on precise coercivity estimates for the linearized operator $\Lref$ and trilinear estimates for the nonlinearity, proved by an intricate Littlewood-Paley argument on the lifted paraboloid, and we make use of these estimates extensively in our result. More precisely, we define the norm
\begin{align}
	\label{HH1_norm}
	\norm{f}_{\HH^1} = \norm{\Mref^{-1/2}f}_{N^{s,\gamma}},
\end{align}
and we define the Hilbert space $\HH^1$ as the set of measurable $f:\RR^3\to\RR$ such that $\Mref^{-1/2}f\in N^{s,\gamma}$ and $R_2 f = R_3 f = f$. We see from \eqref{NSG_Sobolev_bounds} and the hard potential assumption \eqref{Hypothesis3} that we have the norm bound $\norm{f}_{L^2}\le C\norm{f}_{N^{s,\gamma}}$ for some $C>0$, implying the norm bound
\begin{align}
	\label{NSG_coercivity}
	\norm{f}_{\HH^0}\le C\norm{f}_{\HH^1}.
\end{align} 
Finally, for any separable Banach space $\mathsf{E}$ and any integer $k\ge0$, we define the rescaled Bochner-Sobolev space norms
\begin{align*}
	\norm{g}_{H^k_\epsilon\mathsf{E}}^2 = \sum_{j=0}^k\epsilon^{1-2k}\norm{\p_x^j g}_{L^2_x(\RR;\mathsf{E})}^2,
\end{align*}
and similarly define the norm $\norm\cdot_{L^2_\epsilon\mathsf{E}} = \norm{\cdot}_{H^0_\epsilon\mathsf{E}}$.

We now present the main result of this article.

\begin{theorem}
	\label{main_theorem}
	Let $\epsilon\in(0,\epsilon_0)$ for $\epsilon_0=\epsilon_0(k)>0$ sufficiently small, let assumptions \eqref{Hypothesis1}, \eqref{Hypothesis2}, and \eqref{Hypothesis3} hold on the collision kernel, and let $(\bar\rho_{\mrm{NS}},\bar{u}_{\mrm{NS}},\bar T_{\mrm{NS}})(x-\sfs t)$ be a shock solution to the compressible Navier-Stokes equations \eqref{Compressible_Navier-Stokes_Equations} with boundary conditions
	\begin{align*}
		&\lim_{x\to-\infty} \left(
		\bar\rho_{\mrm{NS}},\bar{u}_{\mrm{NS}},\bar T_{\mrm{NS}}
		\right)(x) = V_- := (1,0,1),
		\\
		&
		\lim_{x\to+\infty} \left(
		\bar\rho_{\mrm{NS}},\bar{u}_{\mrm{NS}},\bar T_{\mrm{NS}}
		\right)(x) = V_{+,3}(\epsilon),
	\end{align*}
	and shock speed $\sfs = \sfs_3(\epsilon)$. Write $M_{\mrm{NS}} = M_{(\bar\rho_{\mrm{NS}},\bar{u}_{\mrm{NS}},\bar T_{\mrm{NS}})}$ for the lifted Navier-Stokes shock profile.
	Then for any $k\ge 2$, there exists a solution $F(x,\xi)$ to the non-cutoff Boltzmann equation
	\begin{align*}
		(\xi_1-\sfs)\p_x F = Q(F,F),\qquad (x,\xi)\in\RR\times\RR^3,
	\end{align*}
	satisfying the estimate
	\begin{align*}
		\norm{F - M_{\mrm{NS}}}_{H^k_\epsilon\HH^1}\le C_k\epsilon^2.
	\end{align*}
	Furthermore, $F$ is unique up to translation among solutions $\bar F$ satisfying the symmetries $R_2 \bar F = R_3\bar F = \bar F$ and the estimate
	\begin{align*}
		\norm{\bar F - M_{\mrm{NS}}}_{H^2_\epsilon\HH^1}\le c_1\epsilon
	\end{align*}
	for $c_1>0$ sufficiently small.
\end{theorem}

In a future work, we expect to remove the symmetry condition $R_2 \bar F = R_3\bar F = \bar F$ in the uniqueness class of the data. We note crucially that we are not able to prove positivity of the resulting shock profiles $F\ge0$. Currently, positivity of kinetic shock profiles is only known to hold for the Boltzmann equation with hard sphere interactions \cite{LiuYu04}. In a future work, we plan to prove stronger regularity and sharper decay estimates the velocity variable $\xi$, as well as stretched exponential decay estimates in $x$, for the shock profile $F$.

We compare this with other works in the literature. The theory of inviscid fluid shock waves in one dimension is very old, and was developed by Riemann \cite{Riemann60}, though under an incorrect assumption of isentropic flow, as noted by Rayleigh, who proved entropy dissipation across shock fronts by vanishing viscosity arguments \cite{Rayleigh10}. One-dimensional shock waves for general viscous gas dynamics were first constructed by Gilbarg \cite{Gilbarg51}, and general solutions to the inviscid Riemann problem were developed by Lax \cite{Lax57}. The theory of one-dimensional kinetic shocks is somewhat more recent. Kinetic shock profiles have been rigorously constructed by Nicolaenko and Thurber for the Boltzmann equation with hard sphere interactions \cite{NicolaenkoThurber75}, a result which was generalized by Caflisch and Nicolaenko for hard cutoff particle interactions \cite{CaflischNicolaenko82}, proving unique solutions up to translation, based on bifurcation arguments and Lyapunov-Schmidt reduction, with uniform bounds
\begin{align*}
	\abs{F(x,\xi) - M_{\mrm{NS}}(x,\xi)}\le C_r\epsilon^2 \br\xi^{-r}\Mref^{1/2}\left( e^{-\mu\abs x^\beta - \tau \epsilon\abs x} \right)
\end{align*}
and
\begin{align*}
	\abs{F(x,\xi) - M_{\mrm{NS}}(\xi)}\le C_{\alpha,r}\br\xi^{-r}\Mref^{1/2+\alpha}
\end{align*}
for any $\alpha\in(0,1/2)$ and $r\ge 0$, for some constants $\mu,\tau>0$ and $\beta\in(0,1)$ corresponding to stretched exponential spatial decay. Positivity of solutions for hard sphere interactions was proved by the stability theory of Liu and Yu \cite{LiuYu04} using an orthogonal micro-macro decomposition and novel energy estimates developed in \cite{LiuYangYu04}. Adapting these methods, M\'etivier and Zumbrun constructed positive shock profile solutions to the hard sphere problem \cite{MetivierZumbrun09} with sharper estimates, reducing \eqref{traveling_wave_Boltzmann} to a problem with bounded operators by the normalization $\tilde A = \br\xi^{-1}\xi_1$ and $\tilde Q = \br\xi^{-1/2}Q(\br\xi^{-1/2}\cdot,\br\xi^{-1/2}\cdot)$. This result was generalized by Albritton, Bedrossian, and Novack for the considerably more singular Landau equation with Coulomb potentials in the plasma setting \cite{AlbrittonBedrossianNovack24}, proving bounds
\begin{align*}
	\norm{e^{\delta\br{\epsilon x}^{1/2}}\Mref^{-q}\epsilon^{-\alpha}\p_x^\alpha\p_\xi^\beta\left(F-M_{\mrm{NS}}\right)}_{L^2_{x,\xi}}\le C_{N,q}\epsilon^2
\end{align*}
for $\abs\alpha + \abs\beta\le N$ and $q\in[0,1)$, for $\epsilon>0$ sufficiently small with respect to $N,q$, for some constants $C,\de>0$. We note that these solutions are not known to be positive. To control moment loss, they track multiple velocity norms to prove a priori estimates on a linearized problem, by combining hypocoercive collisional estimates with stability results for viscous shocks. To close a Galerkin scheme, they replace the Landau collision operator $Q_L$ with a lifted operator $Q_\kappa$ that gains moments, and which is well-behaved with respect to a Hermite discretization. This introduces a smallness condition on the shock strength $\epsilon<\epsilon_0(\kappa)$ dependent on the lift parameter $\kappa$, which they overcome with careful functional-analytic continuity arguments. Our strategy follows a similar Galerkin scheme, however we obtain simpler a priori estimates, using the fact that the linearized Boltzmann collision operator exhibits a spectral gap for hard potentials, so we only need to track a single velocity norm.

We note parallel works on one-dimensional kinetic shocks for the Bhatnagar-Gross-Krook (BGK) equation, particularly by the group of Cuesta, Hittmeir, and Schmeiser, who proved existence and stability of BGK shocks for scalar conservation laws \cite{CuestaSchmeiser06} and later for isentropic gases \cite{CuestaHittmeirSchmeiser10}, exploiting a micro-macro decomposition in the spirit of \cite{LiuYu04} and stability results for viscous shocks \cite{KawashimaMatsumura85} \cite{Goodman86}. A general framework for existence and stability of BGK shocks was developed in \cite{CuestaHittmeirSchmeiser09}, whose presentation influences this work. The problem of kinetic shocks in three dimensions is much more challenging and will not be addressed in this work, so we refer to the recent 3D stability result \cite{DengXu24}.

\section{Preliminaries}\label{Section2}

The basic strategy of our proof follows the methods of M\'etivier and Zumbrun \cite{MetivierZumbrun09} and Albritton, Bedrossian, and Novack \cite{AlbrittonBedrossianNovack24}, namely, we use the Chapman-Enskog expansion to approximate solutions to the problem \eqref{traveling_wave_Boltzmann} by shock solutions to the 1D compressible Navier-Stokes equations, we linearize around the Chapman-Enskog approximation, and then close a fixed point argument. To show existence of solutions, we use a Galerkin scheme and reduce to solving a linear ODE problem, where we construct solutions in $L^2_x$ by stable/unstable dimensionality arguments.

One of the key tools we will use is the micro-macro decomposition, first developed by Hilbert in the theory of hydrodynamic limits \cite{Hilbert12}, and applied to kinetic shock profiles by \cite{LiuYu04}. By classical properties of linearized collisions (cf. Theorem 7.2.1 of \cite{CercignaniIllnerPulvirenti94}), we know that the linearized operator $\Lref$ is negative semi-definite
\begin{align*}
	\br{\Lref f,f}_{\HH^0}\le 0
\end{align*}
in the space $\HH^0$, and has a three dimensional kernel
\begin{align*}
	\UU:=\ker\left(\Lref|_{\HH^0}\right)
\end{align*}
with orthonormal basis
\begin{align}
	\label{UU_orthonormal_basis}
	\set{\psi_0,\psi_1,\psi_2}=
	\set{ \Mref,\ \xi_1 \Mref,\ \frac{\abs\xi^2-3}{\sqrt 6}\Mref }.
\end{align}
We note that without the symmetry assumption \eqref{definition_of_HH} in the definition of $\HH^0$, we would have a five-dimensional kernel for the linearized operator $\Lref$. We also define the orthogonal projection operator in $\HH^0$ onto $\UU$ as
\begin{align}
	\label{definition_of_PP}
	\PP F &= \Mref \begin{pmatrix} 1 \\ \xi_1 \\ \frac{\abs\xi^2-3}{\sqrt 6}
	\end{pmatrix}\cdot \int_{\RR^3}\! F(\xi) \begin{pmatrix} 1 \\ \xi_1 \\ \frac{\abs\xi^2-3}{\sqrt 6}
	\end{pmatrix}\dd\xi
	=
	\Mref \begin{pmatrix} 1 \\ \xi_1 \\ \frac{\abs\xi^2-3}{\sqrt 6}
	\end{pmatrix}\cdot
	\begin{pmatrix}
		\rho \\ m \\ \frac{2E-3\rho}{\sqrt6}
	\end{pmatrix}.
\end{align}
We see that $\PP$ is well-defined for any $F\in L^1(1+\abs\xi^2)$, and that by the linear transformation
\begin{align}
	\left( I|_{\UU}\right)^{-1}:\RR_{\rho}\times\RR_m\times\RR_E\to\UU,
\end{align}
we can parametrize the macroscopic space
$\UU$ by the density, momentum and energy coordinates $W=(\rho,m,E)$, which we call the \emph{conservative variables}. Similarly, the smooth, nonlinear transformation
\begin{equation}
	\label{hydro_macro_change_of_variables}
	\left(I|_{\UU}\right)^{-1}\circ\mathsf{f}_0
	:
	\RR_{>0}\times\RR\times\RR_{>0}\to\UU
\end{equation}
allows us to parametrize the macroscopic space $\UU$ near the reference state $$\uref := \Mref\in\UU$$ by the hydrodynamic variables $V=(\rho,\bfu,T)$, which will simplify the analysis for the rest of the paper.

For the rest of this article, we use 
\textbf{\emph{three}} different coordinate systems for the macroscopic state. We will use the symbol $V = (\rho,\bfu,T)$ for the \emph{hydrodynamic variables}, we will write $W = (\rho,m,E)$ for the \emph{conservative variables}, and we use the fraktur symbol $\mathfrak{u}$ to denote elements of the tangent space $\UU = T_{\Mref}\ca M_{\epsilon_0}\subset\HH^0$ of the Maxwellian manifold. Elements of $\UU$ correspond to small fluctuations in the statistical state described by the Boltzmann distribution for large Knudsen numbers, which force data to be approximately locally Maxwellian. 

Following the micro-macro decomposition introduced for the Boltzmann equation by Liu and Yu \cite{LiuYu04}, we define the orthogonal complement
\begin{align*}
	\VV:= \UU^\perp = \ker\left( \PP|_{\HH^0}\right)\subset \HH^0
\end{align*} 
which we consider as the space of \emph{microscopic quantities}, 
giving the orthogonal decomposition
\begin{align*}
	\HH^0 = \UU\oplus\VV.
\end{align*}
Any function $f$ such that $\PP f=0$ will be called \emph{purely microscopic}. By \eqref{Q_is_microscopic}, the collision term $Q(F,F)$ is purely microscopic, as are the linearized operators $f\mapsto Q(f,g)+Q(g,f)$ for any $g$. The terminology is explained as follows: any perturbation $f\in\VV$ cannot be directly observed from the hydrodynamic quantities of density, momentum, or energy, and can therefore be considered as a microscopic kinetic refinement to the macroscopic fluid model.

We note that $\UU\subset\HH^1$ by the Sobolev norm bounds in \eqref{NSG_Sobolev_bounds}, and therefore we can similarly decompose
\begin{align*}
	\HH^1 = \UU + \VV^1,\qquad\UU\cap\VV^1 = \set0
\end{align*}
where we define $\VV^1 = \HH^1\cap\VV$. 
A similar micro-macro decomposition exists for dual spaces. Using the inclusion of spaces $\mathscr{S}(\RR^3)\subset N^{s,\gamma}\subset L^2(\RR^3)$ where $\mathscr{S}(\RR^3)$ denotes the space of Schwartz functions, we have the inclusion of dual spaces $L^2(\RR^3)\subset (N^{s,\gamma})'\subset\mathscr{S}'(\RR^3)$, and therefore we may consider $(N^{s,\gamma})'$ as a Hilbert space contained in the set of tempered distributions. Writing $\mathscr{D}'(\RR^3) = C_c^\infty(\RR^3)'$ for the space of distributions (cf. \cite{Grafakos08}),
we define the space
\begin{align}
	\label{definition_of_HH-1}
	\HH^{-1} = \set{g\in\mathscr{D}'(\RR^3)\ \middle|\ \Mref^{-1/2}g\in \left( N^{s,\gamma}\right)',\ R_2 g = R_3 g = g}
\end{align}
with associated norm
\begin{align*}
	\norm{g}_{\HH^{-1}} = \norm{\Mref^{-1/2}g}_{(N^{s,\gamma})'},
\end{align*}
where we note that the reflection operators $R_2$ and $R_3$ on $\mathscr{D}(\RR^3)$ can be continuously extended to $\mathscr{D}'(\RR^3)$ and we have the inclusion of spaces $\mathscr{S}'(\RR^3)\subset\mathscr{D}'(\RR^3)$, so the conditions in \eqref{definition_of_HH-1} are well-defined. Furthermore, since $g\in\HH^{-1}$ implies $\Mref^{-1/2}g\in\mathscr{S}'(\RR^3)$, the quantity
\begin{align}
	\label{projection_extends_to_HH-1}
	\int_{\RR^3}\!\begin{pmatrix} 1 \\ \xi_1 \\ \abs\xi^2\end{pmatrix} g(\xi)\dd\xi
	=
	\br{\Mref^{1/2}\begin{pmatrix} 1 \\ \xi_1 \\ \abs\xi^2\end{pmatrix},\Mref^{-1/2}g}_{\mathscr{S}(\RR^3),\,\mathscr{S}'(\RR^3)}
\end{align}
is a well-defined continuous operator in $g\in\HH^{-1}$, and therefore the hydrodynamic projection operator $\PP$ may be extended continuously to
\begin{align*}
	\PP:\HH^{-1}\to\UU,
\end{align*}
with kernel
\begin{align*}
	\VV^{-1} = \ker\left( \PP|_{\HH^{-1}}\right)
\end{align*}
corresponding to the space of purely microscopic functions in $\HH^{-1}$.

\subsection{Maxwellian equilibrium manifold}

It will be useful to characterize the set of Maxwellian equilibria more precisely. If we define the Maxwellian manifold
\begin{align}
	\label{Maxwellian_manifold}
	\mathcal{M}_{\epsilon_0}=\set{M_{(\rho,\bfu,T)}\,\Big|\, \abs{(\rho,\bfu,T)-(1,0,1)}\le\epsilon_0}\subset\HH^0
\end{align}
for $\epsilon_0>0$ sufficiently small, we see by direct computation that the macroscopic space $\UU$ is simply the tangent space of $\ca M_{\epsilon_0}$ at $\Mref$. Using the identity
\begin{align*}
	I[M_{(\rho,\bfu,T)}] = \begin{pmatrix} 
		\rho 
		\\
		\rho\bfu 
		\\
		\frac12 \rho\bfu^2 + \frac32\rho T
	\end{pmatrix}
\end{align*}
by \eqref{hydrodynamic_quantities_definition_0}  \eqref{Maxwellian_hydrodynamic_quantities} and the definition of the macroscopic projection \eqref{definition_of_PP}, we see that for $\epsilon_0$ sufficiently small, we can parametrize the Maxwellians $M_{\mathfrak{u}}\in\ca M_{\epsilon_0}$ by macroscopic states $\mathfrak{u}\in\UU$ such that $\abs{\mathfrak{u} - \uref}\lesssim\epsilon_0$, such that
\begin{align*}
	\PP M_{\mathfrak{u}} = \mathfrak{u}.
\end{align*}
This micro-macro formalism, in which fluid states are parametrized directly by the macroscopic Hilbert space $\UU$ was used in \cite{MetivierZumbrun09} \cite{MetivierZumbrun09_semilinear}, see also \cite{CaoCarrapatoso23} for an analogous construction in incompressible hydrodynamic limits.

We now investigate differentiability properties of the manifold $\ca M_{\epsilon_0}$. For all integer $\ell\ge0$, we define the weighted Sobolev norm
\begin{align*}
	\norm{f}_{H^\ell_\ell(\RR^3)}^2 = \sum_{\abs\alpha+\abs\beta\le \ell} \norm{\xi^\alpha\p_\xi^\beta f}_{L^2(\RR^3)}^2,
\end{align*}
and we define the conjugated Hilbert space $\mathsf{H}^\ell_\ell(\RR^3)$ as the set of all $f\in \HH^0$ such that $\Mref^{-1/2}f\in H^\ell_\ell(\RR^3)$ in the sense of weak derivatives for all multi-indices $(\alpha,\beta)\in\mathbb{Z}^3_{\ge0} \times \mathbb{Z}^3_{\ge0}$ satisfying $\abs\alpha+\abs\beta\le \ell$, with Hilbert space norm
\begin{align}
	\label{definition_of_caHl}
	\norm{f}^2_{\mathsf{H}^\ell_\ell} = 
	\norm{\Mref^{-1/2}f}_{H^\ell_\ell}^2.
\end{align}
We then have the following lemma.

\begin{lemma}
	\label{Maxwellian_manifold_smoothness}
	Let $\epsilon_0>0$ be sufficiently small. Then the Maxwellian manifold $\ca M_{\epsilon_0}$ is a smooth submanifold of $\mathsf{H}^\ell_\ell$ for all $\ell\ge 0$, and in particular, $M_{\epsilon_0}$ is a smooth submanifold of $\HH^1$.
	Furthermore, for any $\abs{\bar{\mathfrak{u}}-\uref}\le\epsilon_0$, we have the estimate
	\begin{align*}
		\norm{(I-\PP)\dd M_{\bar{\mathfrak{u}}}}_{\mathsf{H}^\ell_\ell}\le C_\ell\abs{\bar{\mathfrak{u}}-\uref}
	\end{align*}
	for all $\ell\ge0$.
\end{lemma}
\begin{proof}
	The proof of the first statement is a straightforward computation, using the definitions of the spaces $\HH^0$ and $\mathsf{H}^\ell_\ell$ and bounds on the space $\HH^1$. Using the hydrodynamic variables $V=(\rho,\bfu,T)$ for $\UU$, for any $\alpha\in\mathbb{Z}_{\ge0}^3$ we can compute the derivative
	\begin{align}
		\label{Maxwellian_derivatives_computation}
		\p_V^\alpha M_V(\xi)
		&=
		\p_\rho^{\alpha_1}\p_{\bfu}^{\alpha_2}\p_T^{\alpha_3}
		\left(
		\frac{\rho}{(2\pi T)^{3/2}} e^{-\abs{\xi - \bfu e_1}^2/2T}
		\right)
		\\
		&=
		p_\alpha(\rho,\bfu,T^{-1/2},\xi) e^{-\abs{\xi - \bfu e_1}^2/2T}
	\end{align}
	where $p_\alpha(\rho,\bfu,T^{-1/2},\xi)$ is a polynomial in all its variables. To show that this lies in $\mathsf{H}^\ell_\ell$, it suffices to show that 
	\begin{align*}
		\Mref^{-1/2}\p_V^\alpha M_V(\xi) = c p_\alpha(\rho,\bfu,T^{-1/2},\xi) e^{-\abs{\xi - \bfu e_1}^2/2T + \abs\xi^2/4}
	\end{align*}
	is Schwartz class in $\xi$, but this clearly holds for any temperature $T<2$, and therefore for $\epsilon_0>0$ sufficiently small. To prove the last claim, we note by \eqref{NSG_Sobolev_bounds} and the hard cutoff assumption \eqref{Hypothesis3} that we have the continuous inclusion of spaces $\mathsf{H}^2_2\subset\HH^1$.
	
	To prove the microscopic estimate, we note that since $\dd M_{\uref}\in\UU$ by definition of $\UU$, so that $(I-\PP)\dd M_{\uref}=0$. Boundedness of the operator $\PP:\mathsf{H}^\ell_\ell\to\UU$ and smoothness of $\ca M_{\epsilon_0}$ then proves the claim for sufficiently small $\epsilon_0>0$.
\end{proof}

Since we are constructing our shock profile around the travelling wave Maxwellian $M_{\mrm{NS}}$ and not the steady state $\Mref$, it will be useful to define the uncentred linearized operator
\begin{align*}
	L_{\bar{\mathfrak{u}}}f = Q(M_{\bar{\mathfrak{u}}},f) + Q(f,M_{\bar{\mathfrak{u}}})
\end{align*}
for $\bar{\mathfrak{u}}$. Differentiating the identity
\begin{align*}
	0=\dd_{\bar{\mathfrak{u}}}\big[ Q(M_{\bar{\mathfrak{u}}},M_{\bar{\mathfrak{u}}})\big]
	=
	Q(M_{\bar{\mathfrak{u}}},\dd M_{\bar{\mathfrak{u}}}) + Q(\dd M_{\bar{\mathfrak{u}}},M_{\bar{\mathfrak{u}}})
	=
	L_{\bar{\mathfrak{u}}}\dd M_{\bar{\mathfrak{u}}},
\end{align*}
we can define the uncentred projection operator $$\PP_{\bar{\mathfrak{u}}} = \dd M_{\bar{\mathfrak{u}}}\PP.$$
To see that this operator is a projection, we differentiate $\PP M_{\bar{\mathfrak{u}}} = \bar{\mathfrak{u}}$ to get the identity 
\begin{align}
	\label{macro_of_differentiated_Maxwellian}
	\PP\dd M_{\bar{\mathfrak{u}}} = I,
\end{align}
giving that $\PP_{\bar{\mathfrak{u}}}^2 =\PP_{\bar{\mathfrak{u}}}$.

\subsection{Strategy of the proof}

We can now outline the main methods of the proof of Theorem \ref{main_theorem}. We start with a travelling wave solution $\bar V_{\mrm{NS}}(x-\sfs t)$ solution to the 
1D Navier-Stokes equations \eqref{Compressible_Navier-Stokes_Equations}, with boundary conditions $V_\pm$ defined in Theorem \ref{main_theorem}, and we let $\uns = (I|_{\UU})^{-1}\circ\mathsf{f}_0(\bar V_{\mrm{NS}})$ be the $\UU$-valued change of coordinates using the formula in \eqref{hydro_macro_change_of_variables}. In Section \ref{Chapman-Enskog_Theory_section}, we define a perturbative problem
\begin{align}
	F = M_{\uns} + f^\perp[\uns] + f
\end{align}
from the Chapman Enskog expansion, where the second order term $f^\perp[\uns]$ is a microscopic correction of order $O(\epsilon^2)$ defined in \eqref{Chapman-Enskog_correction},
and we then reformulate \eqref{traveling_wave_Boltzmann} as a linearized problem
\begin{align}
	\label{linearized_operator_definition}
	\ca L_\epsilon f := (\xi_1-\sfs)\p_x - L_{\uns}f = \ca E[f]
\end{align}
for a microscopic error term $\ca E[f]\in H^k_\epsilon\VV^{-1}$ defined in \eqref{error_term_computation}. We will find that $\ca L_\epsilon$ has a one-dimensional kernel due to translation symmetry of the problem, which we fix using the following proposition.

\begin{proposition}
	\label{right_inverse_proposition}
	For any $\epsilon\in(0,\epsilon_0)$ for $\epsilon_0>0$ sufficiently small, the operator $\ca L_\epsilon$ admits a right inverse
	\begin{align*}
		\ca L_\epsilon^\dagger: H^k_\epsilon\VV^{-1}\to H^k_\epsilon\HH^1
	\end{align*}
	and a unit vector $\ell_\epsilon\in\UU$ depending continuously on $\epsilon$, such that for any $z\in H^k_\epsilon\VV^{-1}$, the function $f = \ca L_\epsilon^\dagger z$ solves the problem
	\begin{align*}
		\ca L_\epsilon f = z,\qquad \ell_\epsilon\cdot\PP f(0) = 0.
	\end{align*}
	This operator satisfies the estimate
	\begin{align}
		\label{main_right_inverse_bound}
		\norm{\ca L_\epsilon^\dagger z}_{H^k_\epsilon\HH^1}
		\le \frac{C_k}{\epsilon} \norm{z}_{H^k_\epsilon\VV^{-1}},
	\end{align}
	and furthermore, we have the bound $\abs{\ell_\epsilon\cdot\p_x\uns(0)}\ge c_0\epsilon^2$, for constants $C_k,c_0>0$ independent of $\epsilon$.
\end{proposition}

The last bound is a nondegeneracy condition, which will allow us to prove uniqueness of solutions up to translation. The proof of this proposition occupies most of the article, and combined with bounds on $\ca E[f]$ and $f^\perp[\uns]$, this will prove the main theorem by a fixed point argument. To prove this proposition, we follow the strategy of \cite{MetivierZumbrun09}. In Section \ref{Micro_section}, we prove a priori estimates on the linearized problem, exploiting the compensator theory developed by Kawashima \cite{Kawashima90}, and close an energy estimate up to a small macroscopic error $\epsilon\norm{\PP f}_{L^2_\epsilon}$. To bound this error, we use a stability result for viscous shocks, which allows us to close estimates on the linearized operator in Section \ref{Macro_section}. The explicit construction of the right inverse then proceeds by a Galerkin scheme. For non-cutoff Boltzmann, however, additional difficulties arise due to the loss of velocity moments in the Galerkin scheme, as identified by Albritton, Bedrossian, and Novack
for the Coulombic Landau equation \cite{AlbrittonBedrossianNovack24}, and therefore we will require a lifting procedure on the collision operator $Q$, and functional-analaytic continuity methods to show that this procedure can be justified. We adapt the lifting method of \cite{AlbrittonBedrossianNovack24} to the Boltzmann collision operator.

\subsection{Estimates on the collision operator}

We now review basic boundedness and coercivity estimates for the non-cutoff Boltzmann collision operator, following the work of \cite{GressmanStrain11}. We use the following results.

\begin{lemma}
	\label{basic_GS11_estimate}
	For any $f,g,h\in\HH^1$ as defined in \eqref{HH1_norm}, we have the estimates
	\begin{align}
		\label{Q_boundedness}
		\br{Q(g,h),f}_{\HH^0}\le C\norm{f}_{\HH^1}\norm{g}_{\HH^1}\norm{h}_{\HH^1}
	\end{align}
	and
	\begin{align}
		\label{Lref_coercivity}
		-\br{\Lref f,f}_{\HH^0}\ge \de_0\norm{(I-\PP)f}_{\HH^1}^2
	\end{align}
	for some constants $\de_0,C>0$.
\end{lemma}
\begin{proof}
	Using Theorem 2.1 of \cite{GressmanStrain11} and the hard potential assumption \eqref{Hypothesis3}, we have the trilinear estimate
	\begin{align*}
		\abs{\br{Q(g,h),f}_{\HH^0}}
		&=
		\br{\Mref^{-1/2}Q(g,h),\Mref^{-1/2}f}_{L^2(\RR^3)}
		\\
		&\le
		C\norm{\Mref^{-1/2}g}_{L^2(\RR^3)}
		\norm{\Mref^{-1/2}h}_{N^{s,\gamma}}
		\norm{\Mref^{-1/2}f}_{N^{s,\gamma}}
		\\
		&=
		C\norm{g}_{\HH^0}\norm{h}_{\HH^1}\norm{f}_{\HH^1}
		\\
		&\le
		C\norm{g}_{\HH^1}\norm{h}_{\HH^1}\norm{f}_{\HH^1},
	\end{align*}
	where in the last line we have used the norm bound \eqref{NSG_coercivity}. 
	Similarly, using Theorem 8.1 of \cite{GressmanStrain11}, we get the coercivity estimate
	\begin{align*}
		-\br{\Lref f,f}_{\HH^0}
		&=
		-\br{\Mref^{-1/2}\Lref f,\Mref^{-1/2} f}_{L^2(\RR^3)}
		\\
		&
		\ge
		\de_0 \norm{\Mref^{-1/2}(I-\PP)f}_{N^{s,\gamma}}^2
		=
		\de_0\norm{(I-\PP)f}_{\HH^1}^2,
	\end{align*}
	by the definition of the norms $\HH^0$ and $\HH^1$ in \eqref{definition_of_HH} and \eqref{HH1_norm}, proving \eqref{Lref_coercivity}. 
\end{proof}

It will turn out to be useful to restate the bound \eqref{Q_boundedness} in terms of dual spaces, giving the following corollary.

\begin{corollary}
	\label{basic_GS11_estimate_corollary}
	Let $g,h\in\HH^1$, then we have the estimate
	\begin{align*}
		\norm{Q(g,h)}_{\HH^{-1}}\le C\norm{g}_{\HH^1}\norm{h}_{\HH^1}
	\end{align*}
	for some constant $C>0$.
\end{corollary}
\begin{proof}
	The estimate is immediate from the definition \eqref{definition_of_HH-1} of the space $\HH^{-1}$.
\end{proof}

For $\bar{\mathfrak{u}}$ close to $\uref$, we have the following approximate coercivity estimate, analogous to Lemma \ref{basic_GS11_estimate}.

\begin{lemma}
	\label{uncentred_coercivity_estimate}
	Let $f\in\HH^1$, and suppose $\abs{\bar{\mathfrak{u}}-\uref}\lesssim\epsilon\in(0,\epsilon_0)$ for $\epsilon_0>0$ sufficiently small. Then we have the estimate
	\begin{align*}
		-\br{L_{\bar{\mathfrak{u}}}f,f}_{\HH^0}\ge \de\norm{(I-\PP)f}_{\HH^1}^2 - C\epsilon\norm{\PP f}_{\HH^1}\norm{(I-\PP)f}_{\HH^1}
	\end{align*}
\end{lemma}
\begin{proof}
	We first re-express
	\begin{align*}
		-\br{L_{\bar{\mathfrak{u}}}f,f}_{\HH^0}
		&=
		-\br{\Lref f,f}_{\HH^0}
		+ \br{Q(M_{\bar{\mathfrak{u}}} - \Mref,f) + Q(f,M_{\bar{\mathfrak{u}}} - \Mref),f}_{\HH^0}
		\\
		&=
		-\br{\Lref f,f}_{\HH^0}
		+ \br{Q(M_{\bar{\mathfrak{u}}} - \Mref,f) + Q(f,M_{\bar{\mathfrak{u}}} - \Mref),(I-\PP)f}_{\HH^0}
	\end{align*}
	where in the last identity we have used that the correction term $Q(M_{\bar{\mathfrak{u}}} - \Mref,f) + Q(f,M_{\bar{\mathfrak{u}}} - \Mref)$ is purely microscopic by Corollary \ref{basic_GS11_estimate_corollary}, and by Lemma  \ref{basic_GS11_estimate}, this can then be bounded from below by
	\begin{align*}
		&
		\de_0\norm{(I-\PP)f}_{\HH^1}^2 - C\norm{M_{\bar{\mathfrak{u}}} - \Mref}_{\HH^1}\norm{f}_{\HH^1}\norm{(I-\PP)f}_{\HH^1}
		\\
		&\qquad
		\ge
		-\de_0\norm{(I-\PP)f}_{\HH^1}^2 - C'\epsilon\norm{f}_{\HH^1}\norm{(I-\PP)f}_{\HH^1}
		\\
		&\qquad
		\ge
		-\frac12\de_0 \norm{(I-\PP)f}_{\HH^1}^2 -  C'\epsilon\norm{\PP f}_{\HH^1}\norm{(I-\PP)f}_{\HH^1}
	\end{align*}
	using the bound $\norm{M_{\bar{\mathfrak{u}}} - \Mref}_{\HH^1}\le C\epsilon$ from Lemma \ref{Maxwellian_manifold_smoothness}, and where the last bound holds for $\epsilon$ sufficiently small.
\end{proof}

We now wish to introduce a "lifted" collision operator $Q_\kappa$ for $\kappa>0$, such that $Q_\kappa$ exhibits improved moment estimates, which will be useful to show existence from a Galerkin scheme. We first write $B_{s',\gamma'}(\abs{\xi-\xi_*},\cos\theta)$ for a collision kernel satisfying hypotheses \eqref{Hypothesis1} and \eqref{Hypothesis2} but not necessarily \eqref{Hypothesis3}, with coefficients $s'\in(0,1)$ and $\ga'\in\RR$. Setting
\begin{align*}
	Q_{s',\gamma'}(G,F) =\int_{\RR^3_{\xi_*}}\!\int_{\SSS^2_\sigma}\!
	B_{s',\gamma'}\left( \abs{\xi-\xi_*},\sigma\cdot\frac{\xi-\xi_*}{\abs{\xi-\xi_*}}\right)
	\left[ G(\xi_*')F(\xi') - G(\xi_*)F(\xi)\right]\dd\sigma\dd\xi_*,
\end{align*}
we can then define the lifted operator $Q_\kappa$ as
\begin{align}
	\label{regularized_collision_operator}
	\begin{split}
		Q_\kappa(G,F) 
		&= Q(G,F) + \kappa Q_{s,2-2s}(G,F)
		\\
		&=
		\int_{\RR^3_{\xi_*}}\!\!\int_{\SSS^2_\sigma}\!
		\left(\abs{\xi-\xi_*}^\ga + \abs{\xi-\xi_*}^{2-2s}\right) b_s(\cos\theta)
		\\
		&\qquad\qquad\qquad\qquad\qquad\cdot
		\left[ G(\xi_*')F(\xi') - G(\xi_*)F(\xi)\right]\dd\sigma\dd\xi_*
	\end{split}
\end{align}
for any $\kappa\ge0$. We analogously define the linearized collision operator $\Lref_\kappa = \Lref + \kappa\Lref_{s,2-2s}$, where we set
\begin{align*}
	\Lref_{s',\ga'}f = Q_{s',\gamma'}(\Mref,f) + Q_{s',\gamma'}(f,\Mref).
\end{align*}
We then define the associated norm
\begin{align}
	\label{definition_of_H1kappa}
	\norm{f}_{\HH^1_\kappa} = \max\left\{ \norm{\Mref^{-1/2} f}_{N^{s,\gamma}},\kappa^{1/2}\norm{\Mref^{-1/2}f}_{{N}^{s,2-2s}} \right\},
\end{align}
and define $\HH^1_\kappa$ as the space of $f\in\HH^0$ such that this norm is finite.
We see from by the definition \eqref{GS11norm} that the norms $\norm{\cdot}_{N^{s,\gamma}}$ are increasing in $\gamma$, and therefore for all $\kappa>0$, we have the equivalence of norms
\begin{align*}
	\frac{1}{C_\kappa} \norm{\Mref^{-1/2}f}_{N^{s,2-2s}}\le  \norm{f}_{\HH^1_\kappa}
	\le C_\kappa \norm{\Mref^{-1/2}f}_{N^{s,2-2s}}
\end{align*}
for some constants $C_\kappa>0$ so that the space $\HH^1_\kappa$ is Banach isomorphic to a Hilbert space. As with the definition of the space $\HH^{-1}$ in \eqref{definition_of_HH-1}, we can define the associated dual space
\begin{align*}
	\HH^{-1}_\kappa  = \set{g\in\mathscr{D}'(\RR^3)\,\middle|\,\Mref^{-1/2}g\in (N^{s,2-2s})',\ R_2 g = R_3 g= g},
\end{align*}
using the continuous inclusion of spaces $(N^{s,\gamma})'\subset(N^{s,2-2s})'\subset\mathscr{S}'(\RR^3)$, which guarantees that the norm 
\begin{align}
	\label{definition_of_HH-1kappa}
	\begin{split}
		&
		\norm{ g}_{\HH^{-1}_\kappa} 
		\\
		&\quad= \inf\set{\norm{\Mref^{-1/2} g_0}_{(N^{s,\gamma})'} + \kappa^{-1/2} \norm{\Mref^{-1/2}g_1}_{(N^{s,2-2s})'}\,\middle|\, g= g_0+g_1},
	\end{split}
\end{align}
is well-defined. By standard identities in interpolation theory, for all $\kappa>0$, we have the bounds
\begin{align*}
	\abs{\br{g,f}_{\HH^0} } \le \norm{g}_{\HH^{-1}_\kappa}\norm{f}_{\HH^1_\kappa}
\end{align*}
for all $g\in\HH^{-1}_\kappa$ and $f\in\HH^1_\kappa$, which are sharp in the sense that
\begin{align*}
	\sup_{\norm{f}_{\HH^1_\kappa\le 1}}\abs{\br{g,f}_{\HH^0}} = \norm{g}_{\HH^{-1}_\kappa}
\end{align*}
holds for all $g\in\HH^{-1}_\kappa$ 
(cf. Exercise 3.3 of \cite{BennettSharpley_book}). Since $g\in\HH^{-1}_\kappa$ implies $\Mref^{-1/2}g\in\mathscr{S}'(\RR^3)$, the same previous identity \eqref{projection_extends_to_HH-1} shows that $\PP$ can be extended continuously to $\HH^{-1}_\kappa$, so we define the kernel
\begin{align*}
	\VV^{-1}_\kappa = \ker\left( \PP |_{\HH^{-1}_\kappa}\right).
\end{align*}
We can analogously define the microscopic functions of positive order $\VV^1_\kappa \subset\HH^1_\kappa$, and by the inclusion of spaces $\UU\subset\HH^1_\kappa\cap\HH^{-1}_\kappa$, we get the micro-macro decompositions
\begin{align*}
	\HH_\kappa^1 = \UU + \VV_\kappa^1,\qquad \UU\cap\VV_\kappa^1 = \set0
\end{align*}
and
\begin{align*}
	\HH_\kappa^{-1} = \UU + \VV_\kappa^{-1},\qquad \UU\cap\VV_\kappa^{-1} = \set0
\end{align*}
for all $\kappa\ge0$.

We compare the $\HH^1_\kappa$ norm to our previously defined norms by the following bound, which will be useful for controlling the transport term in the Galerkin scheme.

\begin{lemma}
	\label{HH1kappa_Sobolev_bound}
	Let $\kappa\in[0,\kappa_0]$ for a fixed $\kappa_0>0$. Then for all $f\in\HH^1_\kappa$, we have the bound
	\begin{align*}
		\frac{\kappa^{1/2}}{C}\norm{\br\xi f}_{\HH^0}\le \norm{f}_{\HH^1_\kappa}
		\le
		C\norm{f}_{\mathsf{H}^2_2}
	\end{align*}
	for $C$ independent of $\kappa$ (but dependent on $\kappa_0$). Furthermore, we have the uniform coercivity estimate $\norm{f}_{\HH^0}\le C\norm{f}_{\HH^1_\kappa}$ for all $\kappa\ge 0$.
\end{lemma}
\begin{proof}
	From \eqref{NSG_Sobolev_bounds}, we have the estimate
	\begin{align*}
		\norm{f}_{\HH^1_\kappa}\ge\kappa^{1/2}\norm{\Mref^{-1/2}f}_{N^{s,2-2s}}
		\ge
		\frac{\kappa^{1/2}}{C}\norm{\Mref^{-1/2}f}_{L^2_1}
		=
		\frac{\kappa^{1/2}}{C}\norm{\br\xi f}_{\HH^0},
	\end{align*}
	proving the lower bound. As in the proof of Lemma \ref{Maxwellian_manifold_smoothness}, the upper bound in \eqref{NSG_Sobolev_bounds} gives us that
	\begin{align*}
		\max\left\{\norm{\Mref^{-1/2}f}_{N^{s,\ga}},
		\norm{\Mref^{-1/2}f}_{N^{s,2-2s}}\right\}
		\le C\norm{\Mref^{-1/2}f}_{H^s_1}
		\le C'\norm{f}_{\mathsf{H}^2_2},
	\end{align*}
	which proves the desired upper bound. The uniform coercivity estimate follows immediately by definition of $\HH^1_\kappa$.
\end{proof}

Following \cite{GressmanStrain11}, for any $(s',\ga')\in(0,1)\times\RR$ and for all $\ell\ge0$ we define the weighted derivative space $N^{s',\gamma'}_{\ell,\ell}$ by the norm
\begin{align*}
	\norm{f}_{N^{s',\gamma'}_{\ell,\ell}}^2
	=
	\sum_{\abs\alpha+\abs\beta\le\ell}\norm{\xi^\alpha\p_\xi^\beta f}_{N^{s',\gamma'}}^2.
\end{align*}
In fact, the norm in \cite{GressmanStrain11} is written in a slightly different form, however this is easily seen to be equivalent by a fractional integration by parts argument. We can define the space $\HH^{1,\ell}_\kappa$ as the set of all $f\in \HH^0$ such that $\Mref^{-1/2}f\in N^{2,2-2s}_{\ell,\ell}$, with norm
\begin{align}
	\label{definition_of_HH1lkappa}
	\norm{f}_{\HH^{1,\ell}_\kappa} = 
	\max\left\{\norm{\Mref^{-1/2}f}_{N^{s,\ga}_{\ell,\ell}},
	\kappa^{1/2} \norm{\Mref^{-1/2} f}_{N^{s,2-2s}_{\ell,\ell}}
	\right\}
\end{align}
and define the space $\HH^{1,\ell} = \HH^{1,\ell}_0$ with the corresponding norm. We then have the following trilinear and coercivity estimates for $Q_\kappa$.

\begin{lemma}
	\label{Qkappa_trilinear_estimate}
	Let $\kappa\in[0,\kappa_0]$ for some fixed $\kappa_0$.
	Then for all $\ell\ge 0$ and $f,g,h\in\HH^{1,\ell}_\kappa$, we have the trilinear estimate
	\begin{align*}
		\abs{\br{Q_\kappa(g,h),f}_{\mathsf{H}^\ell_\ell}}\le C_\ell\norm{g}_{\HH^{1,\ell}_\kappa}\norm{h}_{\HH^{1,\ell}_\kappa}
		\norm{f}_{\HH^{1,\ell}_\kappa}
	\end{align*}
	for some $C_\ell>0$. independent of $\epsilon$ and $\kappa$.
	For $f\in\HH^1_\kappa$, we have the coercivity estimate.
	\begin{align*}
		-\br{\Lref_\kappa f,f}_{\HH^0}\ge\de\norm{(I-\PP)f}_{\HH^{1}_\kappa}^2,
	\end{align*}
	for some $\de>0$. Meanwhile, for purely microscopic $f\in\HH^{1,\ell}_\kappa$, 
	we can define an equivalent Hilbert norm $\norm\cdot_{\widetilde{\mathsf{H}^\ell_\ell}}$ for $\mathsf{H}^\ell_\ell$ independent of $\kappa$, such that we have the estimate
	\begin{align*}
		-\br{\Lref_\kappa f,f}_{\widetilde{\mathsf{H}^\ell_\ell}}\ge \de_\ell\norm{f}_{\HH^{1,\ell}_\kappa}^2,
	\end{align*}
	that is, the operator $\Lref_\kappa$ is \emph{hypocoercive} on $\mathsf{H}^\ell_\ell$ in the sense of \cite{GualdaniMischlerMouhot17}.
	
	The constants $C_\ell,\de,\de_\ell$ here are all independent of $\kappa$.
\end{lemma}

\begin{proof}

	We first define the nonlinear operator $$\Gamma_{s',\ga'}(G,H) = \Mref^{-1/2}Q_{s',\ga'}(\Mref^{1/2}G,\Mref^{1/2}H)$$
	for $(s',\ga')\in(0,1)\times\RR$. Using Lemma 2.2 of \cite{GressmanStrain11} and \eqref{definition_of_caHl}, we have that
	
	\begin{align}
		\nonumber
		\abs{\br{Q_{s',\ga'}(g,h),f}_{\mathsf{H}^\ell_\ell}}
		&\le
		\sum_{\abs\alpha+\abs\beta\le\ell}\abs{\br{\xi^\alpha\p_x^\beta\Gamma_{s',\ga'}\left(\Mref^{-1/2}g,\Mref^{-1/2} h\right),\xi^\alpha\p_x^\beta(\Mref^{-1/2} f) }_{L^2}}
		\\
		&
		\label{GS11_general_trilinear_bound}
		\le C_{s',\ga',\ell}
		\norm{g}_{\mathsf{H}^\ell_\ell}\norm{\Mref^{-1/2}g}_{N^{s',\ga'}_{\ell,\ell}}
		\norm{\Mref^{-1/2}g}_{N^{s',\ga'}_{\ell,\ell}}
	\end{align}
	and therefore we have that
	\begin{align*}
		\abs{\br{Q_\kappa(g,h),f}_{\mathsf{H}^\ell_\ell}}
		&\le
		C_\ell
		\norm{g}_{\mathsf{H}^\ell_\ell}
		\Big(
		\norm{\Mref^{-1/2}h}_{N^{s,\ga}_{\ell,\ell}}
		\norm{\Mref^{-1/2}f}_{N^{s,\ga}_{\ell,\ell}}
		\\
		&\qquad\qquad\qquad
		+
		\kappa
		\norm{\Mref^{-1/2}h}_{N^{s,2-2s}_{\ell,\ell}}
		\norm{\Mref^{-1/2}f}_{N^{s,2-2s}_{\ell,\ell}}
		\Big)
		\\
		&\le
		2 C_\ell
		\norm{g}_{\HH^{1,\ell}_\kappa}
		\norm{h}_{\HH^{1,\ell}_\kappa}
		\norm{f}_{\HH^{1,\ell}_\kappa},
	\end{align*}
	using the definition of the ${\HH^{1,\ell}_\kappa}$ norm in \eqref{definition_of_HH1lkappa}, proving the first estimate.
	
	For the coercivity estimate, using Theorem 8.1 of \cite{GressmanStrain11}, we have that
	\begin{align*}
		-\br{\Lref_{s,2-2s}f,f}_{\HH^0}\ge\de_0\norm{\Mref^{-1/2}(I-\PP)f}_{N^{s,2-2s}}^2
	\end{align*}
	for some $\de_0>0$
	and therefore we have the estimate
	\begin{align*}
		-\br{\Lref_\kappa f,f}_{\HH^0}
		&=
		-\br{\Lref f,f}_{\HH^0} - \kappa\br{\Lref_{s,2-2s}f,f}_{\HH^0}
		\\
		&
		\ge \de\left( \norm{\Mref^{-1/2}(I-\PP)f}_{N^{s,\ga}}^2
		+
		\kappa
		\norm{\Mref^{-1/2}(I-\PP)f}_{N^{s,2-2s}}^2
		\right)
		\\
		&\ge
		\de
		\norm{(I-\PP)f}_{\HH^1_\kappa}
	\end{align*}
	for some $\de>0$. The hypocoercivity estimate is then proved similarly, using Lemma 2.6 of \cite{GressmanStrain11}.
	
\end{proof}

As previously, for $\bar{\mathfrak{u}}\in\UU$ close to $\uref$, we define the uncentred linearized operator
\begin{align*}
	L_{\bar{\mathfrak{u}},\kappa} f= Q_\kappa(M_{\bar{\mathfrak{u}}},f) + Q_\kappa(f,M_{\bar{\mathfrak{u}}}),
\end{align*}
satisfying the identities
\begin{align}
	\label{twisted_Lukappa_projections}
	(I-\PP) L_{\bar{\mathfrak{u}}} = L_{\bar{\mathfrak{u}}}(I-\PP_{\bar{\mathfrak{u}}}) = L_{\bar{\mathfrak{u}}},
\end{align} 
and we have the following approximate coercivity estimate.

\begin{lemma}
	\label{uncentred_Lkappa_coercivity_estimate}
	Let $\epsilon_0>0$ be sufficiently small, and let $\kappa_0>0$ be fixed. Then for any $\kappa\in[0,\kappa_0]$ and $f\in\HH^1_\kappa$, and any $\bar{\mathfrak{u}}\in\UU$ such that $\epsilon=\abs{\bar{\mathfrak{u}}-\uref}\le\epsilon_0$, we have the bound
	\begin{align*}
		\br{L_{\bar{\mathfrak{u}},\kappa} f,f}_{\HH^0}\ge \de\norm{(I-\PP)f}_{\HH^1_\kappa}^2 - C\epsilon\norm{\PP f}_{\HH^1_\kappa}\norm{(I-\PP)f}_{\HH^1_\kappa}
	\end{align*}
	for constants $\de,C>0$ independent of $\kappa$ and $\epsilon$.
\end{lemma}
\begin{proof}
	The proof is essentially the same as that of Lemma \ref{uncentred_coercivity_estimate}, using the uniform estimates in $\kappa$ from Lemma \ref{Qkappa_trilinear_estimate}, the fact that the Maxwellian manifold $\ca M_{\epsilon_0}$ embeds smoothly in $\mathsf{H}^2_2$, and the uniform estimate on $\HH^1_\kappa$ proved in Lemma \ref{HH1kappa_Sobolev_bound}.
\end{proof}

We can now prove a uniform invertibility estimate, which will be useful for the Chapman-Enskog approximation.

\begin{lemma}
	\label{uniform_invertibility_lemma}
	Let $\epsilon_0>0$ be sufficiently small, and let $\kappa_0>0$ be fixed. Then for $\kappa\in[0,\kappa_0]$ and for any $\abs{\bar{\mathfrak{u}}-\uref}\le \epsilon_0$, the
	operator $L_{\bar{\mathfrak{u}},\kappa}|_{\VV^1_\kappa}$ admits an inverse $L_{\bar{\mathfrak{u}},\kappa}^{-1}:\VV^{-1}_\kappa\to\HH^1_\kappa$, with estimates
	\begin{align*}
		\norm{L_{\bar{\mathfrak{u}},\kappa}^{-1} z}_{\HH^{1}_\kappa}\le C
		\norm{z}_{\VV^{-1}_\kappa}
	\end{align*}
	for a constant $C$ independent of $\epsilon$ and $\kappa$.
\end{lemma}
\begin{proof}
	The existence of a set-theoretic inverse operator for $L_{\bar{\mathfrak{u}},\kappa}|_{\VV^1_\kappa}$ is immediately provided by the coercivity bound in Lemma \ref{uncentred_Lkappa_coercivity_estimate} and the Lax-Milgram theorem, and the uniform bound in $\epsilon$ and $\kappa$ is provided by Lemma \ref{uncentred_Lkappa_coercivity_estimate}.
\end{proof}

In the Chapman-Enskog expansion, it will be useful to also have the following invertibility result in an expanded space, which we prove using the extension of semigroup theory developed by Gualdani, Mischler, and Mouhot \cite{Mouhot06} \cite{GualdaniMischlerMouhot17}.

\begin{lemma}[Extension of the semigroup]
	\label{HTT20_Lemma}
	Let $\abs{\bar{\mathfrak{u}} - \uref}\le\epsilon_0$ and $\kappa\in[0,\kappa_0]$ for $\epsilon_0>0$ sufficiently small, and $\kappa_0>0$ fixed. Denoting by $L^2(\br\xi^k)_{\mrm{micro}}\subset L^2(\br\xi^k)$ the subspace of purely microscopic functions, the operator $L_{\bar{\mathfrak{u}},\kappa}$ is uniformly invertible on $L^2(\br\xi^k)_{\mrm{micro}}$ for all $k\ge 4$, with uniform bound
	\begin{align*}
		\norm{L_{\bar{\mathfrak{u}},\kappa}^{-1}f}_{L^2(\br\xi^k)}\le C_k\norm{f}_{L^2(\br\xi^k)}
	\end{align*}
	for some constant $C_k$ independent of $\bar{\mathfrak{u}}$ and $\kappa$.
\end{lemma}
When $\kappa=0$, invertibility was proved in Theorem 4.1 of \cite{HerauTononTristani20}. Since the arguments for the lifted operator $L_{\bar{\mathfrak{u}},\kappa}$ when $\kappa>0$ are very similar, we defer the proof of this lemma to Appendix \ref{Appendix_C}. 

\subsection{Kawashima compensators}

To compensate for the lack of coercivity, we use the Kawashima compensator method, developed by Kawashima to prove energy decay on the whole space \cite{Kawashima90}, which is related to theories of hypocoercivity, as noted in \cite{Villani09} \cite{DolbeaultMouhotSchmeiser15}. The Kawashima construction involves finding a real skew-symmetric operator $K$ of finite rank such that $[K,\xi_1]$ will be coercive on the macroscopic term. This is useful for energy estimates. This method was crucially used to prove linearized a priori estimates for kinetic shocks for the hard sphere Boltzmann \cite{MetivierZumbrun09} and Landau equations \cite{AlbrittonBedrossianNovack24}, as well as time decay for Vlasov-Poisson Boltzmann equation \cite{DuanStrain11}. Before we can define the compensator, we first introduce the discretized subspaces
\begin{align}
	\label{discretized_subspaces}
	\HH_N=\operatorname*{Span}\limits_{\substack{ \abs\alpha\le N, \\
			\alpha_2 \equiv\alpha_3\equiv 0\ (\operatorname{mod}2)}}\xi^\alpha\Mref
\end{align}
in $\HH^0$, noting the inclusion $\UU\subset\HH_2$, and define 
\begin{align*}
	\Pi_N:\HH^0\to\HH_N
\end{align*}
as the associated orthogonal projection operator for any integer $N\ge 0$. 

\begin{lemma}[Lemma 5.1 of \cite{Kawashima90}]
	\label{Kawashima_compensator_construction}
	There exists a skew-symmetric operator $\bar K:\HH^0\to\HH^0$, which is finite rank in the sense that $\bar K=\Pi_3 \bar K\Pi_3$, and satisfies the coercivity estimate
	\begin{align*}
		\br{[\bar K,\xi_1] f,f}_{\HH^0}\ge c_1\norm{\PP f}_{\HH^0}^2 - C_1\norm{(I-\PP)f}_{\HH^0}^2
	\end{align*}
	for all $f\in\HH_3$, for some constants $c_1,C_1>0$. Furthermore, the operator $\bar K$ satisfies estimates
	\begin{align*}
		\norm{\bar Kf}_{\HH^1_\kappa}\le C\norm{f}_{\HH^1_\kappa},
		\qquad
		\norm{\bar Kf}_{\HH^{-1}_\kappa}\le C\norm{f}_{\HH^{-1}_\kappa}
	\end{align*}
	for a constant $C$ independent of $\kappa$.
\end{lemma}
For completeness, we provide a proof.
\begin{proof}
	We write $A = \Pi_3\xi_1\Pi_3$ and write it as a matrix in macro-micro coordinates
	\begin{align*}
		A = \begin{pmatrix} A_{00} & A_{10} \\ A_{01} & A_{11}\end{pmatrix},
	\end{align*}
	noting the symmetry $A = A^*$. We then look for some 
	\begin{align}
		\label{K_construction_equation}
		\bar K = \begin{pmatrix} \de K_{00} & A_{10} \\ -A_{01} & 0 \end{pmatrix}
	\end{align}
	so that we can compute $\PP[\bar K,A]\PP = \de[K_{00},A_{00}] + 2A_{01}A_{10}$, and we solve for $K_{00}$ such that this expression is positive definite for some small $\de>0$. Noting by construction that $A_{10} = (I-\PP)\xi_1\PP$, we can explicitly compute
	\begin{align*}
		A_{10}\psi_0 = 0,\quad
		A_{10}\psi_1 = \left(\xi_1^2-\frac13\abs\xi^2\right)\Mref,
		\quad
		A_{10}\psi_2 = \frac{\xi_1\abs{\xi}^2 - 5\xi_1}{\sqrt 6}\Mref
	\end{align*}
	where the second and third terms correspond to viscous and heat conduction moments. We can then compute the matrix
	\begin{align*}
		A_{01}A_{10} = \begin{pmatrix} 0 & 0 & 0 \\ 0 & 4/3 & 0 \\ 0 & 0 & 5/3\end{pmatrix}
	\end{align*}
	in the orthonormal basis $\set{\psi_0,\psi_1,\psi_2}\subset\UU$, which we see is only positive semi-definite, corresponding to the lack of dissipation in the momentum equation. Computing
	\begin{align*}
		A_{00} = \begin{pmatrix} 0 & 1 & 0 \\ 1  & 0 &\sqrt{2/3} \\ 0 & \sqrt{2/3}  & 0
		\end{pmatrix}
	\end{align*}
	in orthonormal coordinates, we then see that for
	\begin{align*}
		K_{00} =  \begin{pmatrix}0 & 1 & 0 \\ -1 & 0 & 0 \\ 0 & 0 & 0 \end{pmatrix}
	\end{align*}
	we have that $\br{[K_{00},A_{00}]\psi_0,\psi_0}=2>0$. For $\mathfrak{u}\in\UU$, writing $\mathfrak{u}_0 = \br{\mathfrak{u},\psi_0}$ and $\mathfrak{u}=\mathfrak{u}_0\psi_0+\mathfrak{u}_1$,
	we can therefore bound
	\begin{align*}
		\br{\Big(\de[K_{00},A_{00}]+2A_{01}A_{10}\Big)\mathfrak{u},\mathfrak{u}}
		&
		\ge 
		2\de \abs{\mathfrak{u}_0}^2 - 2C\de \abs{\mathfrak{u}_0}\norm{\mathfrak{u}_1}+
		\frac83\norm{\mathfrak{u}_1}^2
		\\
		&\ge
		c\left( \abs{\mathfrak{u}_0}^2  + \norm{\mathfrak{u}_1}^2 \right) = c\norm{\mathfrak{u}}^2,
	\end{align*}
	where the last bound holds for $\de>0$ sufficiently small, by computations of the discriminant. By similar matrix arguments, the first bound in the lemma is proved.
	
	To prove uniform boundedness on $\HH^{\pm1}_\kappa$, we note that $K$ maps into the finite-dimensional space $\HH_3\subset\ca \HH^2$ on which all norms are equivalent, so we can then compute
	\begin{align*}
		\norm{Kf}_{\HH^1_\kappa}\le C\norm{Kf}_{\mathsf{H}^2_2}\le C\norm{f}_{\HH^0}
		\le C\norm{f}_{\HH^1_\kappa}
	\end{align*}
	for a uniform constant $C>0$ using Lemma \ref{HH1kappa_Sobolev_bound}. The bound in $\HH^{-1}_\kappa$ then follows by taking adjoints $K^* = -K$.
\end{proof}

We note that $K_{00}$ can be constructed for much more general systems, by the theory of Kawashima \cite{Kawashima_thesis}.

\begin{lemma}
	\label{Kawashima_compensator_lemma}
	Let $\abs{\bar{\mathfrak{u}} - \uref}\le\epsilon_0$ and $\kappa\in[0,\kappa_0]$ for $\epsilon_0>0$ sufficiently small and $\kappa_0>0$ fixed, and define $K = \de_1\bar K$. Then there exist constants $\de_0,\de_1>0$ uniform in $\epsilon,\kappa$, such that for all $f\in\HH^1_\kappa$, the bound
	\begin{align*}
		\br{\Big(K\xi_1- L_{\bar{\mathfrak{u}},\kappa}\Big)f,f}_{\HH^0}\ge\de_0\norm{f}_{\HH^1_\kappa}
	\end{align*}
	holds. 
\end{lemma}
\begin{proof}
	We use by symmetry of $[\bar K,\xi_1]$ that $\Pi_3 [\bar K,\xi_1]\Pi_3$, and for $f\in\HH_3$ we have the uniform bounds
	\begin{align*}
		\norm{f}_{\HH^1_\kappa}\lesssim\norm{f}_{\mathsf{H}^2_2}\lesssim\norm{f}_{\HH_3}
	\end{align*}
	by Lemma \ref{HH1kappa_Sobolev_bound} and by equivalence of the $\mathsf{H}^2_2$ and $\HH^0$ norms on the finite subspace $\HH_3$. Therefore, using Lemma \ref{uncentred_Lkappa_coercivity_estimate}, for $f\in\HH^1_\kappa$ we can estimate
	\begin{align*}
		\br{\Big(K\xi_1- L_{\bar{\mathfrak{u}},\kappa}\Big)f,f}_{\HH^0}
		&\ge \de\norm{(I-\PP)f}_{\HH^1_\kappa}^2 - C\epsilon\norm{\PP f}\norm{(I-\PP)f}_{\HH^1_\kappa}
		\\
		&\qquad
		+ \de_1\left( c_1\norm{\PP f}^2 - C_1\norm{(I-\PP)f}_{\HH^1_\kappa}^2\right),
	\end{align*}
	and for $\epsilon>0$ sufficiently small, there exists a uniform $\de_1$ such that this can be bounded below by $\de_0(\norm{\PP f}+\norm{(I-\PP)f}_{\HH^1_\kappa})$. But the norms $\UU\oplus\VV^1_\kappa$ and $\HH^1_\kappa$ are uniformly equivalent for a constant independent of  $\kappa\in[0,\kappa_0]$ by the construction of $\HH^1_\kappa$ and the open mapping theorem, and this proves the lemma.
\end{proof}

\section{The Chapman-Enskog approximation}
\label{Chapman-Enskog_Theory_section}

We now justify the approximation of the Boltzmann equation by the compressible Navier-Stokes system \eqref{Compressible_Navier-Stokes_Equations}, using the Chapman-Enskog expansion \cite{ChapmanCowling39}. The derivation is classical, and has been done first for the Boltzmann equation with hard cutoff potentials \cite{KawashimaMatsumuraNishida79} and later for more general kinetic systems \cite{BardosGolseLevermore91}, however we will require uniform estimates in $\kappa$ for the lifted problem to show convergence of the Galerkin scheme in the limit $\kappa\to0$, so we carry out the explicit derivation. Similar uniform estimates were proved for the Landau collision operator in \cite{AlbrittonBedrossianNovack24}. Introducing a time-dependence which will be useful for our derivation, we suppose we have a solution $F$ to
\begin{align}
	\label{time_dependent_shifted_Boltzmann}
	\p_t F + (\xi_1-\sfs)\p_x F = Q_\kappa(F,F),
\end{align}
and make the assumption of a hyperbolic scaling $\tau = \epsilon t$, $\zeta = \epsilon x$, which gives the rescaled equation
\begin{align}
	\label{hyperbolic_rescaled_equation}
	\p_{\tau}F + (\xi_1-\sfs)\p_{\zeta}F = \frac1\epsilon Q_\kappa(F,F),
\end{align}
so that we have a Knudsen number of order $\epsilon$. Making the symmetry assumption $R_2 F = R_3 F = F$ as in Section 1 and defining the macroscopic part $$U:=\PP F\in\UU,$$ we introduce the uncentred micro-macro decomposition
\begin{align*}
	F = M_U +  f^\perp
\end{align*}
for $f^\perp = O(\epsilon)$ purely microscopic,
and we solve for $f^\perp$ up to $O(\epsilon^2)$ order corrections, such that $F$ solves \eqref{hyperbolic_rescaled_equation}. Using the cancellation $Q_\kappa(M_U,M_U)=0$, we can rewrite equation \eqref{hyperbolic_rescaled_equation} as
\begin{align}
	\frac1\epsilon L_{U,\kappa} f^\perp &= \p_\tau F+ (\xi_1-\sfs)\p_\zeta F + O(\epsilon) \notag
	\\
	&= (\p_\tau + (\xi_1-\sfs)\p_\zeta)M_U + O(\epsilon) \notag
	\\
	\label{0th_order_Chapman-Enskog_equation}
	&= \mrm d M_U \p_\tau U + (\xi_1-\sfs)\mrm d M_U\p_\zeta U + O(\epsilon).
\end{align}
Since $L_{\kappa,U} f^\perp$ is purely microscopic, we can take macroscopic moments of this equation, and writing $J(U)=\PP(\xi_1 M_U)$ for the projected transport term, we get
\begin{align}
	0 &=
	\p_\tau\PP F + \p_\zeta\PP\left[ (\xi_1-\sfs) F\right] + O(\epsilon) \notag
	\\
	\label{0th_order_Chapman-Enskog_macro_equation}
	&=\p_\tau U + (J'(U)-\sfs)\p_\zeta U + O(\epsilon).
\end{align}
Expressing $V=(\rho,\bfu,T)$ in hydrodynamic variables \eqref{hydrodynamic_quantities_definition} by the identification of $\UU$ with the space of hydrodynamic quantities \eqref{definition_of_PP}, and using the Maxwellian moment computations \eqref{Maxwellian_hydrodynamic_quantities} and \eqref{Maxwellian_moment_identities}, we see that this can be expressed in conservative variables as
\begin{align*}
	(\p_\tau - \sfs\p_\zeta)\left[ \mathsf{f}_0(\rho,\bfu,T)\right] + \p_\zeta\left[ \mathsf{f}_1(\rho,\bfu,T)\right] = O(\epsilon),
\end{align*}
which are precisely the Euler equations \eqref{Euler_equation} in a moving reference frame, up to an order $O(\epsilon)$ correction. Plugging the macroscopic equation \eqref{0th_order_Chapman-Enskog_macro_equation} back into the original equation \eqref{0th_order_Chapman-Enskog_equation} then gives us
\begin{align*}
	\frac1\epsilon L_{U,\kappa} f^\perp 
	&=
	\left( -\mrm d M_U (J'(U)-\sfs)+(\xi_1-\sfs)\mrm dM_U\right)\p_\zeta u + O(\epsilon)
	\\
	&= \left( I - \PP_U\right)(\xi_1\mrm dM_U)\p_\zeta U + O(\epsilon),
\end{align*}
where we have used the identity
\begin{align*}
	\mrm d M_U J'(U) = \mrm d M_U \PP(\xi_1\mrm d M_U) = \PP_U(\xi_1\mrm d M_U).
\end{align*}
We can therefore define the \emph{microscopic correction}
\begin{align*}
	b_\kappa^\perp[U] = 
	L_{U,\kappa}^{-1}\left[ \left( I - \PP_U\right)(\xi_1\mrm dM_U)\right],
\end{align*}
such that formally we have
\begin{align}
	\label{Chapman-Enskog_correction}
	f^\perp = f_\kappa^\perp[U] 
	=b_\kappa^\perp[U]\, \epsilon \p_\zeta U  
	=
	b_\kappa^\perp[U] \p_x U
\end{align}
up to an order $O(\epsilon^2)$ correction, and we take this as our definition of the Chapman-Enskog correction, and we write $f^\perp[U] = f_0^\perp[U]$ for the Chapman-Enskog correction for the original equation \eqref{traveling_wave_Boltzmann}. Plugging this into \eqref{hyperbolic_rescaled_equation} and taking macroscopic parts, we can then expand to second order to get
\begin{align*}
	0 &=\PP(\p_\tau F + (\xi_1-\sfs)\p_\zeta F) 
	\\
	&= \p_\tau U + \p_\zeta\PP\left[(\xi_1-\sfs)(M_U + \epsilon b^\perp[U]\p_\zeta U + O(\epsilon^2))\right]
	\\
	&= \p_\tau U + (J'(U)-\sfs)\p_\zeta U - \epsilon\p_\zeta(B(U)\p_\zeta U) + O(\epsilon^2)
\end{align*}
where we define the diffusion term
\begin{align*}
	B_\kappa(U):= -\PP(\xi_1 b_\kappa^\perp[U])=
	-\PP\left(\xi_1L_{U,\kappa}^{-1}\left( \left( I - \PP_U\right)(\xi_1\mrm dM_U)\right) \right),
\end{align*}
and write $B(U) = B_0(U)$ for the original diffusion operator. We now compute the diffusion matrix precisely. Setting $\mathsf{V} = T^{-1/2}(\xi - \bfu e_1)$, we compute
\begin{align*}
	b^\perp_\kappa[U]\p_x U
	&=
	L_{U,\kappa}^{-1}\left( \left( I - \PP_U\right)(\xi_1\p_x M_U)\right) 
	\\
	&=
	L_{U,\kappa}^{-1}\left( \left( I - \PP_U\right)\left[\xi_1
	\left(
	\frac{\p M}{\p\rho}\p_x\rho + \frac{\p M}{\p\bfu}\p_x\bfu
	+
	\frac{\p M}{\p T}\p_x T
	\right)
	\right]\right) 
	\\
	&=
	L_{U,\kappa}^{-1}\left( \left( I - \PP_u\right)\left[\xi_1 M_U
	\left(
	\frac{\xi_1 - \bfu}{T}\p_x\bfu
	+
	\frac{1}{2 T}\left( \frac{\abs{\xi - \bfu e_1}^2}{T}-3\right)\p_x T
	\right)
	\right]\right)
	\\
	&=
	L_{U,\kappa}^{-1}\left[M_U
	\left(
	\left (\mathsf{V}_1^2 - \frac13 \abs{\mathsf{V}}^2\right)\p_x\bfu
	+
	\frac{1}{2 \sqrt T}\mathsf{V}_1 \left( \abs{\mathsf{V}}^2-5\right)\p_x T
	\right)
	\right]
	\\
	&=
	L_{U,\kappa}^{-1}\left[M_U
	\left(
	\Phi(\mathsf{V})\p_x\bfu
	+
	T^{-1/2}\Psi(\mathsf{V})\p_x T
	\right)
	\right]
\end{align*}
where we have defined the Burnett functions (cf. \cite{EllisPinsky75_1})
\begin{align*}
	\Phi(\mathsf{V}) = \mathsf{V}_1^2 - \frac13 \abs{\mathsf{V}}^2,\qquad
	\Psi(\mathsf{V}) = \mathsf{V}_1 \left( \abs{\mathsf{V}}^2-5\right).
\end{align*}
We can then define the inverted Burnett functions
\begin{align}
	\label{inverted_Burnett_functions}
	\tilde\Phi_\kappa\Mref = -\Lref_\kappa^{-1}(\Phi\Mref),\qquad
	\tilde\Psi_\kappa\Mref = -\Lref_\kappa^{-1}(\Psi\Mref)
\end{align}
at the reference Maxwellian. For $U\neq\uref$, we have the following lemma.

\begin{lemma}
	\label{inverted_Burnett_formulas}
	Let $\abs{U-\uref}\le\epsilon_0$ for $\epsilon_0>0$ sufficiently small, then we have the formulas
	\begin{align*}
		-L_{U,\kappa}^{-1}(M_U\Phi(\mathsf{V})) 
		&= \rho^{-1}T^{-\gamma/2}M_U\tilde\Phi_{T^{1-s-\ga/2}\kappa}(\mathsf{V}),
		\\
		-L_{U,\kappa}^{-1}(M_U\Psi(\mathsf{V})) 
		&= \rho^{-1}T^{-\gamma/2}M_U\tilde\Psi_{T^{1-s-\ga/2}\kappa}(\mathsf{V})
	\end{align*}
	for all $\kappa\ge 0$. For brevity, we write $\tilde\Phi_{T,\kappa} = \tilde\Phi_{T^{1-s-\ga/2}\kappa}$ and $\tilde\Psi_{T,\kappa} = \tilde\Psi_{T^{1-s-\ga/2}\kappa}$.
\end{lemma}
\begin{proof}
	This will follow from the temperature scaling law
	\begin{align}
		\label{LUkappa_temperature_dependence}
		L_{U,\kappa}[f(\mathsf{V})] = \rho T^{\ga/2}\left( \Lref_{T^{1-s-\ga/2}\kappa}f\right)(\mathsf{V}).
	\end{align}
	and the identity $M_U(\xi) = \rho T^{-3/2}\Mref(\mathsf{V})$, by inverting the operator $\Lref_{T^{1-s-\ga/2}\kappa}$ and using the definition \eqref{inverted_Burnett_functions}. To prove \eqref{LUkappa_temperature_dependence}, we compute
	\begin{align*}
		&
		L_{U,\kappa}[f(\mathsf{V})]
		\\
		&=
		\int_{\RR^3}\!\int_{\SSS^2}\!\left(\abs{\xi-\xi_*}^\ga + \kappa\abs{\xi-\xi_*}^{2-2s}\right) b_s(\cos\theta)
		\\
		&\qquad\qquad\qquad\cdot
		\Big[
		M_U(\xi_*')f(\mathsf{V}') + M_U(\xi')f(\mathsf{V}_*')
		- M_U(\xi_*) f(\mathsf{V}) - M_U(\xi)f(\mathsf{V}_*)
		\Big]
		\dd\sigma\dd\xi_*
		\\
		&=
		\rho \int_{\RR^3}\!\int_{\SSS^2}\!\left(T^{\ga/2}\abs{\mathsf{V}-\mathsf{V}_*}^\ga + \kappa T^{1-s}\abs{\mathsf{V}-\mathsf{V}_*}^{2-2s}\right) b_s(\cos\theta)
		\\
		&\qquad\qquad\qquad\cdot
		\Big[
		\Mref(\mathsf{V}_*')f(\mathsf{V}') + \Mref(\mathsf{V}')f(\mathsf{V}_*')
		- \Mref(\mathsf{V}_*) f(\mathsf{V}) - \Mref(\mathsf{V})f(\mathsf{V}_*)
		\Big]
		\dd\sigma\dd\mathsf{V}_*
	\end{align*}
	for $\mathsf{V}_* = T^{-1/2}(\xi_*-\bfu)$, and $\mathsf{V}',\mathsf{V}_*'$ defined analogously, and removing the $T^{\ga/2}$ factor from the integral then proves the claim.
\end{proof}

Using this lemma, we can then rewrite 
\begin{align}
	\label{microscopic_correction_form}
	b^\perp_\kappa(U)\p_x U=
	\rho^{-1}T^{-\gamma/2}M_U
	\left[
	\tilde\Phi_{T,\kappa}(\mathsf{V})\p_x\bfu+ T^{-1/2}\tilde\Psi_{T,\kappa}(\mathsf{V})\p_x T
	\right],
\end{align}
so we can now compute the hydrodynamical quantities of $B_\kappa(U)\p_x U$. We see in the mass equation that
\begin{align*}
	\int B_\kappa(U)\p_x U\dd\xi = -\int \xi_1 b_\kappa^\perp(U)\p_x U\dd\xi = 0
\end{align*}
since $b_\kappa^\perp(U)$ is purely microscopic. Meanwhile, we can compute
\begin{align*}
	\int\xi_1 B_\kappa(U)\p_x U\dd\xi 
	&=
	-\int\xi_1^2 b_\kappa^\perp(U)\p_x U\dd\xi 
	\\
	&=
	- T \int\left( \mathsf{V}_1^2 - \frac13\abs{\mathsf{V}}^2\right) b_\kappa^\perp(U)\p_x U\dd\xi 
	\\
	&=
	\frac{T^{1-\gamma/2}}{\rho}\int \Phi(\mathsf{V}) M_U\left(
	\tilde\Phi_{T,\kappa}(\mathsf{V})\p_x\bfu + T^{-1/2}\tilde\Psi_{T,\kappa}(\mathsf{V})\p_x T
	\right)\dd\xi
	\\
	&=
	T^{1-\gamma/2}\int \Phi(\mathsf{V}) \Mref(\mathsf{V})\left(
	\tilde\Phi_{T,\kappa}(\mathsf{V})\p_x\bfu + T^{-1/2}\tilde\Psi_{T,\kappa}(\mathsf{V})\p_x T
	\right)\dd\mathsf{V}
	\\
	&= T^{1-\gamma/2}\tilde\mu\p_x\bfu = \mu_\kappa(T)\p_x\bfu
\end{align*}
for $\tilde\mu=\tilde\mu(T^{1-s-\ga/2}\kappa)$. Similarly, we can compute the energy
\begin{align*}
	\int\! \xi_1^2 &B_\kappa(U)\p_x U\dd\xi
	\\
	&=
	-\int\frac12\abs{\xi}^2\xi_1 b_\kappa^\perp(U)\p_x U\dd\xi
	\\
	&=
	-\int\left(\frac12 T^{3/2}(\abs{\mathsf{V}}^2-5)\mathsf{V}_1+T\left(\mathsf{V}_1^2-\frac13\abs{\mathsf{V}}^2\right)\bfu\right) b_\kappa^\perp(U)\p_x U\dd\xi
	\\
	&=\frac{1}{\rho T^{\gamma/2}}\int\left( T^{3/2}\Psi(\mathsf{V})+T\bfu\Phi(\mathsf{V})\right) M_U\left(
	\tilde\Phi_{T,\kappa}(\mathsf{V})\p_x\bfu + T^{-1/2}\tilde\Psi_{T,\kappa}(\mathsf{V})\p_x T
	\right)\dd\xi
	\\
	&= T^{1-\gamma/2}\left( \tilde\mu\bfu\p_x\bfu + \tilde\varkappa\p_x T\right)
	= \mu(T)\bfu\p_x\bfu + \varkappa_\kappa(T)\p_x T,
\end{align*}
for $\tilde\varkappa=\tilde\varkappa(T^{1-s-\ga/2}\kappa)$, so we see for the original equation $\kappa=0$, we have coefficients $\mu,\varkappa$ varying as $T^{1-\ga/2}$. We then get the diffusion matrix
\begin{align}
	\label{diffusion_matrix_definition}
	B_\kappa(U)\p_x U
	=
	\begin{pmatrix}
		0 & 0 & 0 \\
		0 & \mu_\kappa & 0 \\
		0 & \mu_\kappa\bfu & \varkappa_\kappa
	\end{pmatrix}
	\begin{pmatrix} 
		\p_x\rho \\ \p_x\bfu \\ \p_x T
	\end{pmatrix}
\end{align}
in conservative variables, and we write $\mathsf{b}_{0}$ for the matrix in \eqref{diffusion_matrix_definition}, suppressing the $\kappa$ where there is no risk of ambiguity. We note that the diffusion coefficients are independent of the density. We have the following uniform estimates on the microscopic correction and the diffusion matrix.

\begin{lemma}
	\label{microscopic_correction_smoothness}
	The microscopic correction $b^\perp_\kappa[\bar{\mathfrak{u}}]\in\mathsf{H}^2_2$ varies smoothly in $(\bar{\mathfrak{u}},\kappa)\in B(\uref,\epsilon_0)\times[0,\kappa_0]$ for $\epsilon_0>0$ sufficiently small.
\end{lemma}
\begin{proof}
	Using \eqref{microscopic_correction_form}, this reduces to showing that $\Mref\tilde\Phi_\kappa,\Mref\tilde\Psi_\kappa\in\mathsf{H}^2_2$, with smooth dependence on $\kappa\ge 0$. 
	By Lemma \ref{Qkappa_trilinear_estimate} and the Lax-Milgram theorem in the modified $\widetilde{\mathsf{H}^2_2}$ norm, and the fact that $\Mref\Phi,\Mref\Psi\in\mathsf{H}^2_2$ are microscopic, we get that $\Mref\tilde\Phi_\kappa,\Mref\tilde\Psi_\kappa\in\mathsf{H}^2_2$ with uniform estimates in $\kappa\in[0,\kappa_0]$. To show $\Mref\tilde\Phi_\kappa,\Mref\tilde\Psi_\kappa\in C^\infty([0,\kappa_0];\mathsf{H}^2_2)$, we compute the derivative $\p_\kappa\Lref_\kappa=\Lref_{s,2-2s}$ and derive the identity
	\begin{align*}
		\p_\kappa\Lref_\kappa^{-1} = -\Lref_\kappa^{-1}\Lref_{s,2-2s}\Lref_\kappa^{-1}.
	\end{align*}
	Using that $\p_\kappa^2 \Lref_\kappa=0$ and arguing by induction, we derive
	\begin{align*}
		\p_\kappa^n\Lref_\kappa^{-1} = (-1)^nn!\left( \Lref_\kappa^{-1}\Lref_{2,2-2s}\right)^n\Lref_\kappa^{-1},
	\end{align*}
	and note that $\Lref_{2,2-2s}$ is bounded $\mathsf{H}^{\ell+4}_{\ell+4}\to\mathsf{H}^\ell_\ell$ and $\Lref_\kappa^{-1}$ are uniformly bounded operators on $\mathsf{H}^\ell_\ell\cap \VV$ in $\kappa\in[0,\kappa_0]$, for all $\ell\ge0$. Since $\Mref\Phi,\Mref\Psi\in\mathsf{H}^{2+4n}_{2+4n}$ for arbitrary $n\ge0$, this proves the lemma.
\end{proof}

\begin{corollary}
	\label{smoothness_of_the_diffusion_coefficients}
	The diffusion matrix $B_\kappa(\bar{\mathfrak{u}})$ is smooth for $(\bar{\mathfrak{u}},\kappa)\in B(\uref,\epsilon_0)\times[0,\kappa_0]$ for $\epsilon_0>0$ sufficiently small.
\end{corollary}
\begin{proof}
	The proof is immediate from Lemma \ref{microscopic_correction_smoothness}, by continuity of the operator $\PP[\xi_1\cdot]$ on $\mathsf{H}^2_2$.
\end{proof}

\begin{lemma}[\cite{AlbrittonBedrossianNovack24}]
	\label{ABN24_Prop_B.2}
	Fix the left reference state $V_-=(\rho_-,\bfu_-,T_-)=(1,0,1)$, and let $V_+ = V_{+,3}(\epsilon)$ be the unique state on the third Rankine-Hugoniot curve defined by \eqref{Rankine-Hugoniot_conditions} and \eqref{Lax_entropy}, satisfying $\lambda_3(V_+) = \lambda_3(V_-)-\epsilon$ for $\epsilon<\epsilon_0$ sufficiently small, with shock speed $\sfs = \sfs_3(\epsilon)$ (cf. Chapter 8 of \cite{Dafermos_book}). Then there exists a family of solutions $\bar  V_{\mrm{NS},\epsilon}(x-\sfs t)$ to the compressible Navier-Stokes equations \eqref{Compressible_Navier-Stokes_Equations} with boundary conditions
	\begin{align}
		\label{ABN24_Prop_B.2_boundary_conditions}
		\lim_{x\to\pm\infty}\bar V_{\mrm{NS},\epsilon}(x) = V_\pm,
	\end{align}
	which vary continuously in $C^0(\RR)$ in $\epsilon>0$. Writing $\bar{\mathfrak{u}}_{\mrm{NS},\epsilon} = (I|_{\UU})^{-1}\circ\mathsf{f}_0(V_{\mrm{NS},\epsilon})$ for the $\UU$-valued lift defined by \eqref{hydro_macro_change_of_variables}, we have the uniform estimates
	\begin{align}
		\label{ABN24_Prop_B.2_exponential_decay}
		\abs{\p_x^k \bar{\mathfrak{u}}_{\mrm{NS},\epsilon}(x)}\le C_k\epsilon^{k+1}e^{-\epsilon\theta\abs x},\qquad k\ge 1
	\end{align}
	and
	\begin{align*}
		&\abs{\mathfrak{u}_\pm - \bar{\mathfrak{u}}_{\mrm{NS},\epsilon}(
			\pm x)}\le C_0\epsilon\, e^{-\epsilon\theta\abs x},\qquad x\ge 0
	\end{align*}
	for constants $C_k,\theta>0$ independent of $\epsilon$. We can further assume that $\abs{\bar{\mathfrak{u}}_{\mrm{NS},\epsilon}'(0)}\ge c_0\epsilon^2$ for some constant $c_0>0$ independent of $\epsilon$. Where there is no risk of ambiguity, we write $\uns = \bar{\mathfrak{u}}_{\mrm{NS},\epsilon}$.
\end{lemma}

We rewrite the equation for solutions $\bar{V}_{\mrm{NS}}(x,t) = \bar{V}_{\mrm{NS}}(x-\sfs t)$ to \eqref{Compressible_Navier-Stokes_Equations} in coordinate-invariant form
\begin{align}
	\label{traveling_wave_CNSE}
	\p_x\left[ J(\uns)-\sfs\,\uns\right] + \p_x\left[ B(\uns)\p_x\uns\right]=0,
\end{align}
which will be useful for later computations.
\begin{proof}
	Existence and uniqueness up to translation of smooth solutions $\bar  V_{\mrm{NS},\epsilon}(x-\sfs t)$ to \eqref{Compressible_Navier-Stokes_Equations} with boundary conditions \eqref{ABN24_Prop_B.2_boundary_conditions} was proved in \cite{Gilbarg51}, for general smooth viscosity $\nu(T)>0$ and thermal conductivity $\varkappa(T)>0$ coefficients. The uniform estimates \eqref{ABN24_Prop_B.2_exponential_decay} were proved in Apppendix B of \cite{AlbrittonBedrossianNovack24} for solutions to the general five-quantity Navier-Stokes system with three-dimensional velocity $\bfu\in\RR^3$, with endpoints $\bfu_\pm = (\bfu_{\pm,1},0,0)\in\RR^3$, with uniqueness up to translation shown during the proof. But since this equation is invariant under the transformations $(\bfu,\bfu_2,\bfu_3)\mapsto (\bfu,-\bfu_2,\bfu_3)$ and $(\bfu,\bfu_2,\bfu_3)\mapsto (\bfu,\bfu_2,-\bfu_3)$, we either have $e_2\cdot\bar{\bfu}_{\mrm{NS}}\equiv e_3\cdot\bar{\bfu}_{\mrm{NS}}\equiv 0$, or $e_1\cdot\bar{\bfu}_{\mrm{NS}}(x) = e_1\cdot\bar{\bfu}_{\mrm{NS}}(x+C)$ for some $C>0$ by uniqueness up to translation, but this contradicts the boundary conditions, so the solutions constructed in \cite{AlbrittonBedrossianNovack24} are in fact solutions to \eqref{Compressible_Navier-Stokes_Equations}, and the bounds \eqref{ABN24_Prop_B.2_exponential_decay} are proved.
	
	Continuous dependence on $\epsilon>0$ in the $C^0(\RR)$ norm and the bound $\abs{\bar{\mathfrak{u}}_{\mrm{NS},\epsilon}'(0)}\ge c_0\epsilon^2$ for suitably chosen solutions up to translation are immediate consequences of the geometric construction in \cite{Gilbarg51}.
\end{proof}

We can then directly deduce the following asymptotics on the Chapman-Enskog correction term.

\begin{lemma}
	\label{Chapman-Enskog_correction_lemma}
	Let $\epsilon>0$ be sufficiently small, then the Chapman-Enskog correction $f^\perp[\uns]$ defined in \eqref{Chapman-Enskog_correction} satisfies the bounds
	\begin{align*}
		\norm{ f^\perp[\uns] }_{H^k_\epsilon\HH^1}\le 
		C\norm{ f^\perp[\uns] }_{H^k_\epsilon\mathsf{H}^2_2}\le 
		C_k\epsilon^{k+2}
	\end{align*}
	for every $k\ge 0$.
\end{lemma}
\begin{proof}
	This follows immediately from the estimate \eqref{ABN24_Prop_B.2_exponential_decay} and from smoothness of the microscopic correction $b^\perp\in\mathsf{H}^2_2\subset\HH^1$ proved in
	\ref{microscopic_correction_smoothness}.
\end{proof}

If we define the macroscopic entropy $s = -\log(T^{-3/2}\rho)$ from the first law of thermodynamics, we have the following classical result, which will provide a useful reformulation of the compressible Navier-Stokes equations \eqref{Compressible_Navier-Stokes_Equations}.

\begin{lemma}[\cite{KawashimaShizuta88}]
	\label{Kawashima_CNSE_symmetrizer}
	Define the scalar function $\eta = -\rho s$ and let $A_0 = \na^2_{(\rho,m,E)}\eta$ be its associated Hessian matrix in conservative variables. Then the matrix $A_0$ is positive definite, the matrix $\mathsf f_1'(\mathsf f_0')^{-1}A_0$ is symmetric, and $\mathsf b_0 (\mathsf f_0')^{-1}A_0$ is non-negative definite.
\end{lemma}

This shows that the compressible Navier-Stokes equations \eqref{Compressible_Navier-Stokes_Equations} form a symmetric hyperbolic-parabolic system in the sense of Kawashima \cite{Kawashima_thesis}. The proof in \cite{KawashimaShizuta88} proceeds by direct matrix computations, however it was shown in \cite{KawashimaYong04} that Lemma \ref{Kawashima_CNSE_symmetrizer} can be directly derived from properties of the kinetic entropy (see also the exposition in \cite{MetivierZumbrun09_semilinear}).

\section{Fixed point argument}

We now define the explicit linearization of the problem. Assuming $F\in H^2_\epsilon\HH^1$ is a solution to \eqref{traveling_wave_Boltzmann}, we can write the second order Chapman-Enskog expansion
\begin{align*}
	F = M_{\uns} + f^\perp[\uns] + f,
\end{align*}
and linearizing \eqref{traveling_wave_Boltzmann} around the lifted Navier-Stokes shock profile $M_{\uns}$ gives the operator
\begin{align*}
	\ca L_\epsilon f = \xi_1\p_x f - L_{\uns}f.
\end{align*}
We can then re-express \eqref{traveling_wave_Boltzmann} as 
\begin{align}
	\label{linearization_of_the_problem}
	\ca L_\epsilon f = \ca E[f],
\end{align}
for an error term which we compute as 
\begin{align}
	\begin{split}
		\label{error_term_computation}
		\ca E[f] 
		&= L_{\uns}f^\perp[\uns]-(\xi_1-\sfs)\p_x\left[ M_{\uns}-f^\perp[\uns]\right]
		\\
		&\qquad
		+Q(f^\perp[\uns],f^\perp[\uns]) 
		\\
		&\qquad
		+ Q(f^\perp[\uns],f) + Q(f,f^\perp[\uns])
		\\
		&\qquad
		+ Q(f,f).
	\end{split}
\end{align}

By the construction of the Chapman-Enskog correction $f^\perp[\uns]$, we will see that the error term is of order $\epsilon^3$, and is purely microscopic. This latter fact will somewhat simplify the linearized a priori estimates in Section \ref{Micro_section} and Section \ref{Macro_section}, and is key for showing continuity of the problem in the limit $\kappa\to 0$ once we construct a right inverse to $\ca L_{\epsilon,\kappa}$ for $\kappa>0$ by a Galerkin scheme in Section \ref{Galerkin_scheme}. Using these estimates, we will construct solutions to the original problem \eqref{traveling_wave_Boltzmann} by contraction mapping.
We summarize these facts in the following lemma.

\begin{lemma}
	\label{Linearization_error_estimate}
	Let $\epsilon\in(0,\epsilon_0)$ for sufficiently small $\epsilon_0>0$. Then for any $f,g\in H^k_\epsilon\HH^1$ for $k\ge 2$, the error term is purely microscopic, and has bounds
	\begin{align*}
		\norm{\ca E[f]}_{H^k_\epsilon\HH^{-1}}
		\le
		C_k
		\left(
		\epsilon^3 + \epsilon^2\norm{f}_{H^k_\epsilon\HH^1} + \norm{f}_{H^k_\epsilon\HH^1}^2
		\right)
	\end{align*}
	and
	\begin{align*}
		\norm{\ca E[f]-\ca E[g]}_{H^k_\epsilon\HH^{-1}}
		\le 
		C_k\left( \epsilon^2 + \norm{f+g}_{H^k_\epsilon\HH^1}\right)\norm{f-g}_{H^k_\epsilon\HH^1}.
	\end{align*}
	for some constant $C_k>0$ independent of $\epsilon>0$. 
\end{lemma}
\begin{proof}
	To see that $\ca E[f]$ is microscopic, we compute
	\begin{align*}
		\PP\ca E[f] 
		&= \p_x\PP\left[ (\xi_1-\sfs)(M_{\uns}-f^\perp[\uns]\right]
		\\
		&= \p_x \left[ J(\uns)-\sfs\, \uns- B(\uns)\p_x\uns\right]=0
	\end{align*}
	since the shock profile $\uns$ is constructed to solve \eqref{traveling_wave_CNSE}. But we can therefore re-express
	\begin{align*}
		L_{\uns}f^\perp[&\uns]-(\xi_1-\sfs)\p_x\left[ M_{\uns}-f^\perp[\uns]\right]
		\\
		&=
		\left( I - \PP_{\uns}\right)(\xi_1-\sfs)\p_x M_{\uns}
		- \left( I - \PP_{\uns}\right)(\xi_1-\sfs)\p_x\left[ M_{\uns}-f^\perp[\uns]\right]
		\\
		&= \left( I - \PP_{\uns}\right)(\xi_1-\sfs)\p_xf^\perp[\uns],
	\end{align*}
	by the definition of $f^\perp[\uns]$ in \eqref{Chapman-Enskog_correction},
	and we can then bound
	\begin{align*}
		\norm{\ca E[f]}_{H^k_\epsilon\HH^{-1}}
		\le
		C_k
		\left(
		\epsilon^3 + \epsilon^2\norm{f}_{H^k_\epsilon\HH^1} + \norm{f}_{H^k_\epsilon\HH^1}^2
		\right)
	\end{align*}
	by the estimates on $f^\perp[\uns]$ in Lemma \ref{Chapman-Enskog_correction_lemma} and the bilinear estimates in Corollary \ref{basic_GS11_estimate_corollary}, and the fact that $H^k_\epsilon$ is a Sobolev algebra with uniformly bounded multiplication in $\epsilon>0$ for $k\ge 1$. To bound the difference $\ca E[f] - \ca E[g]$, we compute
	\begin{align*}
		Q(f,f) - Q(g,g) = \frac12 \left(Q(f-g,f+g) + Q(f+g,f-g)\right)
	\end{align*}
	and similarly by Lemma \ref{Chapman-Enskog_correction_lemma} and Corollary \ref{basic_GS11_estimate_corollary} we get the desired bound.
\end{proof}

Combining this error estimate with the right inverse constructed in Proposition \ref{right_inverse_proposition}, we can now prove the main theorem by a contraction mapping argument.

\begin{proof}[Proof of Theorem \ref{main_theorem}]
	We define the operator $\ca T[f] = \ca L_\epsilon^\dagger\ca E[f]$, and we see by the construction of the right inverse $\ca L_\epsilon^\dagger$ that $\ca T[f]=f$ holds if and only if $f$ solves \eqref{linearization_of_the_problem} and $\ell_\epsilon\cdot\PP f(0)=0$.
	We show that $\ca T$  admits a fixed point on $H^k_\epsilon\HH^{-1}$ for $\epsilon>0$ sufficiently small. By Lemma \ref{Linearization_error_estimate} and the bound in Proposition \ref{right_inverse_proposition}, we have the estimates
	\begin{align}
		\begin{split}
			\label{fixed_point_bounds}
			\norm{\ca T[f]}_{H^k_\epsilon\HH^1}
			&
			\le
			C_k\left( \epsilon^2 + \epsilon\norm{f}_{H^k_\epsilon\HH^1} + \frac1\epsilon \norm{f}_{H^k_\epsilon\HH^1}^2\right),
			\\
			\norm{\ca T[f]-\ca T[g]}_{H^k_\epsilon\HH^1}
			&
			\le
			C_k\left( \epsilon + \frac1\epsilon \norm{f+g}_{H^k_\epsilon\HH^1}\right)\norm{f-g}_{H^k_\epsilon\HH^1},
		\end{split}
	\end{align}
	therefore $\ca T$ is a contraction mapping on the subspace $$\set{f\in H^k_\epsilon\HH^1\,\middle|\,\norm{f}_{H^k_\epsilon\HH^1}\le 2 C_k\epsilon^2}$$ for $\epsilon<\epsilon_0(k)$ sufficiently small, and combined with the estimate on the Chapman-Enskog correction $f^\perp[\uns]$ in Lemma \ref{Chapman-Enskog_correction_lemma}, this proves existence of a solution $F\in H^k_\epsilon\HH^1$ satisfying the first estimate in Theorem \ref{main_theorem}.
	
	To prove uniqueness up to translation in a larger set, we note that $\ca T$ is also a contraction mapping in the larger ball
	\begin{align*}
		\set{f\in H^k_\epsilon\HH^1\,\middle|\,\norm{f}_{H^k_\epsilon\HH^1}\le \epsilon/4 C_k}
	\end{align*}
	for $\epsilon<\epsilon_0(k)$ sufficiently small, and therefore $F$ is unique among solutions to \eqref{traveling_wave_Boltzmann} satisfying the estimate
	\begin{align}
		\label{uniqueness_estimate}
		\norm{F - M_{\uns}}_{H^2_\epsilon\HH^1}\le c_1\epsilon
	\end{align}
	and the boundary condition
	\begin{align}
		\label{uniqueness_boundary_condition}
		\ell_\epsilon\cdot\PP F(0) = \ell_\epsilon\cdot M_{\uns(0)}
	\end{align}
	for some small $c_1>0$.
	If $\bar F$ is a solution to \eqref{traveling_wave_Boltzmann} which satisfies \eqref{uniqueness_estimate} but not \eqref{uniqueness_boundary_condition}, then we wish to show that $\bar F_a(x,\xi) = \bar F(x+a,\xi)$ satisfies the boundary condition \eqref{uniqueness_boundary_condition} for some $a\in\RR$, which would prove uniqueness up to translation. We write
	\begin{align*}
		\bar F = M_{\uns} + f^\perp[\uns] + \bar f
	\end{align*} 
	and compute the derivative
	\begin{align*}
		\ell_\epsilon\cdot\PP\p_x\bar F(x) = \ell_\epsilon\cdot\uns'(x) + \ell_\epsilon\cdot\PP\p_x \bar f(x)
	\end{align*}
	where we have used the identity \eqref{macro_of_differentiated_Maxwellian} and the fact that $f^\perp[\uns]$ is purely microscopic. But by bounding
	\begin{align*}
		\abs{\ell_\epsilon\cdot\PP\p_x \bar f(x)}\le\norm{\p_x\bar f}_{L_x^\infty\HH^1}
		\le C\norm{\p_x \bar f}_{H^1_\epsilon\HH^1}\le C c_1 \epsilon^2
	\end{align*}
	for a uniform constant $C>0$ by Sobolev embedding and by assumption on $\bar F$ and Lemma \ref{Chapman-Enskog_correction_lemma}, for $c_1>0$ sufficiently small, we can use the nondegeneracy condition in Proposition \ref{right_inverse_proposition} and the viscous shock asymptotics in Lemma \ref{ABN24_Prop_B.2} to bound
	\begin{align*}
		\abs{\ell_\epsilon\cdot\PP\p_x\bar F(x)}\ge \left( c_0 - C c_1\right)\epsilon^2,\qquad x\in\left[-\frac{M}{\epsilon},\frac{M}{\epsilon}\right]
	\end{align*}
	for $c_1>0$ sufficiently small and some $M>0$. But since $F\in H^2_\epsilon\HH^1$ implies continuity of the quantity inside the absolute value, it must therefore stay the same sign, and since \eqref{uniqueness_estimate} implies
	\begin{align*}
		\abs{\ell_\epsilon\cdot\PP(\bar F(0) - M_{\uns(0)})}\le c_1\epsilon,
	\end{align*}
	the identity $\p_a\bar F_a(0) = \p_x\bar F(x)$ means that $F_a$ must satisfy \eqref{uniqueness_boundary_condition} for some $a\in[-M/\epsilon,M/\epsilon]$, proving the theorem.
\end{proof}

\section{Microscopic energy estimates}
\label{Micro_section}

We now carry out the basic energy estimate on the linearized problem
\begin{align}
	\label{Section5_equation}
	\ca L_{\epsilon,\kappa} f := (\xi_1-\sfs)\p_x f - L_{\uns,\kappa}f=z
\end{align}
for $f\in H^k_\epsilon\HH^1_\kappa$, where we assume a purely microscopic remainder term $z\in H^k_\epsilon\VV^{-1}_\kappa$. Since we can only construction solutions through the Galerkin scheme with the lifting parameter $\kappa>0$, we will require uniform estimates in $\kappa\in[0,\kappa_0]$ to prove existence for the linearized problem in Proposition \ref{right_inverse_proposition}. We keep the assumption of a microscopic remainder term $z$ for the rest of this section.

We can then define the twisted energy functional
\begin{align}
	\label{energy_functional}
	E[f] = \int_{\RR}\!\br{\p_x z,\p_x f}\dd x + \lambda\int_{\RR}\!\br{z,f}\dd x
	+
	\int_{\RR}\!\br{K z,\p_x f}\dd x
\end{align}
for a parameter $\lambda>0$ which we will choose later. This is the terminology used in \cite{AlbrittonBedrossianNovack24}, where the $K$ term allows us to exploit hypocoercivity arising from transport. To show coercivity of the energy, we compute the following.

\begin{lemma}
	\label{main_energy_coercivity}
	Let $\epsilon\in(0,\epsilon_0)$ and $\kappa\in[0,\kappa_0]$ for $\epsilon_0>0$ sufficiently small and $\kappa_0>0$ fixed. Then
	for $f\in H^1_\epsilon\HH^1_\kappa$ and $z\in H^1_\epsilon\HH^{-1}_\kappa$, and for $\lambda>0$ sufficiently large (independently of $\epsilon,\kappa$) we have the coercivity estimate
	\begin{align*}
		E[f]\ge\de_0\left(\norm{\p_x f}_{L^2\HH^1_\kappa}^2+\lambda\norm{\mathfrak{v}}_{L^2\HH^1_\kappa}^2\right)
		&-
		C\epsilon\norm{\mathfrak{u}}_{L^2}
		\left(
		\norm{\p_x f}_{L^2\HH^1_\kappa} + \lambda\norm{\mathfrak{v}}_{L^2\HH^1_\kappa}
		\right)
	\end{align*}
	for constants $\de_0,C>0$ independent of $\epsilon,\kappa,\lambda$.
\end{lemma}
\begin{proof}
	To prove the estimate, we compute
	\begin{align*}
		E[f] 
		&=
		-\int_{\RR}\!\br{\p_x\Big[ L_{\bar{\mathfrak{u}}_{\mrm{NS}},\kappa}f\Big],\p_x f}\dd x
		-
		\lambda \int_{\RR}\! \br{ L_{\uns,\kappa}f,f}\dd x
		\\
		& \qquad\qquad+
		\int_{\RR}\!\br{ K\Big[(\xi_1-\sfs)\p_x-L_{\uns,\kappa}\Big] f,\p_xf}\dd x
		\\
		&=
		\int_{\RR}\!\br{ \Big[ K\xi_1 -L_{\bar{\mathfrak{u}}_{\mrm{NS}},\kappa}\Big]\p_x f,\p_x f}\dd x
		\\
		&\qquad\qquad
		+
		\lambda\int_{\RR}\!\br{- L_{\uns,\kappa}f,f}\dd x
		\\
		&\qquad\qquad
		-
		\int_{\RR}\!\br{Q_\kappa(\p_x M_{\uns},f)+Q_\kappa(f,\p_x M_{\uns}),\p_x f}\dd x
		\\
		&\qquad\qquad
		-
		\int_{\RR}\!\br{K L_{\uns,\kappa}f,\p_x f}\dd x
		\\
		&= I_1 + I_2 - I_3 - I_4,
	\end{align*}
	where we have used antisymmetry of the compensator $K$ and collected terms. Using Lemma \ref{Kawashima_compensator_lemma} and Lemma \ref{uncentred_Lkappa_coercivity_estimate} we can bound
	\begin{align*}
		I_1 + I_2 \ge
		\de_0\left(\norm{\p_x f}_{L^2\HH^1_\kappa}^2+\lambda\norm{\mathfrak{v}}_{L^2\HH^1_\kappa}^2\right)
		- C\epsilon\lambda\norm{\mathfrak{u}}_{L^2}\norm{\mathfrak{v}}_{L^2}.
	\end{align*}
	
	We now bound the error terms. Using Lemmas \ref{Maxwellian_manifold_smoothness}, \ref{HH1kappa_Sobolev_bound}, and \ref{ABN24_Prop_B.2} we can compute
	\begin{align*}
		\norm{\p_x M_{\uns}}_{\HH^1_\kappa}\le\norm{d M_{\uns(x)}}_{\mathsf{H}^2_2}\abs{\p_x\uns(x)}\le C\epsilon^2
	\end{align*}
	for a constant $C$ uniform in $x\in\RR$ and $\epsilon,\kappa$. Combining this with Lemma \ref{Qkappa_trilinear_estimate} we gives
	\begin{align*}
		I_3 \le C\epsilon^2\norm{f}_{L^2\HH^1_\kappa}\norm{\p_x f}_{L^2\HH^1_\kappa}
	\end{align*}
	for a uniform constant $C$. To bound the last term, we define the uncentred mircoscopic projection 
	\begin{align*}
		\tilde{\mathfrak{v}} = (I-\PP_{\uns})f = \mathfrak{v}+(I-\dd M_{\uns})\mathfrak{u}
		= \mathfrak{v} + (I-\PP)\dd M_{\uns}\mathfrak{u}
	\end{align*}
	such that $L_{\uns,\kappa}\tilde{\mathfrak{v}}= L_{\uns,\kappa}f$ by \eqref{twisted_Lukappa_projections}. Using Lemma \ref{Kawashima_compensator_construction}, we can then bound
	\begin{align*}
		\norm{KL_{\uns,\kappa}\tilde{\mathfrak{v}}}_{\HH^{-1}_\kappa}
		&\le
		C\norm{L_{\uns,\kappa}\tilde{\mathfrak{v}}}_{\HH^{-1}_\kappa}
		\\
		&\le C\norm{M_{\uns}}_{\HH^1_\kappa}\norm{\tilde{\mathfrak{v}}}_{\HH^1_\kappa}
		\\
		&
		\le
		C\norm{M_{\uns}}_{\mathsf{H}^2_2}\norm{\tilde{\mathfrak{v}}}_{\HH^1_\kappa}
		\le C\norm{f}_{\HH^1_\kappa}
	\end{align*}
	for constants $C$ uniform in $\kappa,\epsilon$, using the norm bounds on $\HH^1_\kappa$ from Lemma \ref{HH1kappa_Sobolev_bound}, giving the bound
	\begin{align*}
		I_4
		&\le \norm{\tilde{\mathfrak{v}}}_{L^2\HH^1_\kappa}\norm{\p_x f}_{L^2\HH^1_\kappa}
		\\
		&\le\left(\norm{\mathfrak{v}}_{L^2\HH^1_\kappa} + \norm{(I-\PP)\dd M_{\uns}\mathfrak{u}}_{L^2\HH^1_\kappa}\right) \norm{\p_x f}_{L^2\HH^1_\kappa}
		\\
		&\le\left(\norm{\mathfrak{v}}_{L^2\HH^1_\kappa} +C \epsilon\norm{\mathfrak{u}}_{L^2}\right) \norm{\p_x f}_{L^2\HH^1_\kappa}
	\end{align*}
	for a constant $C$ independent of $\epsilon$, using the Maxwellian correction bound from Lemma \ref{Maxwellian_manifold_smoothness}. Combining these estimates and absorbing the term $I_3$ into the bound on $I_4$ for $\epsilon$ sufficiently small gives the estimate
	\begin{align*}
		E[f]\ge\de_0\left(\norm{\p_x f}_{L^2\HH^1_\kappa}^2+\lambda\norm{\mathfrak{v}}_{L^2\HH^1_\kappa}^2\right)
		&-
		C\norm{\p_x f}_{L^2\HH^1_\kappa}\left(\norm{\mathfrak{v}}_{L^2\HH^1_\kappa}
		+
		\epsilon\norm{\mathfrak{u}}_{L^2}\right) 
		\\
		&- C\epsilon\lambda\norm{\mathfrak{u}}_{L^2}\norm{\mathfrak{v}}_{L^2}
	\end{align*}
	for arbitrary $\lambda>0$. Taking $\lambda$ sufficiently large ensures that we can absorb the term $\norm{\p_x f}_{L^2\HH^1_\kappa}\norm{\mathfrak{v}}_{L^2\HH^1_\kappa}$ by the positive terms, proving the claim.
\end{proof}

We now prove the main energy estimate for non-cutoff Boltzmann.

\begin{proposition}
	\label{main_energy_estimate}
	Let $\epsilon\in(0,\epsilon_0]$ and $\kappa\in[0,\kappa_0]$. If $\ca L_{\epsilon,\kappa} f = z\in H^1_\epsilon\VV^{-1}_\kappa$ and $f\in H^1_\epsilon\HH^1_\kappa$, then we have the estimate
	\begin{align*}
		\norm{\p_x f}_{L^2_\epsilon\HH^1_\kappa}^2 
		+ \norm{ \mathfrak{v}}_{L^2_\epsilon\HH^1_\kappa}^2
		\le C\left( 
		\norm{z}_{L^2_\epsilon\HH^{-1}_\kappa}^2 + \norm{\p_x z}_{L^2_\epsilon\HH^{-1}_\kappa}^2
		+
		\epsilon^2 \norm{\mathfrak{u}}_{L^2_\epsilon}^2\right)
	\end{align*}
	for a constant $C$ independent of $\epsilon$ and $\kappa$.
\end{proposition}
\begin{proof}
	We combine the coercivity bound in Lemma \ref{main_energy_coercivity} with the upper bound
	\begin{align*}
		E[f]\le\norm{\p_x z}_{L^2\HH^{-1}_\kappa}\norm{\p_x f}_{L^2\HH^1_\kappa}
		+\lambda\norm{z}_{L^2\HH^{-1}_\kappa}\norm{\mathfrak{v}}_{L^2\HH^1_\kappa}
		+C\norm{z}_{L^2\HH^{-1}_\kappa}\norm{\p_x f}_{L^2\HH^1_\kappa}
	\end{align*}
	to get the estimate
	\begin{align*}
		\de_0\left(\norm{\p_x f}_{L^2\HH^1_\kappa}^2+\lambda\norm{\mathfrak{v}}_{L^2\HH^1_\kappa}^2\right)
		&\le
		\left( \norm{\p_x z}_{L^2\HH^{-1}_\kappa} + C\epsilon\norm{\mathfrak{u}}_{L^2} \right)
		\norm{\p_x f}_{L^2\HH^1_\kappa}
		\\
		&\quad
		+
		\lambda\left( \norm{z}_{L^2\HH^{-1}_\kappa} + C\epsilon\norm{\mathfrak{u}}_{L^2} \right)
		\norm{\mathfrak{v}}_{L^2\HH^1_\kappa}
		\\
		&\quad
		+
		C\norm{z}_{L^2\HH^{-1}_\kappa}\norm{\p_x f}_{L^2\HH^1_\kappa}
	\end{align*}
	and the result follows by Young's inequality and multiplying the equation by $\epsilon$, using the identity $\epsilon\norm\cdot_{L^2}^2=\norm\cdot_{L^2_\epsilon}^2$.
\end{proof}

We have the following high frequency estimate.

\begin{proposition}
	\label{high_frequency_estimate}
	For $\ca L_\epsilon f = z\in H^k_\epsilon\VV^{-1}_\kappa$ and $k\ge 1$ and $f\in H^k_\epsilon\HH^1_\kappa$, we have the estimate
	\begin{align*}
		\norm{\p^{k+1}_x f}_{L^2_\epsilon\HH^1_\kappa}^2 + \norm{\p_x^{k} \mathfrak{v}}_{L^2_\epsilon\HH^1_\kappa}
		&\le
		C
		\norm{\p_x^k z,\p_x^{k+1}z}_{L^2_\epsilon\HH^{-1}_\kappa} 
		\\
		&\qquad\quad
		+
		C_k\epsilon^k\big(
		\norm{\p_x f}_{H^{k-1}_\epsilon\HH^1_\kappa}
		+ \epsilon\norm{\mathfrak{v}}_{H^{k-1}_\epsilon\HH^1_\kappa}
		+\epsilon\norm{\mathfrak{u}}_{L^2_\epsilon\HH^1_\kappa}
		\big)
		\Big)
	\end{align*}
	for constants $C,C_k>0$ independent of $\epsilon,\kappa$.
\end{proposition}

To prove this, we need the following lemma.

\begin{lemma}
	\label{Maxwellian_shock_smoothness}
	For $\epsilon,\kappa>0$ small, we have uniform bounds
	\begin{align*}
		\norm{\p_x^k M_{\uns}}_{L^\infty\HH^1_\kappa}\le C_k\epsilon^{k+1}
	\end{align*}
	for all $k\ge 1$, for $C_k$ independent of $\epsilon$ and $\kappa$.
\end{lemma}
\begin{proof}
	By a Fa\`a di Bruno type identity, we can compute that
	\begin{align*}
		\p_x^k M_{\uns(x)}
		=
		\sum_{\substack{(m_1,\ldots,m_k)\in\mathbb{Z}_{\ge0}^k \\
				\sum_{j=1}^k j m_j=k}}
		\na^{m_1 + \cdots + m_k}M_{\uns}: c_{m_1,\ldots,m_k}\prod_{j=1}^k\left(\p_x^j\uns\right)^{\otimes m_j}
	\end{align*}
	where the product operator is defined in the sense of tensor products, contracted with some tensor-valued coefficients $c_{m_1,\ldots,m_k}$ -- this formula can be derived by induction on $k\ge1$. We can then collect terms, and using the bounds on the shock profile $\uns$ in Lemma \ref{ABN24_Prop_B.2} we can bound
	\begin{align*}
		\norm{\p_x^k M_{\uns(x)}}_{\mathsf{H}^2_2}
		&\le C_k\norm{M_{\bar{\mathfrak{u}}}}_{C^k_{\bar{\mathfrak{u}}}(B(\uref,\epsilon_0),\mathsf{H}^2_2)}\epsilon^{2m_1 + \cdots + k m_{k-1}+(k+1) m_k}
		\\
		&\le
		C_k \epsilon^{k+1}
	\end{align*}
	where we have used smoothness of the Maxwellian manifold from Lemma \ref{Maxwellian_manifold_smoothness}. Comparing the $\mathsf{H}^2_2$ and $\HH^1_\kappa$ norms gives the result.
\end{proof}

We can now prove the proposition.

\begin{proof}[Proof of Proposition \ref{high_frequency_estimate}]
	We differentiate \eqref{Section5_equation} to get
	\begin{align*}
		(\xi_1-\sfs)\p_x^{k+1} f - L_{\uns,\kappa}f = z + r_k
	\end{align*}
	with a purely microscopic remainder term
	\begin{align*}
		r_k = \sum_{j=0}^{k-1}\binom{k}{j}
		\Big[ Q_\kappa\!\left(\p_x^{k-j}M_{\uns},\p_x^j f\right)
		+
		Q_\kappa\!\left(\p_x^j f, \p_x^{k-j}M_{\uns} \right)
		\Big].
	\end{align*}
	We can then apply the energy estimate in Proposition \ref{main_energy_estimate} to get the bound
	\begin{align*}
		\norm{\p_x^{k+1} f}_{L^2_\epsilon\HH_\kappa^1} + \norm{\p_x^k \mathfrak{v}}_{L^2_\epsilon\HH_\kappa^1}
		&\le C\left(
		\norm{\p_x^k z,\p_x^{k+1}z}_{L^2_\epsilon\HH_\kappa^{-1}}
		+\norm{(r_k,\p_x r_k)}_{L^2_\epsilon\HH^{-1}_\kappa}
		+
		\epsilon \norm{\mathfrak{u}}_{L^2_\epsilon}\right)
	\end{align*}
	for a uniform constant $C>0$, and we can then bound the remainder term as
	\begin{align*}
		\norm{(r_k,\p_x r_k)}_{L^2_\epsilon\HH^{-1}_\kappa}
		&
		\le C_k\sum_{\substack{j+\ell\in\set{k,k+1} \\[1.5pt] j\ge0,\,\ell\ge 1}}
		{\Big\|  Q_\kappa\!\left(\p_x^\ell M_{\uns},\p_x^j f\right)
			+
			Q_\kappa\!\left(\p_x^j f, \p_x^\ell M_{\uns} \right)
			\Big\|}_{L^2_\epsilon\HH^{-1}_\kappa}
		\\
		&\qquad
		\le
		C_k\sum_{\substack{j+\ell\in\set{k,k+1} \\[1.5pt] j\ge0,\,\ell\ge 1}} C\epsilon^{\ell+1}\norm{\p_x^j f}_{L^2_\epsilon\HH^{1}_\kappa}
		\\
		&\qquad
		\le
		C_k\epsilon^k\left( \norm{\p_x f}_{H^{k-1}_\epsilon\HH^1_\kappa}
		+
		\epsilon\norm{f}_{L^2_\epsilon\HH^{1}_\kappa}
		\right)
		\\
		&\qquad
		\le
		C_k\epsilon^k\left( \norm{\p_x f}_{H^{k-1}_\epsilon\HH^1_\kappa}
		+
		\epsilon\norm{\mathfrak{v}}_{H^{k-1}_\epsilon\HH^{1}_\kappa}
		+
		\epsilon\norm{\mathfrak{u}}_{L^2_\epsilon\HH^1_\kappa}
		\right),
	\end{align*}
	where we have used Lemma \ref{Maxwellian_shock_smoothness} to bound the $\p_x^\ell M_{\uns}$ terms, and this proves the claim.
\end{proof}

\section{Macroscopic stability estimates}
\label{Macro_section}

To close the energy estimates in Proposition \ref{main_energy_estimate} and Proposition \ref{high_frequency_estimate}, we need to estimate the macroscopic remainder term $\epsilon\norm{\mathfrak{u}}$ at lowest order. We do this by a linearized Chapman-Enskog procedure, and by using a stability result for viscous shocks, and close the following stability estimate.

\begin{proposition}
	\label{macro_estimate_proposition}
	Let $\epsilon\in(0,\epsilon_0)$ and $\kappa\in [0,\kappa_0]$, for $\epsilon_0,\kappa_0>0$ sufficiently small. Then there exists a unit vector $\ell_\epsilon\in\UU$, such that for any functions $z\in H^k_\epsilon\VV^{-1}_\kappa$ and $f\in H^k_\epsilon\HH^1_\kappa$, and $d\in\RR$ solving the problem
	\begin{align}
		\label{Prop_6.2_equation}
		\ca L_{\epsilon,\kappa} f = z,\qquad \ell_{\epsilon}\cdot\PP f(0)=d,
	\end{align}
	where $k\ge 2$, we have the uniform estimate
	\begin{align*}
		\norm{f}_{H^k_\epsilon\HH^1_\kappa}\le C_k\left( \frac{1}{\epsilon}\norm{z}_{H^k_\epsilon\VV^{-1}_\kappa} + \abs d\right).
	\end{align*}
	The unit vector $\ell_\epsilon\in\UU$ can be taken to vary continuously in $\epsilon\in(0,\epsilon_0)$, and satisfies the estimate $\abs{\ell_\epsilon\cdot\uns'(0)}\ge c_0\epsilon^2$ for a uniform constant $c_0>0$.
\end{proposition}

The proof of Proposition \ref{macro_estimate_proposition} ultimately relies on the following fluid stability estimate, which we prove by rewriting the perturbation problem as a non-degenerate ODE, and computing explicit eigenvalue asymptotics.

\begin{lemma}
	\label{Lemma_6.1}
	For $\epsilon\in(0,\epsilon_0)$ and $\kappa\in[0,\kappa_0]$ for $\epsilon_0>0$ sufficiently small, and for any $d\in\RR$, there exists a vector $\ell_{\epsilon}\in\UU$ such that the problem
	\begin{align}
		\label{Lemma_6.1_Problem}
		(B_\kappa(\uns)\p_x + \sfs I - J'(\uns))\mathfrak{u} = \mathfrak{h},\qquad \mathfrak{u}(0)\cdot\ell_{\epsilon}=d
	\end{align}
	admits a unique solution $\mathfrak{u}\in L^2(\RR)$ for any $\mathfrak{h}\in L^2(\RR;\UU)\cap C^0(\RR;\UU)$ and $d\in\RR$. This solution satisfies the uniform estimate
	\begin{align}
		\label{Lemma_6.1_first_estimate}
		\norm{\mathfrak{u} }_{L^2_\epsilon}
		\le
		C\left( 
		\frac{1}{\epsilon} \norm{\mathfrak{h}}_{L^2_\epsilon} + \abs d \right)
	\end{align}
	for $C$ independent of $\epsilon$ and $\kappa$. Furthermore, the vector $\ell_{\epsilon}$ can be chosen to depend continuously on $\epsilon$, and such that $\abs{\ell_{\epsilon}}=1$ and
	\begin{align*}
		\abs{\ell_{\epsilon}\cdot \p_x\bar{\mathfrak{u}}_{\mrm{NS}}(0)}\ge c_0\epsilon^2
	\end{align*}
	holds for $\epsilon\in(0,\epsilon_0]$.
\end{lemma}

This was proved in \cite{MetivierZumbrun09} and \cite{MetivierZumbrun09_semilinear} using abstract results from \cite{MasciaZumbrun03}, and later for general dissipative systems of Kawashima type \cite{AlbrittonBedrossianNovack24}. Here we present an elementary self-contained proof for our specific problem, and derive quantitative estimates.

\begin{proof}
	The proof relies on an eigenvalue analysis of the equation, however this is complicated by the fact that the diffusion $B_\kappa$ is degenerate, so we must first introduce an appropriate decomposition of the equation. Expressing 
	\begin{align*}
		\mathsf{f}_0' = \begin{pmatrix}1 & 0 \\ \mathsf v & A\end{pmatrix},
		\qquad
		\left( 	\mathsf{f}_0' \right)^{-1} = \begin{pmatrix}1 & 0 \\ -A^{-1}\mathsf v & A^{-1}\end{pmatrix}
	\end{align*}
	where $\mathsf{f}_0$ was defined in \eqref{hydrodynamic_quantities_definition} and writing
	\begin{align*}
		\mathsf{b}_0 = \begin{pmatrix}0 & 0 \\ 0 & b_0\end{pmatrix},
	\end{align*}
	for the matrix derived in \eqref{diffusion_matrix_definition}, we define the associated matrices $a = \mathsf{f}_1'(\mathsf{f}_0')^{-1}$ and 
	\begin{align*}
		b = \mathsf{b}_0 (\mathsf{f}_0')^{-1}
		=
		\begin{pmatrix} 0 & 0 \\ b_{21} & b_{22}\end{pmatrix}=
		\begin{pmatrix} 0 & 0 \\ -b_0 A^{-1}\mathsf v& b_0 A^{-1}\end{pmatrix}
	\end{align*}
	such that \eqref{Lemma_6.1_Problem} can be expressed as
	\begin{align}
		\label{rewritten_Lemma_6.1_Problem}
		\big[	b(\uns)\p_x + (\sfs - a(\uns))\big]w = h,\qquad \ell_\epsilon\cdot \mathfrak{u}=d
	\end{align}
	where we have used a change of variables $w = I[\mathfrak{u}]$ and $h = I[\mathfrak{h}]$ to conservative coordinates as in \eqref{definition_of_PP}, and where we have temporarily suppressed the dependence of the diffusion $b$ on $\kappa$. Computing the inverse
	\begin{align*}
		\mathsf f_0'(\rho,\bfu,T)^{-1}
		=
		\begin{pmatrix}
			1 & 0 & 0 \\
			-\dfrac{\bfu}{\rho} & \dfrac1\rho & 0 \\
			\dfrac{\abs{\bfu}^2}{3\rho} - \dfrac{T}{\rho}
			&
			-\dfrac{2\bfu}{3\rho}
			&
			\dfrac{2}{3\rho}
		\end{pmatrix},
	\end{align*}
	we can then compute the matrices
	\begin{align*}
		a
		=
		\begin{pmatrix}
			0 & 1 & 0
			\\[0.5em]
			-\dfrac23 \bfu^2 & \dfrac43 \bfu & \dfrac23
			\\[0.5em]
			-\dfrac16 \bfu^3 - \dfrac52 \bfu T
			&
			-\dfrac16\bfu^2 + \dfrac52 T 
			&
			\dfrac53 \bfu
		\end{pmatrix}
	\end{align*}
	and
	\begin{align*}
		b
		=
		\begin{pmatrix}
			0 & 0 & 0
			\\[0.1em]
			-\mu_\kappa\dfrac{\bfu}{\rho} & \mu_\kappa\dfrac1\rho & 0
			\\[1em]
			\left(\dfrac{\varkappa_\kappa}{3}-\mu_\kappa \right)\dfrac{\bfu^2}{\rho} - \varkappa_\kappa\dfrac{T}{\rho}
			&
			\left( \mu_\kappa - \dfrac{2\varkappa_\kappa}{3}\right)\dfrac{\bfu}{\rho}
			&
			\dfrac{2\varkappa_\kappa}{3}\dfrac{1}{\rho}
		\end{pmatrix}
	\end{align*}
	and we note that the ODE \eqref{rewritten_Lemma_6.1_Problem} is degenerate in the density coordinate of $u$. We then define the transform
	\begin{align*}
		P(\bar{V}) = \begin{pmatrix} 1 & 0 \\ b_{22}^{-1}b_{21} & I \end{pmatrix}
		= \begin{pmatrix} 1 & 0 \\ - \mathsf v & I\end{pmatrix}
		= \begin{pmatrix} 
			1 & 0 & 0 \\
			-\bar{\bfu} & 1 & 0
			\\ 
			-\dfrac12 \bar{\bfu}^2 - \dfrac32\bar{T}
			& 0 & 1
		\end{pmatrix}
	\end{align*}
	for general hydrodynamic states $\bar{V} = (\bar\rho,\bar{\bfu},\bar T)$, such that 
	\begin{align*}
		b P^{-1} = \begin{pmatrix} 0 & 0 \\ 0 & b_{22}\end{pmatrix},
		\qquad
		b_{22}
		=
		\begin{pmatrix}  \dfrac{\mu_\kappa}{\bar\rho} & 0 
			\\[1em]
			\left(\mu_\kappa - \dfrac{2\varkappa_\kappa}{3}\right)
			\dfrac{\bar{\bfu}}{\bar\rho}
			&
			\dfrac{2\varkappa_\kappa}{3}\dfrac{1}{\bar\rho},
		\end{pmatrix}
	\end{align*}
	and define the transformed variables 
	\begin{align*}
		\tilde{w} = \begin{pmatrix} \tilde{w}_1\\\tilde{w}_2\end{pmatrix}
		=
		P(\bar{V}) w = \begin{pmatrix} w_1 \\ -\mathsf v w_1+w_2\end{pmatrix}
		=
		\begin{pmatrix}
			\rho \\ m-\bar{\bfu}\rho 
			\\
			E - \frac12(\bar{\bfu}^2 + 3\bar{T})\rho,
		\end{pmatrix}
	\end{align*}
	the new source term $\tilde h = Ph$, and the transformed matrix $\tilde a = aP$, which we can compute as
	\begin{align*}
		\tilde a(\bar{V}) =
		\begin{pmatrix}\tilde a_{11} & \tilde a_{12} \\ \tilde a_{21} & \tilde a_{22}
		\end{pmatrix}
		=
		\begin{pmatrix}
			\bar{\bfu} & 1 & 0
			\\[0.5em]
			\bar{\bfu}^2 + \bar{T} & \dfrac43 \bar{\bfu} & \dfrac23
			\\[1em]
			\dfrac12\bar{\bfu}^3 + \dfrac52 \bar{T}\bar{\bfu}
			&
			-\dfrac16\bar{\bfu}^2 + \dfrac52 \bar{T}
			&
			\dfrac53 \bar{\bfu}
		\end{pmatrix}.
	\end{align*}
	Evaluating $\bar V$ at $\bar{V}_{\mrm{NS}}$, 
	we can then rewrite \eqref{Lemma_6.1_Problem} as
	\begin{align}
		(\sfs - \tilde a_{11})\tilde{w}_1 - \tilde a_{12}\tilde{w}_2 &= \tilde h_1 + O(\epsilon^2\abs{w}),
		\notag
		\\
		\label{non-degenerate_macro_equation}
		b_{22}\p_x\tilde{w}_2 + (\sfs \hspace{1pt}\mathsf{v}- \tilde a_{21})\tilde{w}_1 + (\sfs-\tilde a_{22})\tilde{w}_2 
		&=
		\tilde h_2 + O(\epsilon^2\abs{w})
	\end{align}
	where for the error term we have used the fact that $[\p_x,P(\uns)]=O(\epsilon^2)$ by the asymptotics in Lemma \ref{ABN24_Prop_B.2}. We note that for $\epsilon$ sufficiently small.  the term $\sfs-\tilde a_{11}= \sfs-\bar{\bfu}_{\mrm{NS}}$ is invertible,
	by the asymptotics
	\begin{align}
		\begin{split}
			\label{rightward_shock_speed_asymptotics}
			\sfs &= \bfu_- + \sfc_- - \frac\epsilon2 + O(\epsilon^2)
		\end{split}
	\end{align}
	which we derive from \eqref{shock_speed_perturbation},
	and by positivity of the speed of sound $\sfc_->0$. We note that this is exactly the nondegeneracy condition used in \cite{FreistuhlerFriesRohde01} to construct shocks for hyperbolic-parabolic systems of Kawashima type. Solving for $\tilde{w}_1$ and plugging the resulting equation into \eqref{non-degenerate_macro_equation} then gives the ODE
	\begin{align*}
		\p_x\tilde{w}_2 - \mathfrak{m}\tilde{w}_2 = O(\abs{h_1} + \abs{h_2}) + O(\epsilon^2\abs{w}),
	\end{align*}
	where
	\begin{align*}
		\mathfrak{m} = b_{22}^{-1}\left[
		(-\sfs \hspace{1pt}\mathsf{v}+\tilde a_{21})(\sfs-\tilde{a}_{11})^{-1}\tilde{a}_{12}
		-
		\sfs I + \tilde{a}_{22}
		\right],
	\end{align*}
	which we will show has a large eigenvalue and an $O(\epsilon)$ eigenvalue that crosses the origin. We explicitly compute
	\begin{align*}
		\mathfrak{m}
		=
		b_{22}^{-1}\mathfrak{m}_0
		=
		\begin{pmatrix}
			\dfrac{1}{\mu_\kappa} \rho & 0 \\[0.7em]
			\left(\dfrac{1}{\mu_\kappa} - \dfrac{3}{2\varkappa_\kappa}\right)\rho\bfu 
			& 
			\dfrac{3}{2\varkappa_\kappa}\rho
		\end{pmatrix}
		\begin{pmatrix}
			\dfrac{\bar{T}}{\sfs - \bar{\bfu}} + \dfrac{\bar{\bfu}}{3} -\sfs 
			&
			\dfrac23
			\\[1em]
			\dfrac{\sfs \bar{T}}{\sfs - \bar{\bfu}} - \dfrac{2 \bar{\bfu}^2}{3}
			&
			\dfrac{5 \bar{\bfu}}{3}-\sfs 
		\end{pmatrix}
	\end{align*}
	for which we can compute
	\begin{align*}
		\det \mathfrak{m} = \frac{3\bar\rho^2}{2\mu_\kappa\varkappa_\kappa} 
		\left( (\sfs - \bar\bfu)^2 - \bar\sfc^2\right)
	\end{align*}
	and
	\begin{align*}
		\operatorname{tr} \mathfrak{m} &= -(\sfs - \bar\bfu)\bar\rho\left[ \frac{2}{5\mu_\kappa} + \frac{3}{2\varkappa_\kappa}\right]
		+
		\frac{3\bar\rho}{5\mu_\kappa(\sfs - \bar\bfu)}
		\left[ (\sfs - \bar\bfu)^2 - \bar\sfc^2\right].
	\end{align*}
	We define the small parameter
	\begin{align*}
		\hat\epsilon= \big(\sfs - \bar\bfu\big)^2 - \bar\sfc^2
	\end{align*}
	where $\hat\epsilon>0$ for supersonic flow and $\hat\epsilon<0$ for subsonic flow, and when $\bar{\mathfrak{u}} = \uns$, using \eqref{rightward_shock_speed_asymptotics} and \eqref{Rankine-Hugoniot_parametrization} we can compute $\hat\epsilon(\pm\infty) = \pm \sfc_-\epsilon + O(\epsilon^2)$, and the uniform asymptotics
	\begin{align}
		\label{hatepsilon_bounds}
		\abs{\hat\epsilon(\pm\infty) -\hat\epsilon(\pm x)}\le C_0\epsilon\, e^{-\epsilon\theta\abs x},\qquad x\ge 0.
	\end{align}
	using the fluid shock bounds in Lemma \ref{ABN24_Prop_B.2}. We can then compute that the matrix $\mathfrak{m}$ has eigenvalues $\lambda_0$ and $\lambda_-$ with asymptotics
	\begin{align*}
		\lambda_{0}
		=
		-\frac{15\bar\rho}{(15\mu_\kappa+4\varkappa_\kappa)(\sfs-\bar\bfu)} \hat\epsilon + O(\hat\epsilon^2)
	\end{align*}
	and
	\begin{align*}
		\lambda_{-}
		=
		-(\sfs - \bar\bfu)\bar\rho\left[ \frac{2}{5\mu_\kappa} + \frac{3}{2\varkappa_\kappa}\right]
		+
		O(\hat\epsilon).
	\end{align*}
	For $\epsilon$ and $\kappa$ sufficiently small, by smoothness of the coefficients $\mu_\kappa,\varkappa_\kappa$ from Corollary \ref{smoothness_of_the_diffusion_coefficients}, we have that the eigenvalues $\lambda_0,\lambda_-$ are distinct. Therefore we can diagonalize $\mathfrak{m}$, and construct a smoothly varying matrix $\omega$ such that we have
	\begin{align*}
		p:=\omega^{-1} \mathfrak{m} \omega = \begin{pmatrix}\lambda_{0} & 0 \\ 0 & \lambda_{-}\end{pmatrix},
	\end{align*}
	and writing $\tilde{w}_2 = \omega z$ gives the equation
	\begin{align*}
		\p_x z - p z = O(\abs{h_1}+\abs{h_2}) + O(\epsilon^2\abs{u})=: g,
	\end{align*}
	where we have used that $[\omega(\uns),\p_x] = O(\epsilon^2)$, again by the asymptotics in Lemma \ref{ABN24_Prop_B.2}. We can then solve for the stable component by
	\begin{align*}
		z_2(x) = \int_{-\infty}^x\! e^{\int_y^x\! \lambda_{-}(w)\dd w}g_2(y)\dd y
	\end{align*}
	and using that $\lambda_-(x)\le -c_0<0$ for some $c_0$, we have the bounds
	\begin{align*}
		\norm{z_2}_{L^2_\epsilon}\le C\norm{g_2}_{L^2_\epsilon}
	\end{align*}
	and
	\begin{align*}
		\abs{z_2(0)}\le \int_{-\infty}^0\! e^{c_0 y}\abs{g_2(y)}\dd y\le C\norm{g_2}_{L^2}
		= C\sqrt\epsilon\norm{g_2}_{L^2_\epsilon}
	\end{align*}
	for a constant $C$ independent of $\epsilon,\kappa$. We now impose the boundary condition
	\begin{align*}
		\alpha z_1(x) + \beta z_2(x) = d
	\end{align*}
	for some $\alpha\ne 0$ and $\beta\in\RR$, giving the solution
	\begin{align*}
		z_1(x) = \frac{ d - \beta z_2(0)}{\alpha} e^{\int_0^x\!\lambda_0(y)\dd y}
		+
		\int_0^x\! e^{\int_y^x\! \lambda_{0}(w)\dd w}g_1(y)\dd y
	\end{align*}
	for $x\ge 0$ and
	\begin{align*}
		z_1(x) = \frac{ d - \beta z_2(0)}{\alpha} e^{-\int_x^0\!\lambda_0(y)\dd y}
		-
		\int_x^0\! e^{-\int_x^y\! \lambda_{0}(w)\dd w}g_1(y)\dd y
	\end{align*}
	for $x\le 0$ but by applying \eqref{hatepsilon_bounds}, we get the estimate
	\begin{align*}
		\norm{z}_{L^2_\epsilon}\le C\left( \frac1\epsilon\norm{g}_{L^2_\epsilon} + \abs d\right)
	\end{align*}
	for a constant $C$ depending on $\beta$ and $\alpha^{-1}$. But since the conjugation matrix $\omega$ is uniformly continuous in small $\epsilon$ and $\kappa$, we can find a $2$-vector $\tilde\ell$ independent of $\epsilon$ and $\kappa$ such that $\omega(0)^*\tilde\ell$ has first component uniformly bounded away from $0$ in $\epsilon$ and $\kappa$, and setting 
	\begin{align*}
		\ell_\epsilon = c_\epsilon \left( I|_{\UU}\right)^*P(\uns(0))^*\begin{pmatrix}0 \\ \tilde\ell_\epsilon\end{pmatrix}, 
	\end{align*}
	up to a normalization constant $c_\epsilon$ such that $\abs{\ell_\epsilon}=1$ and using the change of variables formula $w = I[\mathfrak{u}]$, we see that the estimate \eqref{Lemma_6.1_first_estimate} holds, and $\ell_\epsilon$ varies continuously in $\epsilon$.
	To prove the last part of the lemma, commuting the derivative in \eqref{traveling_wave_CNSE} when $\kappa=0$ gives
	\begin{align*}
		(B(\uns)\p_x + \sfs I - J'(\uns))\p_x\uns = -\p_x B(\uns)\p_x\uns,
	\end{align*}
	but applying \eqref{Lemma_6.1_first_estimate} and the asymptotics in Lemma \ref{ABN24_Prop_B.2} proves the claim for $\epsilon$ sufficiently small.
\end{proof}

To prove Proposition \ref{macro_estimate_proposition}, we will need to control excess moments in the linearized Chapman-Enskog procedure, for which we use Lemma \ref{HTT20_Lemma}.

\begin{proof}[Proof of Proposition \ref{macro_estimate_proposition}] To compensate for the macroscopic error term $\epsilon\norm{w}_{L^2_\epsilon}$, we use a linearized Chapman-Enskog procedure, as in Section 3. Taking uncentred microscopic parts of \eqref{Prop_6.2_equation}, we get the equation
	\begin{align*}
		(I-\PP_{\uns})(\xi_1-\sfs)\p_x f - z = L_{\uns,\kappa}f = L_{\uns,\kappa}\tilde{\mathfrak{v}}
	\end{align*}
	for the uncentred microscopic part $\tilde{\mathfrak{v}} = (I-\PP_{\uns})f$, so we can compute
	\begin{align*}
		\tilde{\mathfrak{v}} &= L_{\uns,\kappa}^{-1}\Big(
		(I-\PP_{\uns})(\xi_1-\sfs)\p_x f - z
		\Big)
		\\
		&= L_{\uns,\kappa}^{-1}
		\Big(
		(I-\PP_{\uns})(\xi_1-\sfs)\p_x(\dd M_{\uns}\mathfrak{u}+\tilde{\mathfrak{v}}) -z
		\Big)
		\\
		&= b^\perp[\uns]\p_x \mathfrak{u} + r
	\end{align*}
	with the microscopic remainder term
	\begin{align*}
		r
		=
		L_{\uns,\kappa}^{-1}
		\Big(
		(I - \PP_{\uns}) (\xi_1-\sfs) 
		\Big[
		\mrm d^2 M_{\uns}(\p_x\uns,\mathfrak{u}) + \xi_1\p_x\tilde{\mathfrak{v}} \Big] - z
		\Big).
	\end{align*}
	
	Taking macroscopic parts of \eqref{Prop_6.2_equation} then gives that
	\begin{align*}
		\p_x\PP(\xi_1-\sfs) f=0,
	\end{align*}
	but since $\PP(\xi_1-\sfs)$ is a bounded operator on $\HH^1_\kappa\subset\HH^0$ and since we have $f\in L^2_\epsilon\HH^1_\kappa$ by assumption, we can integrate this equation in $x$ and compute that
	\begin{align}
		\begin{split}
			\label{linearized_Chapman_Enskog}
			0 &= \PP(\xi_1-\sfs)f = \PP(\xi_1-\sfs)\left( \dd M_{\uns}\mathfrak{u}+\tilde{\mathfrak{v}}\right)
			\\
			&
			= \big[J'(\uns)-\sfs\big]\mathfrak{u} - B_\kappa(\uns)\p_x \mathfrak{u} - \PP(\xi_1 r),
		\end{split}
	\end{align}
	and by using Lemma \ref{Lemma_6.1}, it remains to bound the remainder term 
	\begin{align*}
		\PP(\xi_1 r)
		&=
		\PP\left[\xi_1 L_{\uns,\kappa}^{-1}
		\Big(
		(I - \PP_{\uns}) (\xi_1-\sfs) 
		\mrm d^2 M_{\uns}(\p_x\uns,\mathfrak{u})
		\Big)
		\right]
		\\
		&\qquad
		+
		\PP\left[ \xi_1 L_{\uns,\kappa}^{-1}\Big( (I - \PP_{\uns}) (\xi_1-\sfs) \p_x\tilde{\mathfrak{v}}\Big)\right]
		- \PP(\xi_1 L_{\uns,\kappa}^{-1} z)
		\\
		&= I_1 + I_2 - I_3.
	\end{align*}
	We can bound the first term as
	\begin{align*}
		\norm{I_1}
		&\le C\norm{L_{\uns,\kappa}^{-1}
			\Big(
			(I - \PP_{\uns}) (\xi_1-\sfs) 
			\mrm d^2 M_{\uns}(\p_x\uns,\mathfrak{u})
			\Big)}_{\HH^1_\kappa}
		\\
		&\le
		C \norm{(I - \PP_{\uns}) (\xi_1-\sfs) 
			\mrm d^2 M_{\uns}(\p_x\uns,\mathfrak{u})
		}_{\HH^{-1}_\kappa}
		\\
		&\le
		C\norm{\mrm d^2 M_{\uns}}_{\mathsf{H}^1_1}\norm{\p_x\uns}\norm{\mathfrak{u}}
		\le C\epsilon^2\norm{\mathfrak{u}}
	\end{align*}
	for a uniform constant $C>0$, where we have used uniform invertibility of $L_{\uns(x),\kappa}^{-1}$ from Lemma \ref{uniform_invertibility_lemma}, the Maxwellian bounds from Lemma \ref{Maxwellian_manifold_smoothness}, and the viscous shock asymptotics from Lemma \ref{ABN24_Prop_B.2}. 
	
	To bound the second term, we require a uniform bound on $\norm{(\xi_1-\sfs)\p_x\tilde{\mathfrak{v}}}_{\HH^{-1}_\kappa}$. While $(\xi_1-\sfs)\p_x\tilde{\mathfrak{v}}\in\HH^0$ is guaranteed by Lemma \ref{HH1kappa_Sobolev_bound} for $\kappa>0$, this comes with a constant that is unbounded as $\kappa\to0$, 
	so we instead use the extension of semigroup estimate in Lemma \ref{HTT20_Lemma} to get the bound
	\begin{align*}
		\norm{I_2}
		&\le
		\norm{ L_{\uns,\kappa}^{-1}\Big( (I - \PP_{\uns}) (\xi_1-\sfs) \p_x\tilde{\mathfrak{v}}\Big)}_{L^2(\br\xi^4)}
		\\
		&\le C\norm{ (I - \PP_{\uns}) (\xi_1-\sfs) \p_x\tilde{\mathfrak{v}} }_{L^2(\br\xi^4)}
		\\
		&\le
		C
		\norm{\p_x\tilde{\mathfrak{v}}}_{L^2(\br\xi^5)}
		\le C\norm{\p_x\tilde{\mathfrak{v}}}_{\HH^1_\kappa}.
	\end{align*}
	Finally, we can bound the last term as
	\begin{align*}
		\norm{I_3} \le C\norm{L_{\uns,\kappa}^{-1}z}_{\HH^{1}_\kappa}
		\le C\norm{z}_{\HH^{-1}_\kappa}
	\end{align*}
	by uniform invertibility in Lemma \ref{uniform_invertibility_lemma}, and applying Lemma \ref{Lemma_6.1} to \eqref{linearized_Chapman_Enskog} \eqref{Prop_6.2_equation} gives the estimate
	\begin{align*}
		\norm{\mathfrak{u}}_{L^2_\epsilon}
		\le\frac{C}{\epsilon}
		\left( 
		\epsilon^2\norm{\mathfrak{u}}_{L^2_\epsilon} +
		\norm{\p_x\tilde{\mathfrak{v}}}_{L^2_\epsilon\HH^{1}_\kappa} +
		\norm{z}_{L^2_\epsilon\HH^{-1}_\kappa}
		\right) + C\abs{d},
	\end{align*}
	which for $\epsilon>0$ and $\kappa\ge 0$ sufficiently small, we 
	\begin{align*}
		\norm{\mathfrak{u}}_{L^2_\epsilon}
		\le\frac{C}{\epsilon}
		\left( 
		\norm{\p_x\tilde{\mathfrak{v}}}_{L^2_\epsilon\HH^{1}_\kappa} +
		\norm{z}_{L^2_\epsilon\HH^{-1}_\kappa}
		\right) + C\abs{d}
	\end{align*}
	for a uniform constant $C>0$. Combining this with the energy estimate in Lemma \ref{main_energy_estimate} and the derivative estimate in Lemma \ref{high_frequency_estimate} for $k=1$, we get the bound
	\begin{align}
		\nonumber
		&
		\norm{\p_x^2 f}_{L^2_\epsilon\HH^1_\kappa}
		+
		\norm{\p_x\mathfrak{v}}_{L^2_\epsilon\HH^1_\kappa}
		+\epsilon\norm{\p_x f}_{L^2_\epsilon\HH^1_\kappa}
		+\epsilon\norm{\mathfrak{v}}_{L^2_\epsilon\HH^1_\kappa}
		+\epsilon^2\norm{\mathfrak{u}}_{L^2_\epsilon}
		\\
		\nonumber
		&\qquad
		\le
		C\left(
		\norm{\p_x^2 z}_{L^2_\epsilon\HH^{-1}_\kappa}
		+
		\norm{\p_x z}_{L^2_\epsilon\HH^{-1}_\kappa}
		+
		\epsilon\norm{z}_{L^2_\epsilon\HH^{-1}_\kappa}
		+
		\epsilon^2\norm{\mathfrak{u}}_{L^2_\epsilon} 
		\right)
		\\
		\label{main_estimate_part1}
		&\qquad
		\le
		C\left(
		\norm{\p_x^2 z}_{L^2_\epsilon\HH^{-1}_\kappa}
		+
		\norm{\p_x z}_{L^2_\epsilon\HH^{-1}_\kappa}
		+
		\epsilon\norm{z}_{L^2_\epsilon\HH^{-1}_\kappa}
		+
		\epsilon\norm{\tilde{\mathfrak{v}}}_{L^2_\epsilon\HH^1_\kappa}  + \epsilon^2\abs{d}
		\right),
	\end{align}
	but by computing
	\begin{align*}
		\p_x\tilde{\mathfrak{v}} = \p_x\left[\mathfrak{v}+ (I-\dd M_{\uns})\mathfrak{u}\right]
		=\p_x\mathfrak{v} + (I-\dd M_{\uns})\p_x\mathfrak{u} - \dd^2 M_{\uns}(\p_x\uns)(\mathfrak{u}),
	\end{align*}
	we can uniformly bound
	\begin{align*}
		\norm{\p_x\tilde{\mathfrak{v}}}_{\HH^1_\kappa}
		\le
		\norm{\p_x\mathfrak{v}}_{\HH^1_\kappa} + C\epsilon\norm{\p_x\mathfrak{u}} + C\epsilon^2\norm{\mathfrak{u}}
	\end{align*}
	using the viscous shock asymptotics in Lemma \ref{ABN24_Prop_B.2} and the Maxwellian derivative estimates in Lemma \ref{Maxwellian_manifold_smoothness}, and plugging this back into \eqref{main_estimate_part1} for $\epsilon>0$ sufficiently small (independently of $\kappa\ge0$) gives the estimate
	\begin{align*}
		&
		\norm{\p_x^2 f}_{L^2_\epsilon\HH^1_\kappa}
		+
		\norm{\p_x\mathfrak{v}}_{L^2_\epsilon\HH^1_\kappa}
		+\epsilon\norm{\p_x f}_{L^2_\epsilon\HH^1_\kappa}
		+\epsilon\norm{\mathfrak{v}}_{L^2_\epsilon\HH^1_\kappa}
		+\epsilon^2\norm{\mathfrak{u}}_{L^2_\epsilon}
		\\
		&\qquad\qquad\qquad
		\le
		C\left(
		\norm{\p_x^2 z}_{L^2_\epsilon\HH^{-1}_\kappa}
		+
		\norm{\p_x z}_{L^2_\epsilon\HH^{-1}_\kappa}
		+
		\epsilon\norm{z}_{L^2_\epsilon\HH^{-1}_\kappa} + \epsilon^2\abs{d}
		\right)
	\end{align*}
	which implies that
	\begin{align*}
		\norm{f}_{H^2_\epsilon\HH^1_\kappa}
		&
		= \epsilon^{-2}\norm{\p_x^2 f}_{L^2_\epsilon\HH^1_\kappa}
		+ \epsilon^{-1}\norm{\p_x f}_{L^2_\epsilon\HH^1_\kappa}
		+ \norm{ f}_{L^2_\epsilon\HH^1_\kappa}
		\\
		&
		\le
		C\left( \sum_{k=0}^2\epsilon^{1-k}\norm{\p_x^k z}_{L^2_\epsilon\HH^{-1}_\kappa}
		+ \abs{d}\right)
		\\
		&
		= C\left( \frac1\epsilon\norm{z}_{H^2_\epsilon\HH^{-1}_\kappa}+\abs d\right),
	\end{align*}
	proving the claim for $k=2$.
	
	To prove the result for $k\ge 3$, we argue by induction. Assuming the claim for the case of $k-1$, and summing with the bound from Proposition \ref{high_frequency_estimate} for $k-1$ gives the estimate
	\begin{align*}
		&
		\epsilon^{-k}\norm{\p_x^k f}_{L^2_\epsilon\HH^1_\kappa}
		+
		\epsilon^{-k}\norm{\p_x^{k-1}v}_{L^2_\epsilon\HH^1_\kappa}
		+
		\lambda\norm{f}_{L^2_\epsilon\HH^1_\kappa}
		\\[0.2em]
		&\qquad\qquad
		\le
		\frac{C_k}{\epsilon}\left(
		\epsilon^{1-k}\norm{\p_x^{k-1}z,\p_x^k z}_{L^2_\epsilon\HH^{-1}_\kappa}
		+\epsilon\norm{f}_{H^{k-1}_\epsilon\HH^1_\kappa}
		\right)
		+
		\lambda C_k\left( \frac1\epsilon\norm{z}_{L^2_\epsilon\HH^{k-1}_\epsilon}+\abs d
		\right)
	\end{align*}
	by using the bound
	\begin{align*}
		\norm{\p_x f}_{H^{k-2}_\epsilon\HH^1_\kappa}
		+\epsilon\norm{\mathfrak{v}}_{H^{k-2}_\epsilon\HH^1_\kappa}
		+\epsilon\norm{\mathfrak{u}}_{L^2_\epsilon}\le C\norm{f}_{H^{k-1}_\epsilon\HH^1_\kappa},
	\end{align*}
	where the parameter $\lambda>0$ may vary, but for $\lambda$ sufficiently large (dependent on $k$ but not on $\epsilon,\kappa$) this gives the bound
	\begin{align*}
		\norm{f}_{H^k_\epsilon\HH^1_\kappa}\le C_k\left(\frac1\epsilon \norm{z}_{H^k_\epsilon\HH^{-1}_\kappa} + \abs d\right),
	\end{align*}
	proving the proposition.
\end{proof}

\section{Galerkin approximation}
\label{Galerkin_scheme}

In Proposition \ref{macro_estimate_proposition} we have shown uniform estimates for sufficiently smooth $f$ and $z$ solving \eqref{Prop_6.2_equation}. In this section we show that for any $z\in H^k_\epsilon\VV^{-1}_\kappa$, we can solve \eqref{Prop_6.2_equation} for a unique $f\in H^k_\epsilon\HH^1_\kappa$. We proceed by a Galerkin scheme, as in \cite{MetivierZumbrun09} \cite{AlbrittonBedrossianNovack24}. More precisely, we have the following proposition.

\begin{proposition}
	\label{Lepsilonkappa_invertibility}
	Let $\kappa>0$ be sufficiently small, and $\epsilon\in(0,\epsilon_0(\kappa))$ for $\epsilon_0(\kappa)>0$ sufficiently small. Then for any $d\in\RR$ and $z\in H^2_\epsilon\VV^{-1}_\kappa$, the problem
	\begin{align*}
		\ca L_{\epsilon,\kappa}f = (\xi_1-\sfs)\p_x f - L_{\uns,\kappa}f = z,\qquad\ell_\epsilon\cdot\PP f(0)=d
	\end{align*}
	admits a unique solution $f\in H^2_\epsilon\HH^1_\kappa$.
\end{proposition}

The proof of this proposition will occupy the rest of this section. We argue by proving similar a priori estimates to Sections \ref{Micro_section} and \ref{Macro_section} to a discretized problem, then explicitly constructing solutions by a stable-unstable dimension counting argument, and using uniform estimates to converge to the infinite-dimensional problem by compactness. We now introduce our notation. Using the discretized subspaces $\HH_N\subset\HH^0$ and associated projections $\Pi_N:\HH^0\to\HH_N$ defined in \eqref{discretized_subspaces}, we define the discretized transport term
\begin{align*}
	A^{(N)} = \Pi_N\xi_1\Pi_N
\end{align*}
and the discretized collision operators
\begin{align*}
	Q^{(N)}_\kappa(g,f) = \Pi_N Q_\kappa(\Pi_N g,\Pi_N f),\qquad\Lref_\kappa^{(N)} = \Pi_N\Lref_\kappa\Pi_N,
\end{align*}
and we introduce the macro-micro decomposition
\begin{align*}
	\HH_N = \UU\oplus\VV_N
\end{align*}
where $\VV_N$ and $\UU$ are orthogonal in the $\HH^0$ norm. It will also be useful to define the dual space
\begin{align*}
	\HH^{-1}_{\kappa,N} = (\HH^{1}_\kappa\cap\HH_N)^*
\end{align*}
whose norm may differ from $\HH^{-1}_\kappa$ for large $N\gg2$, where we have taken duals with respect to the $\HH^0\cap\HH_N$ inner product. We would then like to reprove the uniform estimates in Section \ref{Micro_section} and Section \ref{Macro_section} for the discretized system, but since we no longer expect $Q_\kappa^{(N)}(M_{\bar{\mathfrak{u}}},M_{\bar{\mathfrak{u}}})$ to vanish for $\bar{\mathfrak{u}}\ne\uref$, we must construct a discretized Maxwellian manifold $\ca M_{\epsilon_0,\kappa}^{(N)}\subset\HH_N$ corresponding to equilibria for the discretized collision operator $Q_\kappa^{(N)}$, which we construct by an implicit function theorem argument.

First, we need to show that the discrete projections $\Pi_N$ are well-adapted to the spaces $\mathsf{H}^\ell_\ell$. To do so, we define the quantum harmonic oscillator
\begin{align*}
	\ca L_{\mrm{osc}}:= \Mref^{-1/2}\bar{\ca L}_{\mrm{osc}}\Mref^{1/2},
	\qquad
	\bar{\ca L}_{\mrm{osc}} = -\De_\xi + \frac14\abs{\xi}^2,
\end{align*}
which is known to be a positive self-adjoint operator on $\HH^0$, whose eigenfunctions are the Hermite functions 
\begin{align*}
	\psi_\alpha = \frac{(-1)^{\abs\alpha}}{\sqrt{\alpha!}}\p_\xi^\alpha\Mref,
\end{align*}
which satisfy
\begin{align}
	\label{harmonic_oscillator_spectrum}
	\ca L_{\mrm{osc}}\psi_\alpha = \left(\abs\alpha+\frac32\right)\psi_\alpha,
\end{align}
and form an orthonormal basis for $\HH^0$ for $\alpha\in\mathbb{Z}_{\ge0}^3$ such that $\alpha_2,\alpha_3$ are even (cf. \cite{Szego39}). We note in particular that the set
\begin{align*}
	\set{\psi_\alpha\ \Big|\ \alpha\in\mathbb{Z}_{\ge0}^3,\ \alpha_2\equiv\alpha_3\equiv0\pmod{2},\ \abs{\alpha}\le N}
\end{align*}
forms an orthonormal basis for $\HH_N$ for all $N\ge0$. We note as an immediate corollary that we have the commutation relation
\begin{align}
	\label{commutation_relation}
	[\Pi_N,\ca L_{\mrm{osc}}]=0
\end{align}
for all $N\ge 0$. The basis $(\psi_\alpha)_\alpha$ then allows us to measure regularity in the space $\mathsf{H}^\ell_\ell$ in the following way.

\begin{lemma}
	\label{harmonic_oscillator_norm_lemma}
	For any $\ell\ge 1$ and $f\in\mathsf{H}^{2\ell}_{2\ell}$, we have equivalence of norms
	\begin{align}
		\label{equivalence_of_norms}
		\frac{1}{C_\ell}\norm{\ca L_{\mrm{osc}}^\ell f}_{\HH^0}\le 
		\norm{f}_{\mathsf{H}^{2\ell}_{2\ell}}\le C_\ell\norm{\ca L_{\mrm{osc}}^\ell f}_{\HH^0}.
	\end{align}
	Furthermore, the projection operators $\Pi_N$ are uniformly bounded in $\mathsf{H}^{2\ell}_{2\ell}$ and converge to the identity strongly in $\mathsf{H}^{2\ell}_{2\ell}$.
\end{lemma}
\begin{proof}
	The first bound is immediate by the definition of $\ca L_{\mrm{osc}}$, so we prove the second bound. When $\ell=1$, writing $F = \Mref^{-1/2}f$, we can explicitly compute
	\begin{align*}
		\norm{\ca L_{\mrm{osc}}f}_{\HH^0}^2
		&=\norm{\bar{\ca L}_{\mrm{osc}}F}_{L^2}^2
		\\
		&=
		\br{ \left( -\De+ \frac14\abs\xi^2\right) F,\left( -\De+ \frac14\abs\xi^2\right) F}
		\\
		&=\int_{\RR^3}\!\left( \abs{\De F}^2 + \frac12 \abs\xi^2 \abs{\na F}^2 +
		\frac{1}{16} \abs\xi^4 \abs{F}^2 
		\right)\dd\xi
		-
		\frac32\int_{\RR^3}\!  \abs F^2\dd\xi,
	\end{align*}
	and combining this with the bound
	\begin{align*}
		\norm{\ca L_{\mrm{osc}}f}_{\HH^0}^2\ge\frac94\norm{f}_{\HH^0}^2
	\end{align*}
	from the fact that $\ca L_{\mrm{osc}}$ has spectrum in $[3/2,\infty)$ due to \eqref{harmonic_oscillator_spectrum}, the lemma is proved when $\ell=1$.
	We now prove the general case. We first note the commutation relations
	\begin{align*}
		[\na_\xi,\bar{\ca L}_{\mrm{osc}}] = \frac12\xi,\qquad
		[\xi,\bar{\ca L}_{\mrm{osc}}] = -2\na_\xi,
	\end{align*}
	allowing us to compute the identity
	\begin{align*}
		\xi^\alpha\p_\xi^\beta\bar{\ca L}_{\mrm{osc}} F
		=
		\bar{\ca L}_{\mrm{osc}} \xi^\alpha\p_\xi^\beta F
		+
		\sum_{\abs{\alpha'}+\abs{\beta'}\le\abs\alpha+\abs\beta} 
		c_{\alpha,\beta,\alpha',\beta'}\xi^{\alpha'}\p_\xi^{\beta'} F\in L^2(\RR^3)
	\end{align*}
	for arbitrary multi-indices $\alpha,\beta$ and some coefficients $c_{\alpha,\beta,\alpha',\beta'}$. We now
	assume by induction that the lemma holds for $\ell-1$, and using the norm identity $ \|\cdot\|_{\mathsf{H}^\ell_\ell} =\norm{\Mref^{-1/2}\cdot}_{H^\ell_\ell}$.
	For any $\abs\alpha+\abs\beta\le 2(\ell-1)$, we can then compute that
	\begin{align*}
		\norm{\xi^\alpha\p_\xi^\beta F}_{H^2_2}
		&\lesssim
		\norm{ 	\bar{\ca L}_{\mrm{osc}} \xi^\alpha\p_\xi^\beta F}_{L^2}
		\\
		&\lesssim \norm{ 	\xi^\alpha\p_\xi^\beta\bar{\ca L}_{\mrm{osc}} F}_{L^2}
		+\sum_{\abs{\alpha'}+\abs{\beta'}\le k} \norm{ \xi^{\alpha'}\p_\xi^{\beta'} F }_{L^2}
		\\
		&\lesssim\norm{\bar{\ca L}_{\mrm{osc}}F}_{H^{2(\ell-1)}_{2(\ell-1)}} 
		+ \norm{ F}_{H^{2(\ell-1)}_{2(\ell-1)}}
		\\
		&\lesssim\norm{\bar{\ca L}_{\mrm{osc}}^\ell F}_{L^2} + \norm{\bar{\ca L}_{\mrm{osc}}^{\ell-1}F}_{L^2}
		\lesssim\norm{\bar{\ca L}_{\mrm{osc}}^\ell F}_{L^2}
	\end{align*}
	where in the last line we have used the induction hypothesis and the coercivity estimate on $\bar{\ca L}_{\mrm{osc}}$ in $L^2$, and summing over all $\alpha,\beta$ such that $\abs\alpha+\abs\beta\le 2\ell$ proves the equivalence of norms.
	
	To prove the strong convergence $\Pi_N\to I$ in $\mathsf{H}^{2\ell}_{2\ell}$, this follows immediately by norm equivalence, the commutation relation \eqref{commutation_relation}, and the fact that the spaces $\HH_N$ form a dense subset of $\mathsf{H}^{2\ell}_{2\ell}$.
\end{proof}

For the rest of this section, we will use Lemma \ref{harmonic_oscillator_norm_lemma} as the definition of the $\mathsf{H}^{2\ell}_{2\ell}$ norm, by defining the inner product
\begin{align*}
	\br{f,g}_{\mathsf{H}^{2\ell}_{2\ell}} := \br{\ca L_{\mrm{osc}}^{\ell}f,\ca L_{\mrm{osc}}^\ell g}_{\HH^0}
\end{align*}
for all $f,g\in\mathsf{H}^{2\ell}_{2\ell}$. We then have the following trilinear and coercivity estimates for the discretized collision operators.

\begin{lemma}
	\label{discretized_collisions_lemma}
	Let $f,g,h\in\HH_N$, and let $\kappa\in[0,\kappa_0]$ for $\kappa_0>0$ sufficiently small, then we have the trilinear estimate
	\begin{align*}
		\br{Q_\kappa^{(N)}(g,h),f}_{\mathsf{H}^{2\ell}_{2\ell}}\lesssim C_\ell\norm{f}_{\HH^{1,2\ell}_\kappa}\norm{g}_{\HH^{1,2\ell}_\kappa}
		\norm{h}_{\HH^{1,2\ell}_\kappa}
	\end{align*}
	for all $\ell\ge 0$, and the coercivity estimate
	\begin{align*}
		-\br{\Lref_\kappa^{(N)}f,f}_{\HH^0}\ge \de\norm{(I-\PP)f}_{\HH^1_\kappa}^2
	\end{align*}
	Using the equivalent norm $\widetilde{\mathsf{H}^{2\ell}_{2\ell}}$ for $\mathsf{H}^{2\ell}_{2\ell}$ defined in Lemma \ref{Qkappa_trilinear_estimate}, for any purely microscopic $z\in\HH_N$, we  also have the hypocoercivity estimate
	\begin{align*}
		-\br{\Lref_\kappa^{(N)}z,z}_{\widetilde{\mathsf{H}^{2\ell}_{2\ell}}}\ge\de_\ell\norm{z}_{\HH^{1,2\ell}_\kappa}
	\end{align*}
	for all $\ell\ge 1$, for some constants $C_\ell,\de,\de_\ell>0$ independent of $\kappa$ and $N$. 
\end{lemma}
\begin{proof}
	The trilinear estimate follows immediately by Lemma \ref{Qkappa_trilinear_estimate} and by re-expressing
	\begin{align*}
		\br{Q_\kappa^{(N)}(g,h),f}_{\mathsf{H}^{2\ell}_{2\ell}}
		= \br{\Pi_N Q_\kappa(g,h),f}_{\mathsf{H}^{2\ell}_{2\ell}} = \br{Q_\kappa(g,h),f}_{\mathsf{H}^{2\ell}_{2\ell}}
	\end{align*}
	by orthogonality of the projections $\Pi_N$ in the norm $\mathsf{H}^{2\ell}_{2\ell}$ norm defined in Lemma \ref{harmonic_oscillator_norm_lemma}, and similar arguments hold for coercivity and hypocoercivity of $\Lref_\kappa^{(N)}$. 
\end{proof}

\subsection{Discretizing the Maxwellian manifold}

\begin{lemma}
	\label{discretized_Maxwellians_lemma}
	Let $\epsilon_0>0$ and $\kappa_0>0$ be sufficiently small, then for any $\mathfrak{u}$ such that $\abs{ \mathfrak{u}-\uref}\le\epsilon_0$, there exist discretized Maxwellians $M_{\mathfrak{u},\kappa}^{(N)}\in\HH_N$ 
	satisfying
	\begin{align}
		\label{discretized_Maxwellian_equation}
		Q_\kappa^{(N)}(M_{\mathfrak{u},\kappa}^{(N)},M_{\mathfrak{u},\kappa}^{(N)})=0,\qquad
		\PP M_{\mathfrak{u},\kappa}^{(N)}= u
	\end{align}
	depending continuously on the parameters $M_{\mathfrak{u},\kappa}^{(N)}\in C_{ u,\kappa}^2(B(\uref,\epsilon_0)\times[0,\kappa_0],\mathsf{H}^{12}_{12})$, and furthermore we have uniform convergence
	\begin{align*}
		\norm{M_{\mathfrak{u},\kappa}^{(N)} - M_{\mathfrak{u}}}_{C^2_{ u,\kappa}(B(\uref,\epsilon_0)\times[0,\kappa_0],\,\mathsf{H}^{12}_{12})}\le C_k N^{-k}
	\end{align*}
	for some constant $C_k>0$, for any $k\ge1$. 
	Additionally, we have the estimate
	\begin{align*}
		\norm{ (I -\PP)\dd M_{\mathfrak{u},\kappa}^{(N)}}_{\HH^1_\kappa}\le C\abs{ \mathfrak{u}-\uref}
	\end{align*}
	for a uniform constant $C>0$.
\end{lemma}

To prove this, we need the following lemma.

\begin{lemma}
	\label{approximate_Maxwellian_error}
	Let $\epsilon_0>0$ and $\kappa_0>0$ be sufficiently small, then we have uniform convergence
	\begin{align*}
		\norm{Q_\kappa(\Pi_N M_{\bar{\mathfrak{u}}},\Pi_N M_{\bar{\mathfrak{u}}})}_{C_{\bar{\mathfrak{u}},\kappa}^2(B(\uref,\epsilon_0)\times[0,\kappa_0],\,\mathsf{H}^{12}_{12})}\le C_k N^{-k}
	\end{align*}
	for all $k\ge1$, for some $C_k>0$.
\end{lemma}
\begin{proof}
	To show this, we must first prove uniform convergence
	\begin{align}
		\label{lemma7.3_eq}
		\norm{\Pi_N M_{\bar{\mathfrak{u}}}-M_{\bar{\mathfrak{u}}}}_{C^1_{\bar{\mathfrak{u}}}( B(\uref,\epsilon_0),\mathsf{H}^{12}_{12})}\le C_k N^{-k}
	\end{align}
	for some $C_k>0$. Arguing generally, for any $\ell,k\ge 0$ and $g\in\mathsf{H}^{2(\ell+k)+4}_{2(\ell+k)+4}$, by using the $\mathsf{H}^{2\ell}_{2\ell}$ norm defined in Lemma \ref{harmonic_oscillator_norm_lemma}, the fact that $\Pi_N$ and $\ca L_{\mrm{osc}}$ commute, and the spectral decomposition in \eqref{harmonic_oscillator_spectrum}, we can compute that
	\begin{align}
		\norm{(I - \Pi_N)g}_{\mathsf{H}^{2\ell}_{2\ell}}^2
		&=
		C_\ell\norm{(I-\Pi_N)\ca L_{\mrm{osc}}^\ell g}_{\HH^0}^2
		\nonumber
		\\
		&= C_\ell\sum_{\abs\alpha\ge N+1}\left(\abs\alpha+\frac32\right)^{2\ell}\abs{\br{g,\psi_\alpha}}^2
		\nonumber
		\\
		&= C_\ell \sum_{\abs\alpha\ge N+1}\left(\abs\alpha+\frac32\right)^{-4-2k}\abs{\br{\ca L_{\mrm{osc}}^{\ell+k+2}g,\psi_\alpha}}^2
		\nonumber
		\\
		&\le C_{k,\ell} N^{-1-2k}\norm{g}_{\mathsf{H}^{2(\ell+k)+4}_{2(\ell+k)+4}}^2,
		\label{projections_norm_convergence}
	\end{align}
	but since $M_{\bar{\mathfrak{u}}}\in C^2_{\bar{\mathfrak{u}}}(B(\uref,\epsilon_0),\mathsf{H}^\ell_\ell)$ for all $\ell\ge0$ by Lemma \ref{Maxwellian_manifold_smoothness}, this proves \eqref{lemma7.3_eq}. Using that $Q_\kappa(M_{\bar{\mathfrak{u}}},M_{\bar{\mathfrak{u}}})=0$, we can then derive by bilinear identities that
	\begin{align*}
		Q_\kappa(\Pi_N M_{\bar{\mathfrak{u}}},\Pi_N M_{\bar{\mathfrak{u}}})
		&=
		-L_{\bar{\mathfrak{u}},\kappa}(I-\Pi_N) M_{\bar{\mathfrak{u}}} + Q_\kappa\Big( (I-\Pi_N) M_{\bar{\mathfrak{u}}},(I-\Pi_N) M_{\bar{\mathfrak{u}}}\Big)
	\end{align*}
	and using the uniform bounds proved in Lemma \ref{Qkappa_trilinear_estimate}, combined with the derivative estimates in Lemma \ref{Maxwellian_manifold_smoothness} and the uniform Sobolev bounds in Lemma \ref{HH1kappa_Sobolev_bound} proves uniform convergence in $C^1_{\bar{\mathfrak{u}}}\mathsf{H}^{12}_{12}$. To prove uniform convergence of $\kappa$ derivatives, we note that $Q_\kappa$ is affine in $\kappa$, so the result follows immediately by uniform convergence in $\kappa$.
\end{proof}

\begin{proof}[Proof of Lemma \ref{discretized_Maxwellians_lemma}]
	We construct the approximate Maxwellians by an implicit function argument. Using the macro-micro decomposition $f = (\mathfrak{u},\mathfrak{v})\in\HH_N$, we define the function
	\begin{align*}
		\mathsf{F}(\mathfrak{u},\mathfrak{v})= (\mathfrak{u},Q_\kappa^{(N)}(\mathfrak{u}+\mathfrak{v},\mathfrak{u}+\mathfrak{v}))
	\end{align*}
	which has derivatives
	\begin{align*}
		\mathsf{F}'(\Mref) = (\PP,\Lref_\kappa^{(N)}),
		\qquad
		\mathsf{F}''(\mathfrak{u},\mathfrak{v}) = (0,Q_\kappa^{(N)}(\mathfrak{u}+\mathfrak{v},\mathfrak{u}+\mathfrak{v})),
	\end{align*}
	and $\mathsf{F}'''=0$. By Lemma \ref{discretized_collisions_lemma}, we have the uniform estimate
	\begin{align*}
		\norm{\mathsf{F}'(\Mref)^{-1}\mathsf{F}''(f)}_{\HH^1_\kappa}
		\le C\big\|{Q_\kappa^{(N)}(f,f)}\big\|_{\HH^{-1}_{\kappa,N}}
		\le C\norm{f}_{\HH^1_\kappa}^2
	\end{align*}
	for a constant $C$ independent of $\kappa$ and $N$, therefore we can define the inverse
	\begin{align*}
		\mathsf{F}^{-1}(\mathfrak{u},\mathfrak{v})= (\mathfrak{u},G_\kappa^{(N)}(\mathfrak{u},\mathfrak{v}))
	\end{align*}
	whenever $\norm{f-\Mref}_{\HH^1_\kappa}\le\epsilon_0$ for $\epsilon_0>0$ sufficiently small, independently of $\kappa$ and $N$ (see Chapter XIV, Lemma 1.3 of \cite{Lang93} for estimates on the size of the neighbourhood for the inverse function theorem). For any $\mathfrak{u}$ such that $\abs{\mathfrak{u}-\uref}<\epsilon_0$ we define
	\begin{align*}
		M_{\mathfrak{u},\kappa}^{(N)} = \mathsf{F}^{-1}(\mathfrak{u},0)
	\end{align*}
	which satisfies \eqref{discretized_Maxwellian_equation} by construction.
	
	To prove uniform convergence to $M_{\mathfrak{u}}$, we note that for $g\in\HH_N$ such that $\norm{g-\Mref}_{\mathsf{H}^{12}_{12}}\le\epsilon_0$ sufficiently small, independently of $\kappa,N$, we have the bound
	\begin{align*}
		\norm{\mathsf{F}'(g)f}_{\mathsf{H}^{2\ell}_{2\ell}}\ge\de_\ell\norm{f}_{\mathsf{H}^{2\ell}_{2\ell}}
	\end{align*}
	for all $\ell\ge 1$ by Lemma \ref{Qkappa_trilinear_estimate}, by the uniform Sobolev bound in Lemma \ref{HH1kappa_Sobolev_bound}, and the fact that $\Pi_N$ form orthogonal projections on $\mathsf{H}^{12}_{12}$. To bound the constant second derivative $\mathsf{F}''$, we use Lemma \ref{harmonic_oscillator_norm_lemma} to rewrite the trilinear estimate in Lemma \ref{Qkappa_trilinear_estimate} as
	\begin{align}
		\begin{split}
			\label{high_frequency_Qkappa_bound}
			\br{\ca{L}_{\mrm{osc}}^{\ell-1}Q_\kappa(g,h),
				\ca{L}_{\mrm{osc}}^{\ell+1}f}_{\HH^0}
			&=
			\br{\ca{L}_{\mrm{osc}}^{\ell}Q_\kappa(g,h),
				\ca{L}_{\mrm{osc}}^{\ell}f}_{\HH^0}
			\\
			&
			\le	C_\ell\norm{\ca L_{\mrm{osc}}^{\ell+1} g}_{\HH^0}
			\norm{\ca L_{\mrm{osc}}^{\ell+1} h}_{\HH^0}
			\norm{\ca L_{\mrm{osc}}^{\ell+1} f}_{\HH^0},
		\end{split}
	\end{align}
	for a uniform constant $C_\ell>0$, but by uniform boundedness of the projections $\Pi_N$ on $\mathsf{H}^{2\ell-2}_{2\ell-2}$, this implies
	\begin{align}
		\label{Fprimeprime_bound}
		\norm{\mathsf{F}''(h,f)}_{\mathsf{H}^{2\ell-2}_{2\ell-2}}
		=
		\le C_\ell\norm{h}_{\mathsf{H}^{2\ell+2}_{2\ell+2}}\norm{f}_{\mathsf{H}^{2\ell+2}_{2\ell+2}}
	\end{align}
	holds for all $\ell\ge1$. We can therefore estimate
	\begin{align*}
		\norm{M_{\mathfrak{u},\kappa}^{(N)}-\Pi_N M_{\mathfrak{u}}}_{\mathsf{H}^{12}_{12}}
		&
		=
		\norm{\mathsf{F}^{-1}(\mathfrak{u},0) - \mathsf{F}^{-1}\left(\mathfrak{u},Q_\kappa^{(N)}(M_{\mathfrak{u}},M_{\mathfrak{u}})\right)}_{\mathsf{H}^{12}_{12}}
		\\
		&\le
		\norm{\mathsf{F}^{-1}}_{C^1\left(B_{\mathsf{H}^{12}_{12}}(\Mref,\epsilon_0),\, \mathsf{H}^6_6\right)}
		\norm{Q_\kappa^{(N)}(M_{\mathfrak{u}},M_{\mathfrak{u}})}_{\mathsf{H}^{12}_{12}} \le C_k N^{-k}
	\end{align*}
	where the last bound holds for a uniform constant $C_k>0$, by Lemma \ref{approximate_Maxwellian_error} and boundedness of $\Pi_N$ on $\mathsf{H}^6_6$. Using the derivative estimates in Lemma \ref{approximate_Maxwellian_error} and \eqref{Fprimeprime_bound} then similarly proves
	\begin{align*}
		\norm{\p_{\mathfrak{u}}^\alpha \left( M_{\mathfrak{u},\kappa}^{(N)}-\Pi_N M_{\mathfrak{u}}\right)}_{\mathsf{H}^{12}_{12}}\le C_k N^{-k}
	\end{align*}
	for any $1\le\abs\alpha\le 2$ and for a uniform constant $C_k>0$, for $k$ arbitrary, and together with \eqref{lemma7.3_eq}, this proves all but the last estimate.
	
	For the last estimate, differentiating \eqref{discretized_Maxwellian_equation} at $\mathfrak{u}=\uref$ gives that $$\Lref_\kappa^{(N)}\dd M_{\uref,\kappa}^{(N)}=0,$$ but by the coercivity estimate in Lemma \ref{discretized_collisions_lemma} this implies that $\dd M_{\uref,\kappa}^{(N)}\in\UU$, and so the last bound follows by differentiability of $M_{\bar{\mathfrak{u}},\kappa}^{(N)}$.
\end{proof}

We then define the uncentred linearization
\begin{align*}
	L^{(N)}_{\bar{\mathfrak{u}},\kappa}f = Q_\kappa^{(N)}\!\left (M_{\bar{\mathfrak{u}},\kappa}^{(N)},f\right)
	+
	Q_\kappa^{(N)}\!\left(f,M_{\bar{\mathfrak{u}},\kappa}^{(N)}\right)
\end{align*}
for $\abs{\bar{\mathfrak{u}}-\uref}\le\epsilon_0$ for $\epsilon_0>0$ sufficiently small, and we have the following discrete equivalent to Lemma \ref{uniform_invertibility_lemma}.

\begin{lemma}
	\label{discretized_uniform_invertibility_lemma}
	Let $\epsilon_0>0$ be sufficiently small, and $\kappa_0>0$ fixed. Then for $\kappa\in[0,\kappa_0]$, any $\abs{\bar{\mathfrak{u}}-\uref}\le\epsilon_0$, and any $N\ge 3$, the operator $L_{\bar{\mathfrak{u}},\kappa}^{(N)}|_{\VV_N}$ admits an inverse $\VV_N\to\VV_N$ with uniform estimates
	\begin{align*}
		\norm{\left( L_{\bar{\mathfrak{u}},\kappa}^{(N)}\right)^{-1}z}_{\mathsf{H}^{2\ell}_{2\ell}}
		\le C\norm{z}_{\mathsf{H}^{2\ell}_{2\ell}}
	\end{align*}
	and
	\begin{align*}
		\norm{\left( L_{\bar{\mathfrak{u}},\kappa}^{(N)}\right)^{-1}z}_{\HH^1_\kappa}
		\le C\norm{z}_{\HH^{-1}_\kappa}
	\end{align*}
	for purely microscopic $z\in\HH_N$, for $0\le\ell\le 5$, and for a constant $C$ independent of $\epsilon ,\kappa$ and $N$.
\end{lemma}
\begin{proof}
	From Lemma \ref{discretized_Maxwellians_lemma} and the fact that $M_{\uref,\kappa}^{(N)} = \Mref$, we have that
	\begin{align*}
		\sup_{\kappa\in[0,\kappa_0]}\norm{M_{\bar{\mathfrak{u}},\kappa}^{(N)} - \Mref}_{\mathsf{H}^{12}_{12}}\le C\abs{\bar{\mathfrak{u}} - \uref},
	\end{align*}
	so applying Lemma \ref{discretized_collisions_lemma} gives the estimate
	\begin{align*}
		-\br{ L_{\bar{\mathfrak{u}},\kappa}^{(N)}f,f}_{\widetilde{\mathsf{H}^{2\ell}_{2\ell}}}\ge 
		\left( \de_\ell - C_\ell\epsilon_0\right)\norm{f}_{\mathsf{H}^{2\ell}_{2\ell}}
	\end{align*}
	for $f\in\VV_N$, and taking $\epsilon_0>0$ sufficiently small for all $0\le \ell\le 5$ and applying the Lax-Milgram theorem proves the first estimate. Similarly arguing from Lemma \ref{discretized_collisions_lemma} we have the estimate
	\begin{align}
		\label{uncentred_discretized_coercivity_estimate}
		\norm{ L_{\bar{\mathfrak{u}},\kappa}^{(N)} f}_{\HH^{-1}_\kappa}\norm{f}_{\HH^1_\kappa}
		\ge 
		-\br{ L_{\bar{\mathfrak{u}},\kappa}^{(N)}f,f}_{\HH^0}
		\ge 
		\left( \de - C\epsilon_0\right)\norm{f}_{\HH^1_\kappa}^2,
	\end{align}
	and taking $\epsilon_0>0$ and inverting $L_{\bar{\mathfrak{u}},\kappa}^{(N)}$ then proves the claim.
\end{proof}

As in the infinite-dimensional setting, we will define the uncentred projections $\PP^{(N)}_{\bar{\mathfrak{u}},\kappa} = \dd M_{\bar{\mathfrak{u}},\kappa}^{(N)}\PP$ which satisfy the identities
\begin{align*}
	L^{(N)}_{\bar{\mathfrak{u}},\kappa}f = L^{(N)}_{\bar{\mathfrak{u}},\kappa}\left(I-\PP^{(N)}_{\bar{\mathfrak{u}},\kappa}\right)f
\end{align*}
and $\PP(I-\PP^{(N)}_{\bar{\mathfrak{u}},\kappa}) = 0$, which are proved by differentiating the identities \eqref{discretized_Maxwellian_equation}, and which also gives the useful reformulation of the microscopic projections
\begin{align}
	\label{microscopic_projections_reformulation}
	I-\PP_{\bar{\mathfrak{u}},\kappa}^{(N)} = \left(\left. L_{\bar{\mathfrak{u}},\kappa}^{(N)}\right|_{\VV_N}\right)^{-1} L_{\bar{\mathfrak{u}},\kappa}^{(N)}.
\end{align}
We now introduce the discretized linearized operator
\begin{align*}
	\ca L_{\epsilon,\eta,\kappa}^{(N)} = \left(A^{(N)}-\sfs\right)\p_x - \eta\p_x^2 - L_{\uns,\kappa}^{(N)}
\end{align*}
where we have introduced an artificial viscosity term $\eta\p_x^2$ for $\eta>0$. We solve the problem
\begin{align}
	\label{Galerkin_problem}
	\ca L_{\epsilon,\eta,\kappa}^{(N)}f = z,\qquad\ell_\epsilon\cdot\PP f(0)=d
\end{align}
for $d\in\RR$ and microscopic $z\in\HH_N$. We have the following existence result for this problem.

\begin{proposition}
	\label{Galerkin_proposition}
	Let $\kappa\in(0,\kappa_0),\eta\in(0,\eta_0)$ for $\kappa_0>0$ and $\eta_0>0$ sufficiently small, let $\epsilon\in(0,\epsilon_0(\kappa))$ for $\epsilon_0(\kappa)>0$ sufficiently small, and let $N\ge N_0(\epsilon)$ be sufficiently large. Then for any purely microscopic $z\in H^2_\epsilon\HH_N$ and $d\in\RR$, the problem \eqref{Galerkin_problem} admits a solution $f\in H^2_\epsilon\HH_N$, which satisfies the estimate
	\begin{align}
		\label{Galerkin_macroscopic_estimate}
		\norm{f}_{H^2_\epsilon\HH^1_\kappa}\le C_\kappa\left( \frac1\epsilon \norm{z}_{H^2_\epsilon\HH^{-1}_\kappa}+\abs{d}\right)
	\end{align}
	for a constant $C_\kappa>0$ depending only on $\kappa>0$.
\end{proposition}

\begin{remark}
	By elementary algebraic-geometric arguments, we see that $A^{(N)}-\sfs$ is an injective (and therefore invertible) operator if and only if $\sfs\in\RR$ is not a zero of the Hermite function $\psi_{N+1-2j}$ for any integer $j\in[0,(N+1)/2]$. However, it is known that as $N\to\infty$ the zeros of the Hermite functions $\psi_N$ become dense in $\RR$ (Theorem 4 of \cite{Gawronski87}) and therefore the ODE without viscosity would be degenerate for a countably dense subset of shock strengths $\epsilon\in(0,\epsilon_0)$. To avoid these technicalities, and to simplify the boundary value problem for the ODE, we introduce an artificial viscosity term $\eta\p_x^2$ for a varying parameter $\eta\in(0,\infty)$, which we will eventually take to be small. In particular, stable-unstable dimensionality computations will be much more straightforward in this context.
\end{remark}

\subsection{Discretized stability estimates}

In this subsection, we take $f\in H^2_\epsilon\HH_N$ and microscopic $z\in H^2_\epsilon\HH_N$ satisfying \eqref{Galerkin_problem}, and write $f=(\mathfrak{u},\mathfrak{v})$ in the macro-micro decomposition. As for the problem \eqref{Prop_6.2_equation}, we can prove the following energy estimate.

\begin{lemma}
	\label{Galerkin_energy_estimate} 
	Let $\epsilon\in(0,\epsilon_0]$, $\eta\in[0,\eta_0]$, and $\kappa\in[0,\kappa_0]$ for $\epsilon_0>0$ and $\kappa_0>0$ sufficiently small, independently of $N,\eta$, and $\eta_0>0$ fixed. Then we have the estimate
	\begin{align*}
		\eta\norm{\p_x^2 f}^2_{L^2_\epsilon\HH^0}+\norm{\p_x f}_{L^2_\epsilon\HH^1_\kappa}^2 
		+ \norm{ \mathfrak{v}}_{L^2_\epsilon\HH^1_\kappa}^2
		\le C\left( 
		\norm{z}_{L^2_\epsilon\HH^{-1}_\kappa}^2 + \norm{\p_x z}_{L^2_\epsilon\HH^{-1}_\kappa}^2
		+
		\epsilon^2 \norm{\mathfrak{u}}_{L^2_\epsilon}^2\right)
	\end{align*}
	for a constant $C>0$ independent of $\epsilon,\eta,\kappa$ and $N$.
\end{lemma}

We have the following discrete equivalent to Lemma \ref{Kawashima_compensator_lemma}.
\begin{lemma}
	\label{Galerkin_Kawashima_compensator_lemma}
	Let $\abs{\bar{\mathfrak{u}} - \uref}\le\epsilon_0$ and $\kappa\in[0,\kappa_0]$ for $\epsilon_0>0$ and $\kappa_0>0$ sufficiently small. Then for all $f\in\HH_N$ and $N\ge 3$, we have the bound
	\begin{align*}
		\br{\Big(K A^{(N)}- L_{\bar{\mathfrak{u}},\kappa}^{(N)}\Big)f,f}_{\HH^0}\ge\de_0\norm{f}_{\HH^1_\kappa}.
	\end{align*}
\end{lemma}
\begin{proof}
	Noting that $[K,A^{(N)}]$ is symmetric and satisfies $\Pi_3[K,A^{(N)}]\Pi_3$ by the proof of Lemma \ref{Kawashima_compensator_lemma}, we can apply Lemma \ref{discretized_collisions_lemma} and the bounds on the approximate Maxwellian in Lemma \ref{discretized_Maxwellians_lemma}, and follow the rest of the proof of Lemma \ref{Kawashima_compensator_lemma}, and this proves the claim.
\end{proof}

\begin{proof}[Proof of \ref{Galerkin_energy_estimate} ]
	Proof is analogous to the proof of Lemma \ref{main_energy_estimate}. Using the energy functional defined in \eqref{energy_functional}, we can compute
	\begin{align*}
		E[f] 
		&=
		-\int_{\RR}\!\br{\p_x\Big[ \eta\p_x^2 f+ L_{\bar{\mathfrak{u}}_{\mrm{NS}}(x),\kappa}^{(N)}f\Big],\p_x f}\dd x
		-
		\lambda \int_{\RR}\! \br{\eta\p_x^2 f + L_{\uns,\kappa}^{(N)}f,f}\dd x
		\\
		& \qquad\qquad+
		\int_{\RR}\!\br{ K\Big[(A^{(N)}-\sfs)\p_x-\eta\p_x^2 - L_{\uns,\kappa}^{(N)}\Big] f,\p_xf}\dd x
		\\
		&=
		\int_{\RR}\!\br{ \Big[ K A^{(N)} -L_{\uns,\kappa}^{(N)}\Big]\p_x f,\p_x f}\dd x
		-
		\lambda\int_{\RR}\!\br{L_{\uns,\kappa}^{(N)}f,f}\dd x
		\\
		&\qquad\qquad
		-
		\int_{\RR}\!\br{Q_\kappa^{(N)}\Big(\p_x M_{\uns,\kappa}^{(N)},f\Big)
			+Q_\kappa^{(N)}\Big(f,\p_x M_{\uns,\kappa}^{(N)}\Big),\p_x f}\dd x
		\\
		&\qquad\qquad
		-
		\int_{\RR}\!\br{K L_{\uns,\kappa}^{(N)}f,\p_x f}\dd x
		\\
		&\qquad\qquad
		+ \eta\int_{\RR}\!\norm{\p_x^2 f}_{\HH^0}^2\dd x 
		+ \eta\lambda \int_{\RR}\!\norm{\p_x f}_{\HH^0}^2\dd x
		-
		\eta\int_\RR\!\br{\p_x^2 f,K\p_x f}\dd x.
	\end{align*}
	The last integral can be absorbed by the previous two integrals for $\lambda>0$ sufficiently large, by boundedness of the operator $K:\HH^0\to\HH^0$. We can then use Lemma \ref{Galerkin_Kawashima_compensator_lemma}, the trilinear and coercivity estimates in Lemma \ref{discretized_collisions_lemma}, and differentiability of the approximate Maxwellians from Lemma \ref{discretized_Maxwellians_lemma} to derive the bound
	\begin{align*}
		E[f]
		&\ge
		\de_0\left(\norm{\p_x f}_{L^2\HH^1_\kappa}^2+\lambda\norm{\mathfrak{v}}_{L^2\HH^1_\kappa}^2\right)
		-
		C\epsilon\norm{\mathfrak{u}}_{L^2}
		\left(
		\norm{\p_x f}_{L^2\HH^1_\kappa} + \lambda\norm{\mathfrak{v}}_{L^2\HH^1_\kappa}
		\right)
		\\
		&\qquad\qquad
		+ \eta\int_{\RR}\!\norm{\p_x^2 f}_{\HH^0}^2\dd x 
		+ \eta\lambda \int_{\RR}\!\norm{\p_x f}_{\HH^0}^2\dd x,
	\end{align*}
	for $\de_0$ independent of all parameters, as in the proof of Lemma \ref{main_energy_coercivity}. Using the bound
	\begin{align*}
		E[f]\le\norm{\p_x z}_{L^2\HH^{-1}_\kappa}\norm{\p_x f}_{L^2\HH^1_\kappa}
		+\lambda\norm{z}_{L^2\HH^{-1}_\kappa}\norm{\mathfrak{v}}_{L^2\HH^1_\kappa}
		+C\norm{z}_{L^2\HH^{-1}_\kappa}\norm{\p_x f}_{L^2\HH^1_\kappa}
	\end{align*}
	as in the proof of Lemma \ref{main_energy_estimate}, by absorbing the $\eta\norm{\p_x f}_{\HH^0}$ term by $\norm{\p_x f}_{\HH^1_\kappa}$ by the norm bound in Lemma \ref{HH1kappa_Sobolev_bound} and proceeding as in the infinite-dimensional case, the lemma is proved.
\end{proof}

We combine this with the macroscopic stability arguments to get the following.

\begin{proposition}
	\label{Galerkin_macroscopic_stability}
	Let $\eta\in(0,\eta_0)$ and $\kappa\in(0,\kappa_0)$ for $\kappa_0>0$ and $\eta_0>0$ sufficiently small, and let $\epsilon\in(0,\epsilon_0(\kappa))$ for $\epsilon_0(\kappa)>0$ sufficiently small. Then the estimate
	\begin{align*}
		\norm{f}_{H^2_\epsilon\HH^1_\kappa}\le C_\kappa\left( \frac1\epsilon \norm{z}_{H^2_\epsilon\HH^{-1}_\kappa}+\abs{d}\right)
	\end{align*}
	holds for some constant $C_\kappa>0$ independent of $\epsilon,\eta$ and $N$.
\end{proposition}

\begin{proof}
	As in the proof of Proposition \ref{macro_estimate_proposition}, we take uncentred microscopic parts of \eqref{Galerkin_problem} and get the equation
	\begin{align*}
		\left( I - \PP_{\uns,\kappa}^{(N)}\right)\left( A^{(N)}-\sfs\right)\p_x f
		- \eta\left( I - \PP_{\uns,\kappa}^{(N)}\right)\p_x^2 f-z = L_{\uns,\kappa}^{(N)}f = L_{\uns,\kappa}^{(N)}\tilde{\mathfrak{v}}
	\end{align*}
	for $\tilde{\mathfrak{v}}= \left( I - \PP_{\uns,\kappa}^{(N)}\right)f$. 
	Inverting the operator $L_{\uns,\kappa}^{(N)}$ gives the expression
	\begin{align*}
		\tilde{\mathfrak{v}} 
		&= 
		\left( L_{\uns,\kappa}^{(N)}\right)^{-1}\Big[
		\left( I - \PP_{\uns,\kappa}^{(N)}\right)\left( A^{(N)}-\sfs\right)\p_x \left(\dd M_{\uns,\kappa}^{(N)}\mathfrak{u} + \tilde{\mathfrak{v}}\right) 
		\\
		&\qquad\qquad\qquad\qquad\qquad\qquad\qquad\qquad\qquad
		- \eta\left( I - \PP_{\uns,\kappa}^{(N)}\right)\p_x^2 f - z
		\Big]
		\\
		&=
		b^\perp_{\kappa,N}[\uns]\p_x\mathfrak{u}  + r
	\end{align*}
	for
	\begin{align}
		\label{discretized_Chapman-Enskog_correction}
		b^\perp_{\kappa,N}[\bar{\mathfrak{u}}] = \left( L_{\bar{\mathfrak{u}},\kappa}^{(N)}\right)^{-1}
		\Big[
		\left( I - \PP_{\bar{\mathfrak{u}},\kappa}^{(N)}\right)A^{(N)}\dd M_{\bar{\mathfrak{u}},\kappa}^{(N)}
		\Big]
	\end{align}
	and
	\begin{align*}
		r &= \left( L_{\uns,\kappa}^{(N)}\right)^{-1}\Big[
		\left( I - \PP_{\uns,\kappa}^{(N)}\right)\Pi_N(\xi_1-\sfs)\dd^2 M_{\uns,\kappa}^{(N)}(\p_x\uns)\mathfrak{u} \Big]
		\\
		&\qquad
		+
		\left( L_{\uns,\kappa}^{(N)}\right)^{-1}\Big[
		\left( I - \PP_{\uns,\kappa}^{(N)}\right)\Pi_N(\xi_1-\sfs)\p_x\tilde{\mathfrak{v}}
		\Big]
		\\
		&\qquad
		-
		\eta\left( L_{\uns,\kappa}^{(N)}\right)^{-1}\Big[\left( I - \PP_{\uns,\kappa}^{(N)}\right)\p_x^2 f \Big]- \left( L_{\uns,\kappa}^{(N)}\right)^{-1}z
		\\
		&=
		I_1 + I_2 - I_3 - I_4.
	\end{align*}
	Defining 
	\begin{align}
		\label{approximate_convection}
		J_\kappa^{(N)}(\bar{\mathfrak{u}}) = \PP\left(\xi_1 M_{\bar{\mathfrak{u}},\kappa}^{(N)}\right)
	\end{align}
	and
	\begin{align}
		\label{approximate_diffusion}
		B_\kappa^{(N)}(\bar{\mathfrak{u}}) = -\PP\!\left(\xi_1 b_{\kappa,N}^\perp[\bar{\mathfrak{u}}]\right)
	\end{align}
	and taking macroscopic parts of \eqref{Galerkin_problem} then gives that
	\begin{align*}
		0 &=
		\p_x\PP\left[ (\xi_1-\sfs)\left( \dd M_{\uns,\kappa}^{(N)}\mathfrak{u}+\tilde{\mathfrak{v}}\right)\right]-\eta\p_x^2 \mathfrak{u}.
	\end{align*}
	Using that $f\in H^2_\epsilon$ by assumption, we can integrate this in $x$ and re-express it as
	\begin{align*}
		\left( {J_{\kappa}^{(N)}}'(\uns) - \sfs\right)u - B_\kappa^{(N)}(\uns)\p_x \mathfrak{u}  = \PP(\xi_1 r) + \eta\p_x \mathfrak{u}
	\end{align*}
	where we have used that $\PP A^{(N)}\Pi_N = \PP\xi_1\Pi_N$ by construction of $A^{(N)}$. Using Lemma \ref{convergence_of_coefficients}, we can rewrite this as
	\begin{align}
		\label{Galerkin_macro_error_estimate}
		(J'(\uns)-\sfs)u - B_\kappa(\uns)\p_x \mathfrak{u} = \PP(\xi_1 r) + \eta\p_x \mathfrak{u} + O(N^{-1})\mathfrak{u} + O(N^{-1})\p_x \mathfrak{u},
	\end{align}
	so it remains to bound the remainder term $r$. Using that $$\Pi_N(\xi_1-\sfs) \dd^2 M_{\bar{\mathfrak{u}},\kappa}^{(N)}$$ is uniformly bounded in $\HH^0$ by Lemma \ref{discretized_Maxwellians_lemma}, and by uniform invertibility of $L^{(N)}_{\bar{\mathfrak{u}},\kappa}$ in the $\HH^0$ norm from Lemma \ref{discretized_uniform_invertibility_lemma}, we can bound
	\begin{align*}
		\norm{I_1}_{\HH^0}\le C\abs{\p_x\uns}\norm{\mathfrak{u}}\le C\epsilon^2\norm{\mathfrak{u}},
	\end{align*}
	using the asymptotics in Lemma \ref{ABN24_Prop_B.2}. To bound the second term, we can use Lemma \ref{HH1kappa_Sobolev_bound} to estimate
	\begin{align*}
		\norm{(\xi_1-\sfs)\p_x\tilde{\mathfrak{v}}}_{\HH^0} \le C\kappa^{-1/2}\norm{\tilde{\mathfrak{v}}}_{\HH^1_\kappa},
	\end{align*}
	which then gives the bound
	\begin{align}
		\label{kappa-1/2_bound}
		\norm{I_2}\le C\kappa^{-1/2}\norm{\p_x\tilde{\mathfrak{v}}}
	\end{align}
	for a uniform constant $C$. The remaining terms can be bounded as
	\begin{align*}
		\norm{I_3}_{\HH^0}\le C\eta\norm{\p_x^2 f}_{\HH^0},\qquad
		\norm{I_4}_{\HH^0}\le C\norm{z}_{\HH^{-1}_\kappa}
	\end{align*}
	for a uniform constant $C$ using Lemma \ref{discretized_uniform_invertibility_lemma}, and combining terms gives the estimate
	\begin{align*}
		\norm{r}_{\HH^1_\kappa}
		\le C\left(\epsilon^2\norm{\mathfrak{u}} + \kappa^{-1/2}\norm{\p_x\tilde{\mathfrak{v}}}_{\HH^1_\kappa}
		+\eta\norm{\p_x^2 f}_{\HH^0} + \norm{z}_{\HH^{-1}_\kappa}\right).
	\end{align*}
	
	We can now apply the viscous shock stability estimate in Lemma \ref{Lemma_6.1_Problem} to \eqref{Galerkin_macro_error_estimate}, giving a bound
	\begin{align*}
		\begin{split}
			\norm{\mathfrak{u}}_{L^2_\epsilon}
			&\le
			\frac{C}{\epsilon}\Big( 
			\norm{z}_{L^2_\epsilon\HH^{-1}_\kappa} + \left(\eta + N^{-1}\right)\norm{\p_x\mathfrak{u}}_{L^2_\epsilon}
			+
			\left( \epsilon^2 + N^{-1}\right)\norm{\mathfrak{u}}_{L^2_\epsilon}
			\\
			&\qquad\qquad\qquad\qquad\qquad\quad
			+ \eta\norm{\p_x^2 f}_{L^2_\epsilon\HH^0} + \kappa^{-1/2}\norm{\p_x\tilde{\mathfrak{v}}}_{L^2_\epsilon\HH^1_\kappa}
			\Big) + C\abs d
		\end{split}
	\end{align*}
	which reduces to
	\begin{align}
		\begin{split}
			\label{Galerkin_macroscopic_error_bound}		
			\norm{\mathfrak{u}}_{L^2_\epsilon}
			&\le
			\frac{C}{\epsilon}\Big( 
			\norm{z}_{L^2_\epsilon\HH^{-1}_\kappa} 
			+ \left(\eta + \epsilon^2\right)\norm{\p_x\mathfrak{u}}_{L^2_\epsilon}
			\\
			&\qquad\qquad\qquad\qquad
			+ \eta\norm{\p_x^2 f}_{L^2_\epsilon\HH^0} + \kappa^{-1/2}\norm{\p_x\tilde{\mathfrak{v}}}_{L^2_\epsilon\HH^1_\kappa}
			\Big) 
			+ C\abs d
		\end{split}
	\end{align}
	for $\epsilon$ sufficiently small, and $N\gg\epsilon^{-2}$. To compensate for the $\p_x\tilde{\mathfrak{v}}$ error term, we proceed as in the infinite-dimensional case and
	differentiating \eqref{Galerkin_problem} to get
	\begin{align*}
		(A^{(N)}\!-\sfs)\p_x^2 f - \eta\p_x^3 f - L_{\uns,\kappa}^{(N)}\p_x f
		= \p_x z + Q^{(N)}_\kappa\!\!\left(\p_x M_{\uns,\kappa}^{(N)},f\right)
		+  Q^{(N)}_\kappa\!\!\left(f,\p_x M_{\uns,\kappa}^{(N)}\right)
	\end{align*}
	and applying the basic energy estimate in Lemma \ref{Galerkin_energy_estimate} gives
	\begin{align*}
		\sqrt\eta\norm{\p_x^3 f}_{L^2_\epsilon\HH^0} + \norm{\p_x^2 f}_{L^2_\epsilon\HH^1_\kappa}
		&+ \norm{\p_x\mathfrak{v}}_{L^2_\epsilon\HH^1_\kappa}
		\\
		&
		\le
		C\left( \norm{\p_x^2 z,\p_x z}_{L^2_\epsilon\HH^{-1}_\kappa} + \epsilon^2\norm{\mathfrak{u}}_{L^2_\epsilon} + \epsilon^2\norm{\mathfrak{v}}_{L^2_\epsilon\HH^{1}_\kappa}
		\right)
	\end{align*}
	where we have used differentiability of $M_{\bar{\mathfrak{u}},\kappa}^{(N)}$ from Lemma \ref{discretized_Maxwellians_lemma} and the viscous shock asymptotics from Lemma \ref{ABN24_Prop_B.2}. Combining this with \eqref{Galerkin_macroscopic_error_bound} and the original energy estimate in Lemma \ref{Galerkin_energy_estimate} gives
	\begin{align*}
		\sqrt\eta & \norm{\p_x^3 f}_{L^2_\epsilon\HH^0} 
		+ \norm{\p_x^2 f}_{L^2_\epsilon\HH^1_\kappa}
		+ \epsilon\sqrt{\eta}\norm{\p_x^2 f}_{L^2_\epsilon\HH^0}^2 
		+\norm{\p_x\mathfrak{v}}_{L^2_\epsilon\HH^1_\kappa}
		\\
		&\quad
		+ \epsilon\norm{\p_x f}_{L^2_\epsilon\HH^1_\kappa} + \epsilon\norm{\mathfrak{v}}_{L^2_\epsilon\HH^1_\kappa} + \epsilon^2\norm{\mathfrak{u}}_{L^2_\epsilon}
		\\
		&\qquad\quad
		\le C\Big( \norm{\p_x^2 z,\p_x z}_{L^2_\epsilon\HH^{-1}_\kappa} + \epsilon\norm{z}_{L^2_\epsilon\HH^{-1}_\kappa}
		+
		\epsilon^2\norm{\mathfrak{v}}_{L^2_\epsilon\HH^{1}_\kappa}
		\\
		&\qquad\quad\quad
		+
		\left(\epsilon^3 + \epsilon\eta\right)\norm{\p_x\mathfrak{u}}_{L^2_\epsilon} + \eta\epsilon\norm{\p_x^2 f}_{L^2_\epsilon\HH^0}
		+\epsilon\kappa^{-1/2}\norm{\p_x\tilde{\mathfrak{v}}}_{L^2_\epsilon\HH^1_\kappa}
		+ \epsilon^2 \abs d\Big)
	\end{align*}
	for some constant $C$ independent of all parameters. We see that for $\eta<\eta_0=4/C^2$, the terms on the right-hand side involving $\eta$ can be absorbed, as can the terms $\epsilon^3\norm{\p_x\mathfrak{u}}_{L^2_\epsilon}$ and $\epsilon^2\norm{\mathfrak{v}}_{L^2_\epsilon\HH^1_\kappa}$ for $\epsilon>0$ sufficiently small. Bounding
	\begin{align*}
		\norm{\p_x\tilde{\mathfrak{v}}}_{\HH^1_\kappa}
		&\le\norm{\p_x\mathfrak{v}}_{\HH^1_\kappa}
		+
		\norm{\mrm{d}^2 M_{\uns,\kappa}^{(N)}(\p_x\uns)\mathfrak{u}}
		+
		\norm{\left(I-\dd M_{\uns,\kappa}^{(N)}\right)\p_x \mathfrak{u}}
		\\
		&\le
		\norm{\p_x\mathfrak{v}}_{\HH^1_\kappa} 
		+
		\epsilon^2\norm{\mathfrak{u}} + \epsilon\norm{\p_x\mathfrak{u}}
	\end{align*}
	using the viscous shock asymptotics \eqref{ABN24_Prop_B.2_exponential_decay} and the discretized Maxwellian derivative estimates in Lemma \ref{discretized_Maxwellians_lemma}, the term $\epsilon\kappa^{-1/2}\norm{\p_x\tilde{\mathfrak{v}}}_{L^2_\epsilon\HH^1_\kappa}$ can then be absorbed by the left-hand side, for $\epsilon\ll\kappa^{1/2}$ sufficiently small, which proves the proposition.
\end{proof}

The proof of the uniform estimate \eqref{Galerkin_macroscopic_estimate} depends on consistency of the Chapman-Enskog matrices $(J^{(N)}_\kappa,B^{(N)}_\kappa)$ arising from the Galerkin scheme, with the original hydrodynamic matrices $(J,B_\kappa)$, and we show this in the following lemma.

\begin{lemma}
	\label{convergence_of_coefficients}
	For $\kappa_0,\epsilon_0>0$ sufficiently small, we have convergence 
	\begin{align*}
		\sup_{\kappa\in[0,\kappa_0]}\norm{J_\kappa^{(N)}-J}_{C^1(B(\uref,\epsilon_0))}
		\le C_k N^{-k}
	\end{align*}
	and
	\begin{align*}
		\sup_{\kappa\in[0,\kappa_0]}\norm{B_\kappa^{(N)}-B_\kappa}_{C^0(B(\uref,\epsilon_0))}\le C N^{-1}
	\end{align*}
	for all $k\ge 1$, for some constants $C_k,C>0$. 
	Furthermore, we have the uniform bounds
	\begin{align*}
		\norm{b^\perp_{\kappa,N}[\bar{\mathfrak{u}}]}_{C_{\bar{\mathfrak{u}}}^0(B(\uref,\epsilon_0),\HH^1_\kappa)}
		\le C
	\end{align*}
	for a constant $C>0$ independent of $\kappa\in[0,\kappa_0]$ and $N\ge 3$, where the discretized Chapman-Enskog correction term $b^\perp_{\kappa,N}$ was defined in \eqref{discretized_Chapman-Enskog_correction}.
\end{lemma}

The convergence rate of the diffusion coefficients $B_\kappa^{(N)}$ can be improved to $C N^{-k_0}$ for any fixed $k_0$, though at the expense of a potentially smaller $\epsilon_0>0$. To see this, we note that super-polynomial convergence in the Galerkin scheme arises from results for convergence rates of the Hermite projections for Schwartz class functions, such as Maxwellian distributions, as in Lemma \ref{discretized_Maxwellians_lemma}. However our methods provide only finitely many velocity derivatives and moments for the Chapman-Enskog correction term $b^\perp_\kappa$ in Lemma \ref{microscopic_correction_smoothness}, due to nonlinear estimates on the collision operator which may worsen at high regularity.

\begin{proof}[Proof of Lemma \ref{convergence_of_coefficients}]
	Convergence $M_{\bar{\mathfrak{u}},\kappa}^{(N)}\to M_{\bar{\mathfrak{u}}}\in C^1_{\bar{\mathfrak{u}}}\HH^0$ from Lemma \ref{discretized_Maxwellians_lemma} and boundedness of the operator $\PP\xi_1$ immediately gives the first claim, from the definition \eqref{approximate_convection}. To prove convergence of the diffusion coefficients, we prove the asymptotics
	\begin{align*}
		\sup_{\kappa\in[0,\kappa_0]}\norm{b_{\kappa,N}^{\perp}[\bar{\mathfrak{u}}] - b_{\kappa}^\perp[\bar{\mathfrak{u}}]}_{C^0_{\bar{\mathfrak{u}}}\HH^0}\le CN^{-1},
	\end{align*}
	and by boundedness of the operator $\PP\xi_1$, this will prove the second claim. Expressing 
	\begin{align}
		\label{gkappaN_equation}
		g_{\kappa,N} &= \left( I -\PP_{\bar{\mathfrak{u}},\kappa}^{(N)}\right)A^{(N)}\dd M_{\bar 		u,\kappa}^{(N)}
		= \left( I - \dd M_{\bar{\mathfrak{u}},\kappa}^{(N)}\PP\right)\Pi_N\xi_1\dd M_{\bar{\mathfrak{u}},\kappa}^{(N)}
	\end{align}
	and
	\begin{align*}
		g = (I-\PP_{\bar{\mathfrak{u}}})\xi_1\dd M_{\bar{\mathfrak{u}}}
	\end{align*}
	such that $b_{\kappa,N}^{\perp} = ( L_{\bar{\mathfrak{u}},\kappa}^{(N)})^{-1}g_{\kappa,N}$ and $b_\kappa^\perp = L_{\bar{\mathfrak{u}},\kappa}^{-1}g$, we can compute
	\begin{align*}
		\norm{\Pi_N\xi_1\dd M_{\bar{\mathfrak{u}},\kappa}^{(N)} - \xi_1\dd M_{\bar{\mathfrak{u}}}}_{\mathsf{H}^{10}_{10}}
		&	\le
		\norm{\Pi_N\left( \xi_1\dd M_{\bar{\mathfrak{u}},\kappa}^{(N)} - \xi_1\dd M_{\bar{\mathfrak{u}}}\right)}_{\mathsf{H}^{10}_{10}}
		\\
		&\qquad
		+
		\norm{(I-\Pi_N)\xi_1\dd M_{\bar{\mathfrak{u}}}}_{\mathsf{H}^{10}_{10}}
		\\
		&\le
		C N^{-1}
	\end{align*}
	uniformly in $\kappa\in[0,\kappa_0]$, where the first term is bounded by uniform convergence of $M_{\bar{\mathfrak{u}},\kappa}^{(N)}$ in $\mathsf{H}^{12}_{12}$ from Lemma \ref{discretized_Maxwellians_lemma} and uniform boundedness of $\Pi_N$ on $\mathsf{H}^{10}_{10}$, while the second term is bounded using smoothness of the Maxwellians $M_{\bar{\mathfrak{u}}}\in\mathsf{H}^8_8$ from Lemma \ref{Maxwellian_manifold_smoothness} and the estimate \eqref{projections_norm_convergence}. By continuity $\PP:\mathsf{H}^{10}_{10}\to \UU$ and by using uniform convergence $\dd M_{\bar{\mathfrak{u}},\kappa}^{(N)}\to\dd M_{\bar{\mathfrak{u}}}$ in $\mathsf{H}^{10}_{10}$ by Lemma \ref{discretized_Maxwellians_lemma}, we get convergence
	\begin{align}
		\label{gNkappa_bound}
		\sup_{\kappa\in[0,\kappa_0]}\norm{g_{\kappa,N}[\bar{\mathfrak{u}}] - g[\bar{\mathfrak{u}}]}_{C^0_{\bar{\mathfrak{u}}}(B(\uref,\epsilon_0),\mathsf{H}^{10}_{10})}\le CN^{-1}.
	\end{align}
	We can then express
	\begin{align*}
		\left( L_{\bar{\mathfrak{u}},\kappa}^{(N)}\right)^{-1}\!\!g_{\kappa,N} 
		- L_{\bar{\mathfrak{u}},\kappa}^{-1}g
		&=
		L_{\bar{\mathfrak{u}},\kappa}^{-1}
		\left( L_{\bar{\mathfrak{u}},\kappa}\Pi_N - L_{\bar{\mathfrak{u}},\kappa}^{(N)}  \right)
		\left(L_{\bar{\mathfrak{u}},\kappa}^{(N)}\right)^{-1} \!g_{\kappa,N}
		+ 
		L_{\bar{\mathfrak{u}},\kappa}^{-1}\left (g_{\kappa,N} - g\right)
		\\
		&=: I_1 + I_2.
	\end{align*}
	and we bound each term. The first term can be bounded by
	\begin{align*}
		\norm{I_1}_{\HH^0}
		&\le
		\norm{ L_{\bar{\mathfrak{u}},\kappa}^{-1}}_{\mathscr{B}(\VV)}
		\norm{ L_{\bar{\mathfrak{u}},\kappa}\Pi_N - L_{\bar{\mathfrak{u}},\kappa}^{(N)}}_{\mathscr{B}(\mathsf{H}^{10}_{10},\HH)}
		\norm{\left( L_{\bar{\mathfrak{u}},\kappa}^{(N)}\right)^{-1} }_{\mathscr{B}(\mathsf{H}^{10}_{10}\cap\VV_N)}\norm{g_{\kappa,N}}_{\mathsf{H}^{10}_{10}}
		\\
		&\le
		C\norm{ L_{\bar{\mathfrak{u}},\kappa}\Pi_N - L_{\bar{\mathfrak{u}},\kappa}^{(N)}}_{\mathscr{B}(\mathsf{H}^{10}_{10},\HH)},
	\end{align*}
	where $\mathscr{B}(E_1,E_2)$ is the space of bounded linear operators $L:E_1\to E_2$ on arbitrary Banach spaces $E_1,E_2$ under the operator norm and $\mathscr{B}(E):=\mathscr{B}(E,E)$, and $C>0$ is a constant independent of $\kappa,N$, and $\bar{\mathfrak{u}}\in B(\uref,\epsilon_0)$. Here we have used Lemmas \ref{uniform_invertibility_lemma} and \ref{discretized_uniform_invertibility_lemma} and equation \eqref{gNkappa_bound}. But we can then bound
	\begin{align*}
		\norm{ L_{\bar{\mathfrak{u}},\kappa}\Pi_N - L_{\bar{\mathfrak{u}},\kappa}^{(N)}}_{\mathscr{B}(\mathsf{H}^{10}_{10},\HH)}
		&\le
		\norm{(I-\Pi_N) L_{\bar{\mathfrak{u}},\kappa}\Pi_N}_{\mathscr{B}(\mathsf{H}^{10}_{10},\HH)}
		\\
		&\quad
		+
		\norm{\Pi_N Q_\kappa'\left( M_{\bar{\mathfrak{u}} } - M_{\bar{\mathfrak{u}},\kappa}^{(N)}\right)\Pi_N}_{\mathscr{B}(\mathsf{H}^{10}_{10},\HH)}
		\\
		&\le \norm{I-\Pi_N}_{\mathscr{B}(\mathsf{H}^6_6,\HH)}
		\norm{ L_{\bar{\mathfrak{u}},\kappa}}_{\mathscr{B}(\mathsf{H}^{10}_{10},\mathsf{H}^6_6)}
		\\
		&\qquad
		+
		C\norm{M_{\bar{\mathfrak{u}} } - M_{\bar{\mathfrak{u}},\kappa}^{(N)} }_{\mathsf{H}^{4}_{4}}
		\\
		&\le
		C N^{-1}
	\end{align*}
	where we write $Q_\kappa'(g) = Q_\kappa(g,\cdot) + Q_\kappa(\cdot,g)$, and
	we have used \eqref{projections_norm_convergence} to bound the projection $I-\Pi_N$ and \eqref{high_frequency_Qkappa_bound} to bound the collision terms, and Lemma \ref{discretized_Maxwellians_lemma} to estimate the difference of Maxwellians. The second term can be bounded as
	\begin{align*}
		\norm{I_2}_{\HH^0}\le \norm{L_{\bar{\mathfrak{u}},\kappa}^{-1}}_{\mathscr{B}(\VV)}
		\norm{g_{\kappa,N} - g}_{\HH^0} 
		\le C N^{-1}
	\end{align*}
	using the estimate \eqref{gNkappa_bound} and uniform invertibility of $L_{\bar{\mathfrak{u}},\kappa}$ on $\VV$ from Lemma \ref{uniform_invertibility_lemma}.
	
	To prove uniform boundedness of $b^\perp_{\kappa,N} = (L_{\bar{\mathfrak{u}},\kappa}^{(N)})^{-1} g_{\kappa,N}$, we note that uniform bounds on $g_{\kappa,N}\in\HH^0$ immediately follow from the definition \eqref{gkappaN_equation} and from Lemma \ref{discretized_Maxwellians_lemma}, and by the uniform invertibility result of Lemma \ref{discretized_uniform_invertibility_lemma} this proves the claim.
\end{proof}

We would now like to explicitly construct our solutions. To do so, we do an eigenvalue analysis of the problem and compute stable and unstable dimensions.

\subsection{Eigenvalue computations} 
\label{eigenvalue_computations_section}

We now rewrite \eqref{Galerkin_problem} as a first-order ODE. Using the macro-micro orthogonal decomposition $\HH_N = \UU\oplus\VV_N$, we can re-express the operators
\begin{align*}
	A^{(N)} - \sfs = \begin{pmatrix} A_{00} & A_{01} \\ A_{10} & A_{11}\end{pmatrix},
	\qquad
	L_{\uns,\kappa}^{(N)} = \begin{pmatrix} 0 & 0 \\ L_0 & L_1 \end{pmatrix}
\end{align*}
in matrix form, giving the macroscopic equation
\begin{align*}
	\p_x(A_{00}\mathfrak{u} + A_{01}\mathfrak{v}) - \eta\p_x^2 \mathfrak{u} = 0,
\end{align*}
which can be integrated in $x$ for $U,V\in L^2(\RR)$, giving the equation
\begin{align*}
	\p_x \mathfrak{u} = \frac1\eta \left( A_{00} \mathfrak{u} + A_{01}\mathfrak{v}\right).
\end{align*}
Inserting this in the microscopic equation, we can then re-express the ODE in \eqref{Galerkin_problem} 
as
\begin{align*}
	-\p_x \begin{pmatrix} \mathfrak{u} \\ \mathfrak{v} \\ \p_x\mathfrak{v} \end{pmatrix}
	+
	\frac1\eta \begin{pmatrix}
		A_{00} & A_{01} & 0\\
		0 & 0 & \eta \\
		\frac1\eta A_{10}A_{00} - L_0 & \frac1\eta A_{10}A_{01}-L_1 & A_{11}
	\end{pmatrix}
	\begin{pmatrix}
		\mathfrak{u} \\ \mathfrak{v} \\ \p_x\mathfrak{v}
	\end{pmatrix}
	&=
	\begin{pmatrix}
		0 \\ 0 \\ z 
	\end{pmatrix},
\end{align*}
which we rewrite in the form
\begin{align}
	\label{rewritten_second_order_ODE}
	-\p_x\ca F + \mathbb{A}\ca F = \ca G,
\end{align}
for the variables $\ca F = (\mathfrak{u},\mathfrak{v},w)^\top\in\UU\oplus\VV_N\oplus\VV_N$, 
and we define the endpoint matrices $\mathbb{A}_\pm = \mathbb{A}(\pm\infty)$. This ODE is asymptotically hyperbolic in the following sense. 

\begin{lemma}
	\label{ODE_hyperbolicity}
	Let $\kappa>0$ be sufficiently small and $\epsilon\in(0,\epsilon_0(\kappa))$ for $\epsilon_0(\kappa)$ be sufficiently small, and let $N\ge N_0(\epsilon)$ be sufficiently large. Then for any $\eta>0$, the matrices $\mathbb{A}_\pm$ have no eigenvalues on the imaginary axis $i\RR$.
\end{lemma}
We note that there is no smallness condition on the artificial viscosity term $\eta\p_x^2$ -- we will use this to compute stable and unstable dimensions of the matrices $\mathbb{A}_\pm$ for large $\eta>0$, and therefore for all $\eta>0$, by a perturbation argument.

\begin{proof}[Proof of Lemma \ref{ODE_hyperbolicity}]
	We argue by contradiction. We see that the eigenvalue equation
	\begin{align*}
		-i\tau\ca F + \mathbb{A}_\pm\ca F = 0
	\end{align*}
	is equivalent to the system
	\begin{align}
		\begin{split}
			\label{Laplace_transformed_equation}
			&
			\hat{\ca L}f:= i\tau(A^{(N)}-\sfs)f + \eta\tau^2 f - L_{\mathfrak{u}_,\kappa}^{(N)}f = 0,
			\\
			&
			\hat{\ca L}_\UU f:=  \PP(A^{(N)}-\sfs)f - i\eta\tau \mathfrak{u}  = 0,
		\end{split}
	\end{align}
	by replacing $\p_x$ with $i \tau$ in the derivation of \eqref{rewritten_second_order_ODE}. To prove that \eqref{Laplace_transformed_equation} has no nontrivial solutions $f\in\HH_N\otimes\mathbb{C}\setminus\set0$, we mimic the previous stability arguments from Proposition \ref{Galerkin_macroscopic_stability} in the spectral setting. For any $\tau\in\RR$ we first define the associated energy
	\begin{align*}
		\hat E_\tau[f] := \RE\br{\left(\tau^2 - i\tau K + \lambda\right)\hat{\ca L}f,f}=0,
	\end{align*}
	for a parameter $\lambda>0$ to be chosen later. Defining $\RE A = (A+A^*)/2$ for any operator $A$ on $\HH_N$, we can compute that
	\begin{align}
		\label{Galerkin_spectral_energy}
		\RE\hat{\ca L}
		=
		\eta\tau^2 - \RE L_{\mathfrak{u}_\pm,\kappa}^{(N)}
	\end{align}
	and
	\begin{align*}
		\RE \left( -i\tau K\hat{\ca L}\right)
		=
		\tau^2\RE\!\left[ K A^{(N)}\right] - i \eta\tau^3 K
		+
		\RE\!\left[ i\tau K L_{\mathfrak{u}_\pm,\kappa}^{(N)}\right]
	\end{align*}
	and collecting terms, we get that
	\begin{align*}
		\begin{split}
			\RE\left[ \left(\tau^2 - i\tau K + \lambda\right)\hat{\ca L}\right]
			&
			= \eta\lambda\tau^2 + \eta \tau^2
			\left[ \tau^2-i\tau K\right]
			-
			\lambda \RE L_{\mathfrak{u}_\pm,\kappa}^{(N)} 
			\\
			&
			\qquad
			+
			\tau^2\RE\!\left[ K A^{(N)}-L_{\mathfrak{u}_\pm,\kappa}^{(N)}\right] +\RE\!\left[ i\tau K L_{\mathfrak{u}_\pm,\kappa}^{(N)}\right].
		\end{split}
	\end{align*}
	Applying this to \eqref{Galerkin_spectral_energy}, by using boundedness of $K$ as a finite rank operator, and using the coercivity and trilinear estimates in Lemma \ref{discretized_collisions_lemma} and the approximate Maxwellian bounds in Lemma \ref{discretized_Maxwellians_lemma}, we get the estimate
	\begin{align}
		\begin{split}
			\label{eq7.20}
			\eta\lambda\tau^2\norm{f}^2 &+ \eta  \tau^4 \norm{f}^2 + \lambda\de\norm{\mathfrak{v}}^2_{\HH^1_\kappa}
			+
			\tau^2\de\norm{f}^2_{\HH^1_\kappa}
			\\
			&\le
			C \left( \eta\abs\tau^3\norm{f}^2
			+
			\lambda\epsilon \norm{\mathfrak{u}}\norm{\mathfrak{v}}_{\HH^1_\kappa}
			+
			\abs\tau\norm{f}_{\HH^1_\kappa}\left(\epsilon\norm{\mathfrak{u}}+\norm{\mathfrak{v}}_{\HH^1_\kappa}\right) \right),
		\end{split}
	\end{align}
	for some uniform constants $C,\de>0$, where for the last term we have used the bound
	\begin{align*}
		\abs{\RE\br{K L_{\mathfrak{u}_\pm,\kappa}^{(N)}f,f}}
		&=
		\abs{\RE\br{K L_{\mathfrak{u}_\pm,\kappa}^{(N)}\tilde{\mathfrak{v}},f}}
	\end{align*}
	where $\tilde{\mathfrak{v}} = (I - \PP_{\mathfrak{u}_\pm,\kappa}^{(N)})f$, but we can bound this by
	\begin{align*}
		C\norm{K}\norm{\tilde{\mathfrak{v}}}_{\HH^1_\kappa}\norm{f}_{\HH^1_\kappa}
		\le 
		C\norm{f}_{\HH^1_\kappa}\left( \epsilon\norm{\mathfrak{u}} + \norm{\mathfrak{v}}_{\HH^1_\kappa}\right)
	\end{align*}
	by Lemmas \ref{discretized_collisions_lemma} and \ref{discretized_Maxwellians_lemma}, and by using that $K$ is finite rank. By taking $\lambda$ sufficiently large (independently of all parameters), the terms $\eta\abs\tau^2\norm{f}^2$ and $\abs\tau\norm{f}_{\HH^1_\kappa}\norm{\mathfrak{v}}_{\HH^1_\kappa}$ in \ref{eq7.20} can be controlled by the left-hand side, giving the bound
	\begin{align*}
		\eta\left( \tau^2 + \tau^4\right)\norm{f}^2 + \tau^2\norm{f}_{\HH^1_\kappa}^2
		+
		\norm{\mathfrak{v}}_{\HH^1_\kappa}^2
		\le
		C\left(
		\epsilon \norm{\mathfrak{u}}\norm{\mathfrak{v}}_{\HH^1_\kappa} + \epsilon \abs\tau
		\norm{f}_{\HH^1_\kappa}\norm{\mathfrak{u}}\right),
	\end{align*}
	but by Young's inequality, this can be improved to
	\begin{align*}
		\eta\left(\tau^2 + \tau^4\right)\norm{f}^2 + \tau^2\norm{f}_{\HH^1_\kappa}^2
		+ \norm{\mathfrak{v}}_{\HH^1_\kappa}^2
		\le
		C\epsilon^2 \norm{\mathfrak{u}}^2 \le C\epsilon^2\norm{f}^2.
	\end{align*}
	By coercivity of the $\HH^1_\kappa$ norm in $\HH^0$ and since $f\ne 0$, this then implies the bound
	\begin{align}
		\label{tau_upper_bound}
		\abs\tau\le\frac{C\epsilon}{\sqrt{1+\eta}}.
	\end{align}
	
	To show that $\tau$ cannot lie in this range, we use macroscopic estimates. Taking uncentred projections of the first equation of \eqref{Laplace_transformed_equation} gives
	\begin{align*}
		i\tau\left( I - \PP_{\mathfrak{u}_\pm,\kappa}^{(N)}\right)
		\left( A^{(N)}-\sfs\right) f + \eta\tau^2\tilde{\mathfrak{v}} = L_{\mathfrak{u}_\pm,\kappa}^{(N)}f
		=
		L_{\mathfrak{u}_\pm,\kappa}^{(N)}\tilde{\mathfrak{v}}
	\end{align*}
	and solving for $\tilde{\mathfrak{v}}$ gives
	\begin{align*}
		\tilde{\mathfrak{v}} &= \left( L_{\mathfrak{u}_\pm,\kappa}^{(N)}\right)^{-1}
		\left[
		i\tau\left( I - \PP_{\mathfrak{u}_\pm,\kappa}^{(N)}\right)\left(A^{(N)}-\sfs\right)
		\left(\dd M_{\mathfrak{u}_\pm,\kappa}^{(N)} \mathfrak{u} +\tilde{\mathfrak{v}}\right)
		+\eta\tau^2\tilde{\mathfrak{v}}
		\right]
		\\
		&=
		i\tau b^\perp_{\kappa,N}[\mathfrak{u}_\pm]u + 
		i\tau\left( L_{\mathfrak{u}_\pm,\kappa}^{(N)}\right)^{-1}
		\left[
		\left( I - \PP_{\mathfrak{u}_\pm,\kappa}^{(N)}\right)\left(A^{(N)}-\sfs\right)
		\tilde{\mathfrak{v}}\right]
		+\eta\tau^2\left( L_{\mathfrak{u}_\pm,\kappa}^{(N)}\right)^{-1}\tilde{\mathfrak{v}}
		\\
		&= i\tau b^\perp_{\kappa,N}[\mathfrak{u}_\pm]u + I_2 + I_3
	\end{align*}
	where the last two terms can be bounded as
	\begin{align}
		\label{discretized_spectral_Chapman-Enskog_remainder_terms}
		\norm{I_2}_{\HH^1_\kappa}\le C\kappa^{-1/2}\abs\tau\norm{\tilde{\mathfrak{v}}}_{\HH^1_\kappa},
		\qquad
		\norm{I_3}_{\HH^1_\kappa}\le C\eta\tau^2\norm{\tilde{\mathfrak{v}}}_{\HH^1_\kappa}
	\end{align}
	as in \eqref{kappa-1/2_bound}, and therefore we can bound
	\begin{align*}
		\norm{\tilde{\mathfrak{v}}}_{\HH^1_\kappa}\le C\abs\tau\norm{\mathfrak{u}}
	\end{align*}
	for $\epsilon<\epsilon_0(\kappa)$ sufficiently small, where we have used uniform boundedness of the Chapman-Enskog correction $b^\perp_{\kappa,N}\in\HH^1_\kappa$ from Lemma \ref{convergence_of_coefficients}, and we note that this also implies $u\ne 0$ by nontriviality of $f$. Applying this to the macroscopic equation of \eqref{Laplace_transformed_equation} then gives
	\begin{align*}
		0
		&=
		\PP\left( A^{(N)} - \sfs\right)\left( \dd M_{\mathfrak{u}_\pm,\kappa}^{(N)}\mathfrak{u} + \tilde{\mathfrak{v}}\right) - i\eta\tau \mathfrak{u} 
		\\
		&= \left[ {J^{(N)}_\kappa}'(\mathfrak{u}_\pm)-\sfs\right] \mathfrak{u} - i\tau B^{(N)}_\kappa(\mathfrak{u}_\pm)\mathfrak{u}- i\eta\tau \mathfrak{u} + \PP\left( \xi_1I_2 + \xi_1 I_3\right).
	\end{align*}
	Bounding the remainder terms $I_2,I_3$ by \eqref{discretized_spectral_Chapman-Enskog_remainder_terms} and
	using the error estimate for the hydrodynamic coefficients from Lemma \ref{convergence_of_coefficients} gives the equation
	\begin{align*}
		\begin{split}
			\left[ J'(\mathfrak{u}_\pm) - \sfs\right] \mathfrak{u} - i\tau \left[ B_\kappa(\mathfrak{u}_\pm) +\eta\right]\mathfrak{u}
			&=
			O\left(\left(
			\frac{1+\abs\tau}{N} + C_\kappa\tau^2 + \eta\tau^3
			\right)\norm{\mathfrak{u}}\right)
			\\
			&= O\left(C_\kappa\epsilon^2\norm{\mathfrak{u}}\right)
		\end{split}
	\end{align*}
	where we have taken $N\gtrsim\epsilon^{-2}$. We now rewrite this system using the symmetric hyperbolic-parabolic theory of Kawashima \cite{Kawashima_thesis}.
	Expressing $W = (\rho,m,E)$ in conservative variables and using the (smooth) symmetrizer $A_0$ defined in Lemma \ref{Kawashima_CNSE_symmetrizer}, we define $S = A_0^{-1}$, such that the matrices $\bar{A}_\pm = S(\mathfrak{u}_\pm) J'(\mathfrak{u}_\pm)$ and $\bar{B}_{\kappa,\pm} = S(\mathfrak{u}_\pm) B_\kappa(\mathfrak{u}_\pm)$ are symmetric, where we express $J'$ and $B_\kappa$ in conservative variables, giving the reformulated equation
	\begin{align}
		\label{eq7.24}
		\bar{\ca L}\mathfrak{u}:=\left[ \bar{A}_\pm -\sfs\right] \mathfrak{u}-i\tau\left[ \bar{B}_{\kappa,\pm} + \eta S_\pm \right]\mathfrak{u}
		=
		O\left(C_\kappa\epsilon^2\norm{\mathfrak{u}}\right),
	\end{align}
	and we wish to show that the operator $\bar{\ca L}$ has inverse of order $\norm{\bar{\ca L}^{-1}}\lesssim\epsilon^{-1}$, thus deriving a contradiction with nontriviality of $u$. We know as in the proof of Lemma \ref{Lemma_6.1} that $J'(\mathfrak{u}_\pm)-\sfs$ has two eigenvalues of order $O(1)$ and a simple eigenvalue of order $\epsilon$, so by computing asymptotics of the determinants for $\epsilon$ and $\tau$ small and using that the matrices $\bar{A}_{\pm},\bar{B}_{\kappa,\pm},S_{\pm}$ are real, this is proved for $(1+\eta)\abs\tau\le c_0$ for some $c_0>0$ sufficiently small. Meanwhile, whenever $\abs\tau\eta\ge c_0\epsilon$, we can use energy estimates, exploiting symmetry of the operators \eqref{eq7.24} and coercivity of the term $\eta S_\pm$ to prove $\norm{\bar{\ca L}^{-1}}\lesssim\epsilon^{-1}$, but for $\epsilon>0$ sufficiently small, one of these conditions must hold, proving the contradiction.
\end{proof}

We note a simplification compared to the analogous ODE hyperbolicity lemma proved in \cite{AlbrittonBedrossianNovack24}. By more fully exploiting the lower order term $\eta\tau^2\norm{f}$, we are able to improve from an $\abs\tau+\eta\tau^2\lesssim\epsilon$ bound to a $\tau\lesssim (1+\eta)^{-1/2}\epsilon$ bound, removing the need for the Kawashima dissipativity theory of \cite{KawashimaShizuta85}.

We can now count the number of stable and unstable eigenvalues of this system. For any linear operator $A$, we have the stable eigenspace $\ca S(A)$ corresponding to all eigenvectors with negative real part, and the unstable eigenspace $\ca U(A)$, corresponding to the eigenvectors with positive real part. Writing $r=\dim\VV_N$, we have the following lemma.

\begin{lemma}
	\label{stable_unstable_dimension}
	Let $\kappa>0$ be sufficiently small, let $\epsilon\in(0,\epsilon_0(\kappa))$ for $\epsilon_0(\kappa)>0$ sufficiently small, and let $N\ge N_0(\epsilon)$ be sufficiently large. Then for any $\eta>0$, we have unstable and stable dimensions $\dim\ca U(\mathbb{A}_-) = r+3$ and $\dim\ca S(\mathbb{A}_+) = r+1$.
\end{lemma}
\begin{proof}
	It will suffice to prove that
	\begin{align}
		\label{stable_unstable_dimension_reduction}
		\dim \ca U(\mathbb{A}_-) = r + \dim\ca U(J'(\mathfrak u_-) - \sfs),
		\quad
		\dim\ca S(\mathbb{A}_+) = r+\dim\ca S(J'(\mathfrak u_+)-\sfs)
	\end{align}
	from the eigenvalue computations in the proof of Lemma \ref{Lemma_6.1}. 
	Expressing $w = \p_x\mathfrak{v}$ in the ODE \eqref{rewritten_second_order_ODE}, we introduce the change of variables
	\begin{align*}
		\ca F = \begin{pmatrix} \mathfrak{u} \\ \mathfrak{v} \\ w \end{pmatrix} \mapsto 
		\begin{pmatrix} u \\ \tilde{\mathfrak{v}} \\ \tilde w\end{pmatrix}
		= \tilde{\ca F},
	\end{align*}
	where 
	\begin{align*}
		\tilde{\mathfrak{v}} = (I-\PP_{\mathfrak{u}_\pm,\kappa}^{(N)})f = \mathfrak{v} + L_2^{-1} L_1 \mathfrak{u}
	\end{align*}
	where we have used \eqref{microscopic_projections_reformulation}, and where we rescale the variable $w$ corresponding to $\p_x\tilde{\mathfrak{v}}$ by a factor of $\eta^{1/2}$, giving
	\begin{align*}
		\tilde w = \eta^{1/2} w + \eta^{-1/2} L_2^{-1} L_1\left( A_{00}\mathfrak{u} + A_{01}\mathfrak{v}\right).
	\end{align*}
	We then express this as $\tilde{\ca F} = P\ca F$ for the transformation matrices
	\begin{align*}
		P
		&=
		\begin{pmatrix} I & 0 & 0 
			\\
			L_1^{-1}L_0 & I &0
			\\
			\eta^{-\frac12}L_1^{-1}L_0 A_{00} & \eta^{-\frac12}L_1^{-1}L_0 A_{01} & \eta^{\frac12}I
		\end{pmatrix},
		\\
		P^{-1}&=
		\begin{pmatrix} I & 0 & 0
			\\
			- L_1^{-1}L_0 & I & 0
			\\
			\eta^{-1} L_1^{-1}L_0(A_{01}L_1^{-1}L_0-A_{00})
			&
			-\eta^{-1}L_1^{-1}L_0 A_{01} & \eta^{-\frac12}I
		\end{pmatrix},
	\end{align*}
	so we can compute the matrix
	\begin{align*}
		\eta^{-1/2}\tilde{\mathbb{A}}_{\pm}
		&=
		P\mathbb{A}_\pm P^{-1}
		\\
		&=
		\eta^{-\frac12}
		\begin{pmatrix}
			0 & 0 & 0 \\ 0 & 0 & I \\ 0 & -L_1 & 0
		\end{pmatrix}
		+
		\eta^{-1}
		\begin{pmatrix}
			A_{00} - A_{01}L_1^{-1}L_0 & A_{01} & 0 
			\\ 0 & 0 & 0 \\
			0 & 0 & L_1^{-1}L_0 A_{01} + A_{11}
		\end{pmatrix}
		\\
		&\qquad
		+
		\eta^{-3/2}
		\begin{pmatrix}
			0 & 0& 0 \\ 0 & 0& 0 \\ R_0 & R_1 & 0
		\end{pmatrix}
		\\
		&=:
		\eta^{-1/2}\left( \mathbb{A}_{\pm,0} + \eta^{-1/2}\mathbb{A}_{\pm,1}
		+ \eta^{-1}\mathbb{A}_{\pm,2}\right)
	\end{align*}
	for second order terms $R_0,R_1$ which we do not make explicit, and taking $\eta\to\infty$, we consider this as an analytic perturbation problem in $\eta^{-1/2}$. We first compute the spectrum of $\mathbb{A}_{\pm,0}$, but by the coercivity estimate \eqref{uncentred_discretized_coercivity_estimate}, we can compute that the spectrum of $-L_1$ lies in $\set{z\in\mathbb{C}\,:\,\RE z\ge \de}$ for some $\de>0$. To calculate the spectrum of $\mathbb{A}_{\pm,0}$, using the strategy of \cite{AlbrittonBedrossianNovack24}, we note that the matrix
	\begin{align*}
		\tilde{\mathbb{A}}_{\pm,0} = \begin{pmatrix} 0 & I \\ -L_1 & 0\end{pmatrix}
	\end{align*}
	has  square
	\begin{align*}
		\tilde{\mathbb{A}}_{\pm,0} ^2=\begin{pmatrix} -L_1 & 0 \\ 0 & -L_1\end{pmatrix},
	\end{align*}
	and using the spectral mapping identity $\sigma(\tilde{\mathbb{A}}_{\pm,0} ^2) = \sigma(\tilde{\mathbb{A}}_{\pm,0} )^2$, 
	we get inclusion of the spectrum
	\begin{align}
		\label{Galerkin_0th_order_spectrum_hyperbolic}
		\sigma\left(	\tilde{\mathbb{A}}_{\pm,0} \right)
		\subset\set{z\in\mathbb{C}\,:\,\abs{\RE z}\ge\sqrt\de},
	\end{align}
	so that the matrix $\tilde{\mathbb{A}}_{\pm,0}$ is hyperbolic. Using the conjugation identity
	\begin{align*}
		\begin{pmatrix} I & 0 \\ 0 & -I\end{pmatrix}
		\begin{pmatrix} 0 & I \\ -L_1 & 0\end{pmatrix}
		\begin{pmatrix} I & 0 \\ 0 & -I\end{pmatrix}
		=
		-\begin{pmatrix} 0 & I \\ -L_1 & 0\end{pmatrix},
	\end{align*}
	we see that the matrices $\tilde{\mathbb{A}}_{\pm,0} $ and $- \tilde{\mathbb{A}}_{\pm,0}$ are conjugate, so that we can compute the dimensions
	\begin{align*}
		\dim\ker\mathbb{A}_{\pm,0} = 3,
		\qquad\dim\mathcal{S}\left( \mathbb{A}_{\pm,0}\right)
		=
		\dim\mathcal{U}\left( \mathbb{A}_{\pm,0}\right)
		=r=\dim\VV_N.
	\end{align*}
	
	We can now apply standard analytic perturbation theory near the semisimple $0$ eigenvalue of $\mathbb{A}_{\pm,0}$. Letting $P_0$ be the projection onto the kernel of $\mathbb{A}_{\pm,0}$ and writing $\tilde\lambda_i$ for $1\le i\le 3$ for the eigenvalues of $\tilde{\mathbb{A}}_\pm$ near $0$, using Theorem II.2.3 of \cite{Kato95} we derive the asymptotics
	\begin{align*}
		\tilde\lambda_i = \eta^{-1/2}\tilde\lambda_i^{(1)} + o(\eta^{-1/2}),\qquad 1\le i \le 3,
	\end{align*}
	where $\tilde\lambda_i^{(1)}$ denote the eigenvalues of the matrix
	\begin{align*}
		\tilde{T}^{(1)} := \Pi_0 \mathbb{A}_{\pm,1}\Pi_0 
		&=
		A_{00} - A_{01} L_1^{-1}L_0
		\\
		&
		=
		\PP(\xi_1-\sfs)\PP_{\mathfrak{u}_\pm,\kappa}^{(N)} = {J^{(N)}_\kappa}'(\mathfrak{u}_\pm)-\sfs,
	\end{align*}
	where we have used \eqref{microscopic_projections_reformulation} and the definition of the approximate flux \eqref{approximate_convection} to simplify this expression. By Lemma \ref{convergence_of_coefficients}, for $N\ge 3$ sufficiently large, the real parts of the eigenvalues of this matrix are nonzero and have the same sign as that of $J'(\mathfrak{u}_\pm)-\sfs$ so that \eqref{stable_unstable_dimension_reduction} holds, proving the lemma for $\eta>0$ sufficiently large, and by the hyperbolicity result of Lemma \ref{ODE_hyperbolicity}, this proves the lemma for all $\eta>0$.
\end{proof}

\subsection{Existence for the linearized problem}

We would now like to use the stable-unstable dimension calculations in Lemma \ref{stable_unstable_dimension} and the uniform estimates from Proposition \ref{Galerkin_macroscopic_stability} to construct solutions to \eqref{Galerkin_problem}, however we will first require the following lemma, on conjugation of non-autonomous equations. This was essentially proved in Lemma 2.6 of \cite{MetivierZumbrun04}, but we provide here a self-contained argument.

\begin{lemma}[Conjugation Lemma]
	\label{conjugation_lemma}
	Consider the ODE
	\begin{align}
		\label{conjugation_lemma_nonautonomous}
		u'(x) + A(x)u(x) = h(x),\qquad x\in[0,\infty),\qquad u(x)\in\RR^n
	\end{align}
	for $A,h\in C^1([0,\infty))$, and suppose we have
	\begin{align}
		\label{conjugation_lemma_force_convergence}
		\abs{A(x) - A_\infty}\le c_0 e^{-\de x}
	\end{align}
	for some matrix $A_\infty$ and constants $c_0,\de>0$. Then for any $\theta\in(0,\de)$, there exists an invertible matrix transformation $T\in C^1([0,\infty))$ satisfying the bound
	\begin{align*}
		\abs{T(x) - I}\le c_1 e^{-\theta x}
	\end{align*}
	for some $c_1>0$, with $T(x)^{-1}$ uniformly bounded on $[0,\infty)$, such that for the transformed variable $z(x)$ defined by $u(x) = T(x)z(x)$, the system \eqref{conjugation_lemma_nonautonomous} is equivalent to the autonomous problem
	\begin{align}
		\label{conjugation_lemma_autonomous}
		z'(x) + A_\infty z(x) = T(x)^{-1}h(x).
	\end{align}
\end{lemma}
\begin{proof}
	To find a transformation $T$ such that the formulations \eqref{conjugation_lemma_nonautonomous} and \eqref{conjugation_lemma_autonomous} are equivalent, we must solve the matrix-valued ODE
	\begin{align}
		\label{gap_lemma_ODE}
		T'(x) + A(x)T(x) - T(x)A_\infty = 0
	\end{align}
	with asymptotic boundary condition $T(+\infty) = I$. We rewrite this problem as
	\begin{align*}
		0 &= T'(x) + [A_\infty,T(x)] + (A(x) - A_\infty) 
		\\
		&
		=:
		T'(x)+ \ca L\, T(x)+ \tilde A(x).
	\end{align*}
	Choosing $\varkappa\in(\theta,\de)$ such that $\sigma(\ca L)\cap(\varkappa+i\RR)=\emptyset$, and defining $\Pi_-$ and $\Pi_+$ as the eigenprojections on $\RR^{n\times n}$ corresponding to eigenvalues in 
	\begin{align*}
		\set{\mu\in\mathbb{C}\,:\, \RE\mu<\varkappa}
	\end{align*}
	and
	\begin{align*}
		\set{\mu\in\mathbb{C}\,:\, \RE\mu>\varkappa}
	\end{align*}
	respectively, we can then define the operator
	\begin{align*}
		\ca R[T](x)
		=
		I
		-
		\int_0^x\! e^{(y-x)\ca L}\operatorname{\Pi}_+\!\left[ \tilde A(y) T(y)\right]\dd y
		+
		\int_x^\infty\! e^{(y-x)\ca L}\operatorname{\Pi}_-\!\left[ \tilde A(y) T(y)\right]\dd y
	\end{align*}
	by a standard hyperbolic variation of constants formula (cf. Chapter 9 of \cite{Teschl12}),
	and close the fixed point problem
	\begin{align}
		\label{gap_lemma_fixed_point}
		\ca R[T] = T,
	\end{align}
	which we see is equivalent to \eqref{gap_lemma_ODE} whenever the integrals converge. We then show that $\ca R$ is a contraction mapping on the set
	\begin{align}
		\label{conjugation_lemma_contraction_set}
		\ca E = \set{T\in C([0,\infty),\RR^{n\times n})\,\Big|\,\abs{T(x)-I}\le c_1e^{-\theta x}}
	\end{align}
	where $0<\theta<\varkappa$, and take $c_1>0$ sufficiently small. For $T\in\ca E$ we get the estimate
	\begin{align*}
		\abs{\ca R[T](x)-I}
		\le
		\frac{C c_0c_1}{\de+\theta-\varkappa}\left( e^{-\varkappa x} + e^{-(\de+\theta)x}\right)
	\end{align*}
	and therefore, for
	\begin{align*}
		c_0\le\frac{\de+\theta-\varkappa}{2C}
	\end{align*}
	we have that $\ca R$ maps $\ca E$ to itself. Defining the metric
	\begin{align*}
		d_{\ca E}(T_1,T_2) = \sup_{x\in[0,\infty)}e^{\theta x}\abs{T_1(x) - T_2(x)}
	\end{align*}
	on $\ca E$ and using linearity of $\ca R-I$, we then see that $\ca R$ is a contraction mapping on the complete metric space $(\ca E,d_{\ca E})$ for 
	\begin{align*}
		c_0\le\frac{\de+\theta-\varkappa}{4C},
	\end{align*}
	so it remains to construct $T$ without this smallness condition, but under the transformation $x\mapsto y+M$, the constant $c_0$ will be replaced by $c_0 e^{-\de M}$, so that taking $M\ge0$ sufficiently large and solving \eqref{gap_lemma_ODE} on the interval $[0,M]$ then gives us our construction.
	
	To prove invertibility of $T$, we see that $T^{-1}(x)$ exists with uniform bounds for $x\in[M,\infty)$ for $M>0$ sufficiently large by the condition in \eqref{conjugation_lemma_contraction_set}, and by solving the equation
	\begin{align*}
		\frac{\mrm d}{\mrm d x}\det T(x) = \det T(x)\operatorname{tr}\left(
		T(x)^{-1}T'(x)
		\right)
		= \det T(x)\operatorname{tr}\left(
		A_\infty-A(x)
		\right),
	\end{align*}
	uniform boundedness of $A-A_\infty$ then proves the claim for all $x\in[0,\infty)$.
\end{proof}

We can now prove Proposition \ref{Galerkin_proposition}, and explicitly construct solutions to the problem \eqref{Galerkin_problem}.

\begin{proof}[Proof of Proposition \ref{Galerkin_proposition}]
	We use the reformulation of the equation \eqref{Galerkin_problem} as the first-order equation 
	\begin{align*}
		-\p_x\ca F(x) + \mathbb{A}(x)\ca F(x) = \ca G,\qquad\ca F = (\mathfrak{u},\mathfrak{v},w)^\top,\quad\ca G = (0,0,z)^\top
	\end{align*}
	defined in
	\eqref{rewritten_second_order_ODE}, with the boundary condition $\ell_\epsilon\cdot U(0)=d$, and we wish to replace the matrices $\mathbb{A}(x)$ with their endpoints $\mathbb{A}_\pm$ by applying Lemma \ref{conjugation_lemma}. We notice that the $x$ dependence in $\mathbb{A}(x)$ only arises from the collision term $L_{\uns(x),\kappa}^{(N)}$, so that the viscous shock asymptotics from Lemma \ref{ABN24_Prop_B.2} then give the bound 
	\begin{align*}
		\abs{\mathbb{A}(-\abs x)-\mathbb{A}_-}\le c_0 e^{-\epsilon\de\abs{x}},
		\qquad
		\abs{\mathbb{A}(\abs x)-\mathbb{A}_+}\le c_0 e^{-\epsilon\de\abs{x}},\qquad x\to\infty
	\end{align*}
	for some constants $c_0,\de>0$. Therefore, by Lemma \ref{conjugation_lemma} there exist $C^1$ uniformly bounded transformations $T_-$ and $T_+$, defined on $[0,\infty)$ and $(-\infty,0]$ respectively, such that the change of coordinates $\ca F = T_\pm Z$ and $\ca G = T_\pm H$ on $\set{\pm x\ge 0}$ converts the equation \eqref{rewritten_second_order_ODE} to the equivalent autonomous problems
	\begin{align}
		\label{autonomous_two-sided_ODE}
		\p_x Z_-(x) - \mathbb{A}_- Z_-(x) = H_-,\qquad
		\p_x Z_+(x) - \mathbb{A}_+ Z_+(x) = H_+,\qquad
	\end{align}
	where we must additionally impose the matching condition
	\begin{align}
		\label{Galerkin_matching_condition}
		T_-(0)Z_-(0) = T_+(0)Z_+(0)
	\end{align}
	so that the solutions $Z_\pm$ can be glued at $x=0$, and the boundary condition
	\begin{align}
		\label{Galerkin_boundary_condition}
		\ell_\epsilon\cdot \PP\big( T_\pm(0) Z_\pm(0)\big)= d.
	\end{align}
	
	We denote by $\ca S(\mathbb{A}_\pm)$ and $\ca U(\mathbb{A}_\pm)$ the stable and unstable subspaces of $\mathbb{A}_\pm$ and define their associated spectral projections $\Pi_{\ca S(\mathbb{A}_\pm)}$ and $\Pi_{\ca U(\mathbb{A}_\pm)}$. We assume for now that the microscopic source term $z$, and therefore $\ca G$, are Schwartz-class in $x$, and using hyperbolicity of the matrices $\mathbb{A}_\pm$ from Lemma \ref{ODE_hyperbolicity}, we see that the problem \eqref{autonomous_two-sided_ODE} has exponentially-decaying solutions of the form
	\begin{align*}
		Z_-(x) &= e^{x\mathbb{A}_-}Z_{0,-} - \int_x^0\! e^{(x-y)\mathbb{A}_+}\Pi_{\ca U(\mathbb{A}_-)} H_-(y)\dd y
		-
		\int^x_{-\infty}\! e^{(x-y)\mathbb{A}_-}\Pi_{\ca S(\mathbb{A}_-)}H_-(y)\dd y
		\\
		&=e^{x\mathbb{A}_-}Z_{0,-} - I_-(x) + I\!I_-(x)
	\end{align*}
	for $x\le0$, and
	\begin{align*}
		Z_+(x) &= e^{x\mathbb{A}_+}Z_{0,+} + \int_0^x\! e^{(x-y)\mathbb{A}_+}\Pi_{\ca S(\mathbb{A}_+)} H_+(y)\dd y
		-
		\int_x^\infty\! e^{(x-y)\mathbb{A}_+}\Pi_{\ca U(\mathbb{A}_+)}H_+(y)\dd y
		\\
		&= e^{x\mathbb{A}_+}Z_{0,+} + I_+(x) - I\!I_+(x)
	\end{align*}
	for $x\ge0$ for some $(Z_{0,-}, Z_{0,+}) \in \ca U(\mathbb{A}_-)\times\ca S(\mathbb{A}_+)$. Writing $I\!I_{0,\pm} = I\!I_\pm(0)$, we see that the matching and boundary conditions \eqref{Galerkin_matching_condition} \eqref{Galerkin_boundary_condition} are equivalent to the conditions
	\begin{align}
		\label{Galerkin_matching_condition_2}
		T_+(0) \left(Z_{0,+} - I\!I_{0,+}\right) - T_-(0)\left( Z_{0,-} + I\!I_{0,-} \right)=0
	\end{align}
	and
	\begin{align}
		\label{Galerkin_boundary_condition_2}
		\ell_\epsilon\cdot\PP T_+(0)Z_{+,0} = d+\ell_\epsilon\cdot\PP T_+(0)I\!I_{+,0}.
	\end{align}
	Solving for $(Z_{0,+}, Z_{0,-}) \in \ca S(\mathbb{A}_+)\times\ca U(\mathbb{A}_-)$ satisfying these conditions, we observe that these impose a $3+2r+1$-dimensional constraint on $(Z_{0,+}, Z_{0,-}) $, but by Lemma \ref{stable_unstable_dimension} we have the dimension
	\begin{align*}
		\dim \ca S(\mathbb{A}_+)\times\ca U(\mathbb{A}_-) = 3 + 2r + 1,
	\end{align*}
	so it remains to show that this problem is non-degenerate in $(Z_{0,+}, Z_{0,-}) $. If we assume that $z=0$ and $d=0$ so that the linear equations \eqref{Galerkin_matching_condition_2} and \eqref{Galerkin_boundary_condition_2} become homogeneous, we have by the stability theory in Proposition \ref{Galerkin_macroscopic_stability} that $f=0$ and therefore $Z_\pm=0$ and $Z_{0,\pm}=0$, so that the problem \eqref{Galerkin_problem} can be solved for any $d\in\RR$ and any Schwartz class $z\in\mathscr{S}(\RR,\VV_N)$. We can then use Proposition \ref{Galerkin_macroscopic_stability} to get the estimates \eqref{Galerkin_macroscopic_estimate}, and we extend to $z\in H^2_\epsilon\VV^{-1}_\kappa$ by continuity.
\end{proof}

\subsection{Convergence by compactness}

It now remains to show convergence of the Galerkin scheme to the original problem. 

\begin{proof}[Proof of Proposition \ref{Lepsilonkappa_invertibility}]
	We first construct $f\in H^2_\epsilon\HH^1_\kappa$ for source terms $z\in H^2_\epsilon\VV_{N_0}$ for $N_0$ arbitrarily large. By Proposition \ref{Galerkin_proposition}, we have a solution $f_{N,\eta}\in H^2_\epsilon\HH_N$ to the problem
	\begin{align}
		\label{Lepsilonkappa_invertibility_eq1}
		\left( A^{(N)} - \sfs\right)\p_x f_{N,\eta} - \eta\p_x^2 f_{N,\eta} - L_{\uns,\kappa}^{(N)}f_{N,\eta} = z,\quad \ell_\epsilon\cdot\PP f_{N,\eta}(0)=d,
	\end{align}
	and using the uniform estimate \eqref{Galerkin_macroscopic_estimate} we can then take a weakly converging subsequence $f_{N,\eta}\rightharpoonup f$ as $N\to\infty$ and $\eta\to0$, converging in the space $H^2_\epsilon\HH^1_\kappa$. The boundary condition $\ell_\epsilon\cdot\PP f(0)=d$ is easily seen to hold by weak convergence in $H^1_\epsilon$, so it remains to show convergence for the rest of the equation. By applying the finite projection $\Pi_{N_1}$ for $N_1\ge N_0$ to \eqref{Lepsilonkappa_invertibility_eq1} we get the equation
	\begin{align*}
		\Pi_{N_1}\left(\xi_1-\sfs\right)\p_x f_{N,\eta} &- \eta\Pi_{N_1}\p_x^2f_{N,\eta}
		\\
		&-
		\Pi_{N_1}\left[ Q_\kappa\left(M_{\uns,\kappa}^{(N)},f_{N,\eta}\right)
		+
		Q_\kappa\left(f_{N,\eta},M_{\uns,\kappa}^{(N)}\right)\right] = z.
	\end{align*}
	The first two terms are easily seen to converge weakly in $L^2_\epsilon\HH^0$, so it remains to show convergence of the collision term. But since $M_{\uns,\kappa}^{(N)}\to M_{\uns}$ strongly in $L^\infty_x\HH^1_\kappa$ by Lemma \ref{discretized_Maxwellians_lemma} and $\Pi_{N_1} Q_\kappa$ is a bilinear operator mapping to a finite-dimensional space, the trilinear estimate from Lemma \ref{discretized_collisions_lemma} and standard results on products of weakly and strongly converging subsequences are enough to show that $f$ solves $\Pi_{N_1}\ca L_{\epsilon,\kappa}f=z$ for all $N_1\ge N_0$. Using that $\ca L_{\epsilon,\kappa}f\in H^2_\epsilon\HH^{-1}_\kappa$ by Lemma \ref{Qkappa_trilinear_estimate} and boundedness of $\xi_1:\HH^1_\kappa\to\HH^{-1}_\kappa$, we can take limits in $\mathsf{H}^{-2}_{-2}$ as $N_1\to\infty$, proving that $\ca L_{\epsilon,\kappa}f=z$.
	To prove the result for general $z\in H^2_\epsilon\VV^{-1}_\kappa$, we use boundedness of the inverted operator
	\begin{align*}
		\left( \ca L_{\epsilon,\kappa},\left.\ell_\epsilon\cdot\PP\right|_{x=0}\right)^{-1}
		:\Big( H^2_\epsilon\VV_{N_0}\times\RR\Big)\subset \Big(H^2_\epsilon\VV^{-1}_\kappa\times\RR\Big)\to H^2_\epsilon\HH^1_\kappa
	\end{align*}
	by Proposition \ref{macro_estimate_proposition} and density of the discretized spaces $\VV_{N_0}\subset\VV^{-1}_\kappa$ as $N_0\to\infty$, and extend by continuity.
\end{proof}

We would now like to take the limit $\kappa\to0$ using the uniform estimates proved in Proposition \ref{macro_estimate_proposition}, but we must first remove the dependence of $\epsilon<\epsilon_0(\kappa)$ on $\kappa.$ To do this, we use the continuity lemma developed in \cite{AlbrittonBedrossianNovack24}.

\begin{lemma}[Lemma 5.11 of \cite{AlbrittonBedrossianNovack24}]
	\label{ABN24_continuity_lemma}
	Suppose $X$ and $Y$ are Banach spaces, and let $A(t):[0,1]\to B(X,Y)$ be a norm-continuous family of bounded linear operators. Suppose we have a uniform coercivity estimate
	\begin{align*}
		\norm{x}_X\le C\norm{A(t)x}_Y
	\end{align*}
	for all $t\in[0,1]$ and $x\in X$. Then if $A(0)$ is surjective, $A(t)$ is surjective for all $t\in[0,1]$.
\end{lemma}

Using this lemma and our previous uniform estimates, we can extend the range of $\epsilon>0$ to be uniform in $\kappa\in(0,\kappa_0)$, with the following lemma.

\begin{lemma}
	\label{extending_range_of epsilon}
	Let $\epsilon\in(0,\epsilon_0)$ and $\kappa\in (0,\kappa_0)$ for $\epsilon_0>0$ and $\kappa_0>0$ sufficiently small. For any $z\in H^2_\epsilon\VV^{-1}_\kappa$ and $d\in\RR$, the problem
	\begin{align*}
		\ca L_{\epsilon,\kappa}f = z,\qquad \ell_\epsilon\cdot\PP f(0)=d
	\end{align*}
	admits a unique solution $f\in H^2_\epsilon\HH^1_\kappa$.
\end{lemma}
\begin{proof}
	We follow the strategy of \cite{AlbrittonBedrossianNovack24}. We first define the spaces
	\begin{align*}
		X_{\sfs} = \set{ f\in H^2_\epsilon \HH^1_\kappa\, :\, (\xi_1-\sfs)\p_x f\in H^2_\epsilon\HH^{-1}_\kappa,\ \PP(\xi_1-\sfs) f=0}
	\end{align*}
	and
	\begin{align*}
		Y = H^2_\epsilon\VV^{-1}_\kappa\times\RR,
	\end{align*}
	where we use that the norms $H^2_\epsilon$ are equivalent for $\epsilon\in(\epsilon_0',\epsilon_0)$ for $0<\epsilon_0'<\epsilon_0$.
	Defining the shift operator $\tau^{\sfs}f(x,\xi) = f(x,\xi+\sfs e_1)$, we can set
	\begin{align*}
		A(\epsilon) = \left(\tau^{\sfs }\bar{\ca L}_{\epsilon,\kappa}\tau^{-\sfs },\left.\ell_{\epsilon}\cdot\PP\right|_{x=0}\circ \tau^{-\sfs }\right),
	\end{align*}
	and by re-expressing
	\begin{align*}
		\tau^{\sfs}\ca L_{\epsilon,\kappa}\tau^{-\sfs }f
		=
		\xi_1\p_x f - Q_\kappa(\tau^{\sfs } M_{\uns},f) - Q_\kappa(f,\tau^{\sfs } M_{\uns}),
	\end{align*}
	we see that $A(\epsilon)$ is bounded $X_0\to Y$ by Lemma \ref{Qkappa_trilinear_estimate}. Continuity of the shock speed $\sfs$ in $\epsilon$ is standard (cf. Theorem 5.14 of \cite{HoldenRisebro_book}), while continuity of $\ell_\epsilon$ was proved in Lemma \ref{Lemma_6.1}, and continuity of $M_{\uns}$ follows from continuity in $\epsilon$ of $\uns$, from Lemma \ref{ABN24_Prop_B.2}, so we see that $A(\epsilon):X_0\to Y$ is a norm-continuous family of operators. By Proposition \ref{macro_estimate_proposition} we have the uniform bound
	\begin{align*}
		\norm{\tau^{\sfs}f}_{X_0}\le \frac{C}{\epsilon}\norm{A(\epsilon)\tau^{\sfs}f}_Y
	\end{align*}
	for all $f\in X_\sfs$ and for small $\epsilon\in(0,\epsilon_0)$ independently of $\kappa$, therefore by Lemma \ref{extending_range_of epsilon}, the operator $A(\epsilon)$ is invertible for this range, proving the lemma.
\end{proof}

We can now prove the main proposition of the article.

\begin{proof}[Proof of Proposition \ref{right_inverse_proposition}]
	We first show that we can solve the problem
	\begin{align}
		\label{H2_problem}
		\ca L_{\epsilon}f =(\xi_1-\sfs)\p_x f - L_{\uns}f= z,\qquad\ell_\epsilon\cdot\PP f(0)=d
	\end{align}
	for $z\in H^2_\epsilon\VV^{-1}_\kappa$ and $d\in\RR$, by taking limits $\kappa\to 0$ in Lemma \ref{extending_range_of epsilon}. If we let $f_\kappa\in H^2_\epsilon\HH^1_\kappa$ solve
	\begin{align*}
		(\xi_1-\sfs)\p_x f_\kappa - L_{\uns,\kappa}f_\kappa = z
	\end{align*}
	where we use that $\HH^{-1}_\kappa\supset\HH^{-1}$ by construction, we have uniform bounds
	\begin{align*}
		\norm{f_\kappa}_{H^2_\epsilon\HH^1}
		\le \norm{f_\kappa}_{H^2_\epsilon\HH^1_\kappa}
		&
		\le\frac{C}{\epsilon} \norm{z}_{H^2_\epsilon\HH^{-1}_\kappa}+C\abs d
		\\
		&
		\le \frac{C}{\epsilon} \norm{z}_{H^2_\epsilon\HH^{-1}} + C\abs d
	\end{align*}
	by Proposition \ref{macro_estimate_proposition}, we can take a weakly converging subsequence $f_\kappa\rightharpoonup f$ in $H^2_\epsilon\HH^1$.
	We immediately have convergence $\ell_\epsilon\cdot\PP f_\kappa(0)\to d$, as well as convergence $(\xi_1-\sfs)\p_x f_\kappa\to(\xi_1-\sfs)\p_x f\in\mathscr{D}'(\RR_x\times\RR^3_\xi)$ in the space of distributions. To show convergence of the remaining term, we write out
	\begin{align}
		\label{kappa_convergence_rate}
		L_{\uns,\kappa}f_\kappa = L_{\uns}f_\kappa
		+ \kappa \left( Q_{s,2-2s}(M_{\uns},f_\kappa)+Q_{s,2-2s}(f_\kappa,M_{\uns})\right)
	\end{align}
	but since
	\begin{align*}
		&
		\norm{\Mref^{-1/2} Q_{s,2-2s}(M_{\uns},f_\kappa)}_{N^{s,2-2s}} + 
		\norm{\Mref^{-1/2} Q_{s,2-2s}(M_{\uns},f_\kappa)}_{N^{s,2-2s}}
		\\
		&\qquad\qquad\qquad\qquad\qquad\qquad\qquad\qquad
		\le C\norm{\Mref^{-1/2}f_\kappa}_{N^{s,2-2s}}
		\le C\kappa^{-1/2}\norm{f_\kappa}_{\HH^1_\kappa},
	\end{align*}
	using \eqref{GS11_general_trilinear_bound}, we get that the second term in \eqref{kappa_convergence_rate} is of order $\kappa^{1/2}$ and is therefore negligible, while we will have weak convergence $L_{\uns}f_\kappa\rightharpoonup L_{\uns}f$ in $\HH^{-1}$ by continuity of the operator $L_{\uns}$, and therefore we have a solution $f\in H^2_\epsilon\HH^1$ to \eqref{H2_problem}.
	
	Assuming now that $z\in H^k_\epsilon\VV^{-1}$ for some $k\ge 3$, we show that $f\in H^k_\epsilon\HH^1$  using the uniform estimates in Proposition \eqref{macro_estimate_proposition} and a spatial regularity bootstrap. We let $\eta(x)\ge0$ be smooth and symmetric with compact support and defining the mollifiers $\eta_\de(x) = \delta^{-1}\eta(x/\de)$. Assuming by induction that $f\in H^{k-1}_\epsilon\HH^1$, we define the spatial mollification $f_\de = \eta_\de* f$, by convolving \eqref{H2_problem} with $\eta_\de$ we get the equation
	\begin{align*}
		(\xi_1-\sfs)\p_x f_\de - L_{\uns}f_\de = z + \ca R
	\end{align*}
	with microscopic remainder term
	\begin{align*}
		\ca R = \int_{\RR}\! \Big[ Q\left(M_{\uns(y)}-M_{\uns(x)},f(y)\right)  +  Q\left(f(y),M_{\uns(y)}-M_{\uns(x)}\right)\Big]\eta_\de(x-y)\dd y,
	\end{align*}
	and we would like to bound this uniformly in $H^k_\epsilon\HH^{-1}$ as $\de\to 0$. Investigating the effect of $\p_x$-derivatives on this remainder, we note that if $\eta_\de$ has more than one derivative, then we can interchange the spatial integral with the integral in the definition of the collision operator and integrate by parts in $y$, so it will suffice to bound
	\begin{align*}
		\norm{\p_x^j M_{\uns(x)}\eta_\de(x-y)}_{\HH^1}\le C_k\de^{-1}1_{\abs{x-y}\le C\de}
	\end{align*}
	for $1\le j\le k$, and
	\begin{align*}
		\norm{\left( M_{\uns(y)}-M_{\uns(x)}\right)\eta_\de(x-y)}_{\HH^1}\le C 1_{\abs{x-y}\le C\de},
	\end{align*}
	using smoothness of the Maxwellian manifold $\ca M_{\epsilon_0}\subset\HH^1$ from Lemma \ref{Maxwellian_manifold_smoothness} and the uniform viscous shock bounds from Lemma \ref{ABN24_Prop_B.2}, giving the bound
	\begin{align*}
		\norm{\p_x^k\ca R}_{L^2_\epsilon\HH^{-1}}\le C_{k}\norm{f}_{H^{k-1}},
	\end{align*}
	where we ignore the dependence of the constant $C_k$ on $\epsilon$.
	Applying this to Proposition \ref{macro_estimate_proposition}, we get the bound
	\begin{align*}
		\norm{f_\de}_{H^k_\epsilon\HH^1}\le C_{k}\left( \norm{z}_{H^k_\epsilon\HH^{-1}_\kappa} + \abs{\ell_\epsilon\cdot \PP f_\de(0)}\right)
	\end{align*}
	but since $f\in H^1_\epsilon$ is continuous in $x$, the remainder term $\ell_\epsilon\cdot\PP f_\de(0)$ is uniformly bounded, so by taking a weakly converging subsequence $f_\de\rightharpoonup f$ in $H^k_\epsilon\HH^1$ and using that $f_\de\to f$ in $H^{k-1}_\epsilon\HH^1$, we get that in fact $f\in H^k_\epsilon\HH^1$. If we set $d=0$ in the problem \eqref{H2_problem}, then this defines the right inverse $f = \ca L_\epsilon^\dagger z$, proving the proposition.
\end{proof}

	\appendix
	\section{Semigroup extension theory for linear collisions} 
	\label{Appendix_C}

	We now provide a proof of Lemma \ref{HTT20_Lemma}. In this section, we impose conditions \eqref{Hypothesis1} and \eqref{Hypothesis2} on the collision kernel $B_{s,\ga}$, but for simplicity of notation, we do not assume \eqref{Hypothesis3}, namely we let $s\in(0,1)$ and $\ga\in\RR$ vary, and we will give more precise conditions when needed. For consistency, we fix an angular collision kernel $b_s(\cos\theta)$ for each $s\in(0,1)$, which by symmetrization, we can assume to be supported in $\theta\in(0,\pi/2]$. Invertibility of the operators $\Lref_{s,\ga}$ in $L^2(\br\xi^k)$ away from the kernel was proved for $s\in(0,1/2)$ and $\ga\in(0,1)$ in \cite{HerauTononTristani20}, using the semigroup extension theory for non-symmetric operators developed in \cite{Mouhot06} \cite{GualdaniMischlerMouhot17}, by decomposing the operator into a hypodissipative part and a regularizing part. We prove uniform invertibility of the regularized uncentred operators $L_{\bar u,\kappa}$ by a similar strategy. 
	
	We now define the notation we will need in this section. We fix a truncation function $\chi\in C^\infty(\RR)$ such that $1_{[-1,1]}\le\chi\le 1_{[-2,2]}$, and write $\chi_\de(x) = \chi(x/\de)$ where the regularization parameter $\de>0$ may vary, and define the truncated angular collision kernel
	\begin{align*}
		b_s^{(\de)}(\cos\theta) = b_s(\cos\theta)\chi_\de(\theta)
	\end{align*}
	and split $b_s = b_s^{(\de)} + b_s^{(\de),\mrm c}$. We then similarly split the collision kernel $B_{s,\ga}$ into singular and remainder parts
	\begin{align*}
		B_{s,\ga}^{(\de)}(r,\cos\theta) = r^\ga b_s^{(\de)}(\cos\theta),
		\quad
		B_{s,\ga}^{(\de),c}(r,\cos\theta) = r^\ga b_s^{(\de),c}(\cos\theta),
	\end{align*}
	and define the corresponding collision operators
	\begin{align*}
		Q_{s,\ga}^{(\de)}(g,f) = \int_{\RR^3\times\SSS^2}\! 	B_{s,\ga}^{(\de)} \ [ g_*' f' - g_* f]\dd\sigma\dd\xi
	\end{align*}
	and $Q_{s,\ga}^{(\de),c} = Q_{s,\ga} - Q_{s,\ga}^{(\de)}$. We then define the anisotropic norm
	\begin{align*}
		\norm{f}_{\dot H^{s,\gamma,*}}^2
		=
		\int_{\RR^3\times\RR^3\times\SSS^2}\!
		b_s^{(\de)}(\cos\theta) \Mref(\xi_*)\br{\xi_*}^{-\ga}( f'\br{\xi'}^{\ga/2} - f\br{\xi}^{\ga/2})^2\dd\sigma\dd\xi_*\dd\xi
	\end{align*}
	for $s\in (0,1)$ and $\ga\in(0,2)$ introduced in \cite{HerauTononTristani20}. For any weight function $m$, we will also define the weighted norm $\norm{f}_{H^{s,\gamma,*}(m)} = \norm{mf}_{H^{s,\gamma,*}}$.
	
	We have the following trilinear estimate on the collision operator.
	
	\begin{lemma}\label{Lemma2.3_HTT20_extended}
		Assume $k>\ga/2 + 2+\max(1/2,2s)$ and $(s,\gamma)\in(0,1/2)\times[0,2]$. Defining the weight $m = \br\xi^k$, we have the bound
		\begin{align*}
			\abs{\br{Q_{s,\gamma}(f,g),h)}_{L^2_\xi(m)}}
			&\lesssim
			\norm{f}_{L^2_\xi(\br{\xi}^{\ga/2}m)}\norm{g}_{H^{2s}_\xi(\br\xi^{\ga+2s}m)}\norm{h}_{L^2_\xi(m)}
			\\
			&\qquad +
			\norm{f}_{L^2_\xi(\br\xi^{\ga/2}m)}\norm{g}_{L^2_\xi(\br\xi^{\ga/2}m)}\norm{h}_{L^2_\xi(\br\xi^{\ga/2}m)}.
		\end{align*}
	\end{lemma}
	This generalizes Lemma 2.3 of \cite{HerauTononTristani20}, which requires $\gamma\in(0,1)$ and requires slightly more polynomial moments $k$. The arguments used for $\ga\in(0,1)$ mostly carry through unchanged for the extended range $\ga\in[0,2]$, so we write the arguments where they differ for the extended range, and refer to \cite{HerauTononTristani20} for the rest of the details of the proof.
	\begin{proof}[Proof of Lemma \ref{Lemma2.3_HTT20_extended}]
		We re-express
		\begin{align*}
			\br{ Q_{s,\gamma}(f,g),h }_{L^2_\xi(m)}
			&= \int_{\RR^3 \times \RR^3 \times \SSS^2} B_{s,\gamma}(\abs{\xi-\xi_*},\cos\theta) (f'_*g' - f_*g)\,  h\, m^2 \dd\sigma \dd\xi_* \dd\xi \\
			&= \int_{\RR^3 \times \RR^3 \times \SSS^2} B_{s,\gamma}(\abs{\xi-\xi_*},\cos\theta) (f'_* g' m' - f_* g\, m) \,h\,  m \dd\sigma \dd\xi_* \dd\xi \\
			&\quad
			+ \int_{\RR^3 \times \RR^3 \times \SSS^2} B_{s,\gamma}(\abs{\xi-\xi_*},\cos\theta) f'_* g' h\, m\, (m-m') 
			\dd\sigma \dd\xi_* \dd\xi \\
			&=: I_1+I_2.
		\end{align*}
		as in the proof of Lemma 2.3 of \cite{HerauTononTristani20}, where
		\begin{align*}
			\cos\theta = \sigma\cdot\frac{\xi-\xi_*}{\abs{\xi-\xi_*}},
		\end{align*}
		and where we write $g'$ for $g(\xi')$, and similarly write $f_*',f_*,$ and $m'$.  We note that by Theorem 1.1 of \cite{He18}, we can bound
		\begin{align*}
			I_1 = \br{ Q_{s,\gamma}(f,gm),hm }_{L^2_\xi}
			&\lesssim
			\norm{f}_{L^1_\xi(\br\xi^{\ga+2s})}
			\norm{g}_{H^{2s}_\xi(\br\xi^{\ga+2s}m)}
			\norm{h}_{L^2_\xi(m)}
			\\
			&\lesssim
			\norm{f}_{L^2_\xi(\br{\xi}^{\ga/2}m)} 
			\norm{g}_{H^{2s}_\xi(\br\xi^{\ga+2s}m)}
			\norm{h}_{L^2_\xi(m)}
		\end{align*}
		for $(s,\ga)\in(0,1)\times[0,2]$, where the second bound follows from $k>\ga/2+2+\max(1/2,2s)$. The proof of Lemma 2.3(i) of \cite{HerauTononTristani20} then gives the estimate
		\begin{align*}
			I_2
			\lesssim
			\norm{f}_{L^2_\xi(\br\xi^{\ga/2}m)}\norm{g}_{L^2_\xi(\br\xi^{\ga/2}m)}\norm{h}_{L^2_\xi(\br\xi^{\ga/2}m)}
		\end{align*}
		for $(s,\ga)\in(0,1)\times[0,2]$, which relies on cancellation estimates developed in \cite{AMUXY10_regularizing} to control the angular singularity when $s<1/2$. This concludes the proof of the lemma. The rest of the proof of Lemma 2.3 of \cite{HerauTononTristani20} can be followed line by line for the case $\ga\in[0,2]$.
	\end{proof}
	
	We use this lemma to prove the following coercivity result.
	
	\begin{lemma}\label{Lemma4.2_HTT20_extended}
		For $k>\ga/2 + 2+\max(1/2,2s)$ and $(s,\gamma)\in(0,1/2)\times[0,2]$, setting the weight $m = \br\xi^k$, we have the coercivity estimate
		\begin{align*}
			\br{\Lref_{s,\gamma} h,h}_{L^2(m)}
			\le - c_0\de^{2-2s}\norm{h}_{H^{s,\gamma,*}(m)}^2 - c_0\de^{-2s}\norm{h}^2_{L^2(\br\xi^{\gamma/2}m)}
			+
			C_\de \norm{h}_{L^2}^2,
		\end{align*}
		for any $\de>0$ sufficiently small.
	\end{lemma}
	
	This was proved in Lemma 4.2 of \cite{HerauTononTristani20} for $\ga\in(0,1)$. This proof relies on the entropy dissipation lower bound of \cite{AlexandreDesvillettesVillaniWennberg00} and estimates developed for the grazing-collision limit in \cite{He14}, and we verify that these hold for $s\in(0,1/2)$ and $\ga\in[0,2]$. Before we can prove Lemma \ref{Lemma4.2_HTT20_extended}, we need the following preparatory results.
	
	\begin{lemma}[Corollary 2.1 of \cite{AlexandreDesvillettesVillaniWennberg00}]
		\label{ADVW00_Corollary_2.1}
		Let $f,g\in L^2(\RR^3)$ be non-negative functions, then we have the bound
		\begin{align*}
			&
			\int_{\RR^3\times\RR^3\times\SSS^2}\!
			b_s^{(\de)}(\cos\theta )
			g_* (f-f')^2\dd\sigma\dd\xi_*\dd\xi
			\\
			&\qquad
			\ge
			\frac{1}{16\pi^3}
			\int_{\RR^3}\! \abs{\hat f(\zeta)}^2
			\left(
			\int_{\SSS^2}\! b_s^{(\de)}\left(\sigma\cdot\frac{\zeta}{\abs\zeta}\right)
			\left( \hat g(0) - \hat g(\zeta^-) \right) \dd\sigma
			\right)
			\dd\zeta
		\end{align*}
		where $\hat f(\zeta)$ denotes the Fourier transform of $f(\xi)$ in the velocity variable, and $\zeta^- = (\zeta - \abs\zeta\sigma)/2$.
	\end{lemma}
	
	To prove Lemma \ref{Lemma4.2_HTT20_extended}, we will need asymptotics for this estimate in $\de>0$, which we prove in the following lemma.
	
	\begin{lemma}
		\label{truncated_ADVW00_Prop_3}
		For any non-negative function $g\in L^1(\br\xi)$, we have the bound
		\begin{align*}
			\int_{\SSS^2}\! b_s^{(\de)}\!\left( \frac{\zeta}{\abs{\zeta}}\cdot\sigma\right)
			\left( \hat g(0) - \abs{\hat g(\zeta^-)}\right)\dd\sigma
			\ge  \de^{2-2s}\left( c_0\abs\zeta^{2s} - c_1\right),
		\end{align*}
		for all $\zeta\in\RR^3$, where the constants $c_0,c_1>0$ depend on $g$.
	\end{lemma}
	\begin{proof}
		Using Lemma 3 of \cite{AlexandreDesvillettesVillaniWennberg00}, we have the estimate
		\begin{align}
			\label{ADVW00_Lemma_3}
			\hat g(0) - \abs{\hat g(\zeta)} \ge C_g' \min(\abs\zeta^2,1)
		\end{align}
		for any $\zeta\in\RR^3$, proved using non-negativity $g\ge0$.
		Using that
		\begin{align*}
			\abs{\zeta^-}^2 = \frac{\abs\zeta^2}{2}
			\left(1 - \sigma\cdot\frac{\zeta}{\abs\zeta}\right)
		\end{align*}
		and defining $\cos\vartheta = \sigma\cdot\zeta/\abs\zeta$,
		we can bound
		\begin{align*}
			&
			\int_{\SSS^2}\! b_s(\cos\vartheta)\,
			\min\!\left(\frac{\abs\zeta^2}{2}(1-\cos\vartheta),1\right)\chi_\de(\theta) \dd\sigma
			\\
			&\quad=
			2\pi\int_0^{\frac\pi2}\!  b_s(\cos\vartheta)\sin(\vartheta)
			\min\!\left(\frac{\abs\zeta^2}{2}(1-\cos\vartheta),1\right)\chi_\de(\vartheta) \dd\vartheta
			\\
			&\quad\ge
			C\int_0^{\de}\!
			\vartheta^{-1-2s}\min\left(\abs\zeta^2\theta^2,1\right)\dd\vartheta
			\\
			&\quad =
			C\abs\zeta^{2s}\int_0^{\abs\zeta \delta}\! \tilde\vartheta^{-1-2s} \min(\tilde\vartheta^2,1)\dd\tilde\vartheta
			\\
			&\quad =
			\begin{cases} 
				C'\de^{2-2s}\abs\zeta^2 & \abs\zeta\le\de^{-1} \\
				C'\de^{-2s} + C''\abs\zeta^{2s} \left(1 - \abs\zeta^{-2s} \de^{-2s}\right) &
				\abs\zeta>\de^{-1}
			\end{cases}
			\\
			&\quad\ge
			\de^{2-2s}\left( c_0\abs\zeta^{2s} - c_1\right)
		\end{align*}
		where we have used the change of variables $\tilde\vartheta = \abs\zeta\vartheta$, and combining this with the bound \eqref{ADVW00_Lemma_3} proves the claim.
	\end{proof}
	
	We can now prove coercivity estimates on $\Lref_{s,\ga}$ by following arguments from the proof of Lemma 4.2 of \cite{HerauTononTristani20}, and showing that they hold for the extended range $s\in(0,1/2)$ and $\ga\in[0,2]$.
	
	\begin{proof}[Proof of Lemma \ref{Lemma4.2_HTT20_extended}]
		As in the proof of Lemma 4.2 of \cite{HerauTononTristani20}, we define $H = hm$, and compute
		\begin{align*}
			\br{Q_{s,\ga}(\Mref,h),h}_{L^2_\xi(m)}
			&=
			\int_{\RR^3\times\RR^3\times\SSS^2}\!
			B(\abs{\xi-\xi_*},\cos\theta)\left[ \Mref_*' h' - \Mref_* h\right]\, h\, m^2\dd\sigma\dd\xi_*\dd\xi
			\\
			&=
			\int_{\RR^3\times\RR^3\times\SSS^2}\!
			B(\abs{\xi-\xi_*},\cos\theta)\left[ \Mref_*' H' - \Mref_* H\right]\dd\sigma\dd\xi_*\dd\xi
			\\
			&\qquad
			+
			\int_{\RR^3\times\RR^3\times\SSS^2}\!
			B(\abs{\xi-\xi_*},\cos\theta) \Mref_*' h' h m(m-m')\dd\sigma\dd\xi_*\dd\xi
			\\
			&=
			\br{Q_{s,\ga}^{(\de)}(\Mref,H),H}_{L^2_\xi}
			+
			\br{Q_{s,\ga}^{(\de),\mrm c}(\Mref,H),H}_{L^2_\xi}
			+R,
		\end{align*}
		and we bound each term separately. The first term can be rewritten, as in Lemma 4.2 of  \cite{HerauTononTristani20}, as
		\begin{align*}
			\br{Q_{s,\ga}^{(\de)}(\Mref,H),H}_{L^2_\xi}
			&=
			\int_{\RR^3\times\RR^3\times\SSS^2}\!
			B_{s,\ga}^{(\de)}(\abs{\xi-\xi_*},\cos\theta)\Mref_* H(H'-H)\dd\sigma\dd\xi_*\dd\xi
			\\
			&= 
			-\frac12 	
			\int_{\RR^3\times\RR^3\times\SSS^2}\!
			B_{s,\ga}^{(\de)}(\abs{\xi-\xi_*},\cos\theta)\Mref_* (H'-H)^2\dd\sigma\dd\xi_*\dd\xi
			\\
			&\quad
			+\frac12 
			\int_{\RR^3\times\RR^3\times\SSS^2}\!
			B_{s,\ga}^{(\de)}(\abs{\xi-\xi_*},\cos\theta)\Mref_* 
			\left( (H')^2 - H^2\right)\dd\sigma\dd\xi_*\dd\xi
			\\
			&= -I_1 + I_2.
		\end{align*}
		We now use the cancellation lemma in \cite{AlexandreDesvillettesVillaniWennberg00}, which gives that for $(s,\ga)\in(0,1)\times(-3,2]$, we can express
		\begin{align*}
			I_2 = \frac12 \int_{\RR^3}\! \left( S_{s,\ga}^{(\de)} * H^2\right) \Mref\dd\xi,
		\end{align*}
		where
		\begin{align*}
			S_{s,\ga}^{(\de)}(z)
			=
			2\pi\int_0^{\frac\pi2}\!\sin\theta\left[
			\frac{1}{\cos(\theta/2)}
			\,B_{s,\ga}^{(\de)}\!\left( \frac{\abs{z}}{\cos(\theta/2)},\cos\theta\right)
			-
			B_{s,\ga}^{(\de)}(\abs{z},\cos\theta)
			\right]\dd\theta.
		\end{align*}
		Using $b_s(\cos\theta)\sin\theta\sim\theta^{-1-2s}$, we can compute the asymptotics
		\begin{align*}
			&
			\sin\theta\left[ \frac{1}{\cos(\theta/2)}
			\,B_{s,\ga}^{(\de)}\!\left( \frac{\abs{z}}{\cos(\theta/2)},\cos\theta\right)
			-
			B_{s,\ga}^{(\de)}(\abs{z},\cos\theta) \right]
			\\
			&\qquad\qquad\qquad\qquad
			=
			b_s(\cos\theta)\sin(\theta )\chi_\de(\theta)\abs{z}^\ga
			\, \Big( \cos^{-1-\ga}(\theta/2) - 1\Big)
			\\
			&\qquad\qquad\qquad\qquad
			\sim \theta^{1-2s}\chi_\de(\theta)\abs{z}^\ga
		\end{align*}
		and using that $\operatorname{spt}\chi_\de\subset[-2\de,2\de]$, we get the asymptotics
		\begin{align*}
			S_{s,\ga}^{(\de)}(z) \lesssim\de^{2-2s}\abs{z}^\ga,
		\end{align*}
		and we deduce from this that
		\begin{align}
			\label{HTT20_I_2_bound}
			I_2\lesssim \de^{2-2s}\norm{h}_{L^2_\xi(\br\xi^{\ga/2}m)}.
		\end{align}
		
		We now bound the $I_1$ term. Using the bound
		\begin{align*}
			\abs{\xi-\xi_*}^\ga \ge \varepsilon\br{\xi-\xi_*}^\ga - \varepsilon 1_{\abs{\xi-\xi_*}^\ga\le\varepsilon/(1-\varepsilon)}
		\end{align*}
		which holds for $\varepsilon\in(0,1/2)$ and $\ga\in[0,2]$, we follow the computations in the proof of Lemma 4.2 of \cite{HerauTononTristani20} to bound
		\begin{align*}
			I_1
			&\ge
			C\frac\varepsilon2 \int_{\RR^3\times\RR^3\times\SSS^2}\!
			b_s^{(\de)}(\cos\theta)\Mref_*\br{\xi_*}^{-\ga}
			\left( H'\br{\xi'}^{\ga/2} - H\br{\xi}^{\ga/2}\right)^2\dd\sigma\dd\xi_*\dd\xi
			\\
			&\quad
			- C\varepsilon \int_{\RR^3\times\RR^3\times\SSS^2}\!
			b_s^{(\de)}(\cos\theta)\Mref_*\br{\xi_*}^{-\ga} H^2
			\left( \br{\xi'}^{\ga/2} - \br{\xi}^{\ga/2}\right)^2\dd\sigma\dd\xi_*\dd\xi
			\\
			&\quad
			-\varepsilon \int_{\RR^3\times\RR^3\times\SSS^2}\!
			b_s^{(\de)}(\cos\theta)
			1_{\abs{\xi-\xi_*}^\ga\le\varepsilon/(1-\varepsilon)}
			\Mref_*
			\left( H'-H \right)^2\dd\sigma\dd\xi_*\dd\xi
			\\
			&=
			I_{11} - I_{12} - I_{13}.
		\end{align*}
		
		By the definition of the $\dot{H}^{s,\ga,*}_\xi$ norm, we immediately get
		\begin{align*}
			I_{11}\gtrsim \varepsilon\norm{h}_{\dot{H}^{s,\ga,*}_\xi(m)}^2.
		\end{align*}
		
		To treat the second term, we use that the bound
		\begin{align*}
			\left( \br\xi^{\ga/2} - \br{\xi'}^{\ga/2}\right)^2
			\lesssim \sin^2(\theta/2)\br\xi^\ga \br{\xi_*}^{4-\ga},
		\end{align*}
		proved in equation (2.7) of \cite{HerauTononTristani20}, holds for $\ga\in[0,2]$, and therefore we can compute
		\begin{align*}
			I_{12}
			&\lesssim
			\varepsilon \int_{\RR^3\times\RR^3\times\SSS^2}\!
			b_s^{(\de)}(\cos\theta)\sin^2(\theta/2)\Mref_*\br{\xi_*}^{4-2\ga} \br{\xi}^\ga H^2
			\dd\sigma\dd\xi_*\dd\xi,
		\end{align*}
		but this has angular integral
		\begin{align*}
			2\pi \int_0^{\frac\pi2}\! b_s(\cos\theta)\sin(\theta)\sin^2(\theta/2)\chi_\de(\theta)\dd\theta
			\sim\de^{2-2s},
		\end{align*}
		which gives us the bound
		\begin{align*}
			I_{12}\lesssim \varepsilon \de^{2-2s}\norm{h}_{L^2(\br\xi^{\ga/2}m)}^2.
		\end{align*}
		
		To treat the third term, we note that since $\varepsilon<1/2$, we can bound
		\begin{align*}
			I_{13}
			&\le
			\varepsilon \int_{\RR^3\times\RR^3\times\SSS^2}\!
			b_s^{(\de)}(\cos\theta)
			1_{\abs{\xi-\xi_*}\le 1}
			\Mref_*
			\left( H'-H \right)^2\dd\sigma\dd\xi_*\dd\xi
			\\
			&\le
			\varepsilon \int_{\RR^3\times\RR^3\times\SSS^2}\!
			b_s(\cos\theta)
			1_{\abs{\xi-\xi_*}\le 1}
			\Mref_*
			\left( H'-H \right)^2\dd\sigma\dd\xi_*\dd\xi
		\end{align*}
		for $\ga\in[0,2]$ (this is trivially true when $\ga=0$).
		But we can then use Lemma \ref{Hs_bound} to bound this by
		\begin{align*}
			I_{13}\lesssim\varepsilon\norm{h}_{H^s_\xi(m)}^2
		\end{align*}
		using Lemma \ref{Hs_bound}. Combining these estimates gives us the bound
		\begin{align}
			\label{HTT20_I_1_first_bound}
			\frac12 I_1 \ge c_1\varepsilon \norm{h}_{\dot H^{s,\ga,*}_\xi(m)}^2
			- c_2\varepsilon\norm{h}_{H^s_\xi(\br\xi^{\ga/2}m)}^2
		\end{align}
		for some constants $c_1,c_2>0$. The negative term is a high order error term which we would like to replace with an $L^2$ norm, so we compensate \eqref{HTT20_I_1_first_bound} with a second bound on $I_1$. Following arguments from the proof of Theorem 3.1 of \cite{He14}, we decompose
		\begin{align*}
			2 I_1
			&
			\gtrsim
			\int_{\RR^3\times\RR^3\times\SSS^2}\!
			B_{s,\ga}^{(\de)}(\abs{\xi-\xi_*},\cos\theta)
			(\Mref\chi_R)_* (H'-H)^2(1-\chi_{3R}(\xi))^2\dd\sigma\dd\xi_*\dd\xi
			\\
			&\quad
			+
			\int_{\RR^3\times\RR^3\times\SSS^2}\!
			B_{s,\ga}^{(\de)}(\abs{\xi-\xi_*},\cos\theta)
			(\Mref\chi_R)_* (H'-H)^2\frac{\br\xi^\ga}{\br{6R}^\ga}\chi_{3R}(\xi)^2\dd\sigma\dd\xi_*\dd\xi
			\\
			&= I_{11}' +I_{12}'
		\end{align*}
		for some $R>0$. To bound $I_{12}'$, we follow the arguments of \cite{AlexandreDesvillettesVillaniWennberg00}, with more precise control of the angular truncation. We take points $\xi^{(j)}$ for $1\le j\le J$ and define the balls $A_j = B(\xi^{(j)},r_0)$ such that 
		\begin{align*}
			B(0,6R)\subset \bigcup_{j=1}^J A_j
		\end{align*}
		for some small $r_0>0$, and we define $B_j = B(0,R)\setminus B(0,2 r_0)$. We define a smooth partition of unity $(\chi_{A_j})_{j=1}^J$ on $B(0,6R)$, such that $0\le\chi_{A_j}\le 1$ and
		\begin{align*}
			\sum_{j=1}^J \chi_{A_j}=1\qquad\hbox{on }\  B(0,6R),
		\end{align*}
		and we define smooth cutoff functions $\chi_{B_j}$ with support in $B_j$, and such that
		\begin{align*}
			1_{B(0,R-r_0)\setminus B(\xi^{(j)},3 r_0)}\le \chi_{B_j} \le 1_{B_j}.
		\end{align*}
		The proof of the truncation lemma of \cite{AlexandreDesvillettesVillaniWennberg00} then immediately gives us that
		\begin{align*}
			I_{12}'
			&
			\ge
			C_0\int_{\RR^3\times\RR^3\times\SSS^2}\!
			b_{s}^{(\de)}(\cos\theta)
			(\Mref\chi_{B_j})_* \left(H'\br{\xi'}^{\ga/2}{\chi_{A_j}}'-H\br\xi^{\ga/2}\chi_{A_j}\right)^2
			\dd\sigma\dd\xi_*\dd\xi
			\\
			&\qquad
			-
			C_0\norm{H}_{L^2_\xi}^2
			\\
			&=
			C_0 I_{121}' - C_0\de^{2-2s}\norm{h}_{L^2_\xi(m)}^2
		\end{align*}
		where $C_0$ depends on $R,r_0$, and $\ga$. By Corollary 2.1 of \cite{AlexandreDesvillettesVillaniWennberg00}, we can take the Fourier transform and bound this as
		\begin{align*}
			I_{121}' \ge
			\frac{1}{16\pi^3}\int_{\RR^3}\! 
			\abs{(H\br\xi^{\ga/2}\chi_{A_j})^\wedge(\zeta)}^2
			\left( \int_{\SSS^2}\! b_s^{(\de)}\left(\sigma\cdot\frac{\zeta}{\abs\zeta}\right)
			\left( (\Mref\chi_{B_j})^\wedge(0) - (\Mref\chi_{B_j})^\wedge(\zeta^-)\right)
			\dd\sigma \right) \dd\zeta
		\end{align*}
		but by Lemma \ref{truncated_ADVW00_Prop_3} and summing over $1\le j\le J$, we can then bound this below as
		\begin{align*}
			I_{121}' \gtrsim \de^{2-2s}\left( c_0\norm{ h\chi_{3R}}_{H^s_\xi(\br\xi^{\ga/2}m)}^2 -
			c_1\norm{h\chi_{3R}}_{L^2_\xi(\br\xi^{\ga/2}m)}^2
			\right).
		\end{align*}
		
		We can then treat the term $I_{11}'$ as in the proof of Theorem 3.1 in \cite{He14}. Observing that $\abs{\xi - \xi_*}\ge R$ on the support of the integrand, we can bound
		\begin{align*}
			I_{11}'
			&\gtrsim
			\int_{\RR^3\times\RR^3\times\SSS^2}\!
			b_{s}^{(\de)}(\cos\theta)
			(\Mref\chi_R)_* (H'-H)^2(1-\chi_{3R}(\xi))^2\br\xi^\ga\dd\sigma\dd\xi_*\dd\xi.
		\end{align*}
		
		We can again immediately use the proof of the truncation lemma in \cite{AlexandreDesvillettesVillaniWennberg00} and the fact that $\ga\ge0$ to bound this from below by
		\begin{align*}
			&
			\frac12\int_{\RR^3\times\RR^3\times\SSS^2}\!
			b_{s}^{(\de)}(\cos\theta)
			(\Mref\chi_R)_* 
			\left(H'(1-{\chi_{3R}}')\br{\xi'}^{\ga/2} -H(1-\chi_{3R})\br\xi^{\ga/2}\right)^2
			\dd\sigma\dd\xi_*\dd\xi
			\\
			&\quad
			-
			\int_{\RR^3\times\RR^3\times\SSS^2}\!
			b_{s}^{(\de)}(\cos\theta)
			(\Mref\chi_R)_* 
			\left((1-{\chi_{3R}}')\br{\xi'}^{\ga/2} -(1-\chi_{3R})\br\xi^{\ga/2}\right)^2
			{H'}^2
			\dd\sigma\dd\xi_*\dd\xi
			\\
			&\qquad
			=
			I_{111}' - I_{112}'.
		\end{align*}
		As with $I_{121}'$, we can apply Corollary 2.1 of \cite{AlexandreDesvillettesVillaniWennberg00} and Lemma \ref{truncated_ADVW00_Prop_3} to get the lower bound
		\begin{align*}
			I_{111}' \gtrsim
			\de^{2-2s}\norm{h (1-\chi_{3R})}_{H^s_\xi(\br\xi^{\ga/2}m)}^2 - 
			\de^{2-2s} \norm{h (1-\chi_{3R})}_{L^2_\xi(\br\xi^{\ga/2}m)}^2.
		\end{align*}
		Meanwhile, the proof of Theorem 3.1 of \cite{He14} immediately gives the bound
		\begin{align*}
			I_{112}'
			&\lesssim
			\int_{\RR^3\times\RR^3\times\SSS^2}\!
			b_{s}^{(\de)}\ 
			(\Mref\chi_R)_* \br{\xi_*}^2\br{\xi'}^{\ga}\sin(\theta/2)^2{H'}^2\dd\sigma\dd\xi_*\dd\xi
			\\
			&\lesssim
			\de^{2-2s}
			\norm{h}_{L^2_\xi(\br\xi^{\ga/2}m)}^2,
		\end{align*}
		for any $\ga\in[0,2]$, where the first estimate is proved by a mean value argument and the second follows by a pre-post collisional change of variables, and by collecting terms we get the bound
		\begin{align*}
			\frac12I_1
			\gtrsim c_2 \de^{2-2s}\norm{h }_{H^s_\xi(\br\xi^{\ga/2}m)}^2 - 
			c_3 \de^{2-2s} \norm{h }_{L^2_\xi(\br\xi^{\ga/2}m)}^2.
		\end{align*}
		Combining this with the bounds \eqref{HTT20_I_2_bound} and \eqref{HTT20_I_1_first_bound} gives the estimate
		\begin{align*}
			\br{Q_{s,\ga}^{(\de)}(\Mref, H),H}_{L^2_\xi}
			&\le
			- c_1\epsilon\norm{h}_{\dot H^{s,\ga,*}_\xi(m)} - \left( c_2\de^{2-2s} - c_3\right)\norm{h}^2_{H^s_\xi(\br\xi^{\ga/2}m)} 
			\\
			&\qquad\qquad
			+ c_4\de^{2-2s}\norm{h}_{L^2_\xi(\br\xi^{\ga/2}m)}
		\end{align*}
		for some constants $c_1,c_2,c_3,c_4>0$. The rest of the proof of Lemma 4.2 of \cite{HerauTononTristani20} is then self-contained and holds for $\ga\in[0,2]$, giving the bounds
		\begin{align*}
			\br{Q_{s,\ga}^{(\de),\mrm{c}}(\Mref, h),h}_{L^2_\xi(m)}
			\lesssim
			\de^{-2s}\left( c_5\de^{1/2} - \nu_0\right)
			\norm{h}_{L^2_\xi(\br\xi^{\ga/2}m)}^2 + C_\de\norm{h}_{L^2_\xi}^2
		\end{align*}
		and,
		\begin{align*}
			R\lesssim c_6 \norm{h}_{L^2_\xi(\br\xi^{\ga/2}m)}^2,
		\end{align*}
		while Lemma \ref{Lemma2.3_HTT20_extended} for $k>\gamma/2+2+\max(1/2,2s)$ then gives us the bound
		\begin{align*}
			\br{Q_{s,\ga}(h,\Mref),h}_{L^2_\xi(m)}^2
			\lesssim
			c_7 \norm{h}_{L^2_\xi(\br\xi^{\ga/2}m)}^2
		\end{align*}
		for some constants $c_5,\nu_0,C_\de,c_6,c_7>0$.
		Gathering these terms, we derive the estimate
		\begin{align*}
			\br{\Lref_{s,\ga} h,h}_{L^2_\xi(m)}
			&\le
			- \left( c_2\de^{2-2s} - c_3\epsilon\right)\norm{h}^2_{H^s_\xi(\br\xi^{\ga/2}m)} 
			- c_1\epsilon\norm{h}_{\dot H^{s,\ga,*}_\xi(m)} 
			\\
			&
			+
			\left( c_6 + \de^{-2s}\left( c_4\de^2+c_5\de^{1/2} - \nu_0\right)\right)
			\norm{h}^2_{L^2_\xi(\br\xi^{\ga/2}m)} + C_\de\norm{h}_{L^2_\xi}^2,
		\end{align*}
		and taking $\de$ sufficiently small, and taking small $\varepsilon\lesssim\de^{2-2s}$ proves the lemma.
	\end{proof}

	We now assume \eqref{Hypothesis3} for the rest of this section, so that $s\in(0,1/2)$ and $\ga\in[0,2]$ and $\Lref_\kappa = \Lref_{s,\gamma} + \kappa \Lref_{s,2-2s}$, and we decompose $\Lref_\kappa = \ca A+ \ca B_\kappa $ for
	\begin{align*}
		\ca A = M\chi_R,\qquad\ca B_\kappa = \Lref_\kappa - M\chi_R
	\end{align*}
	where the parameters $M,R>0$ may vary. When $\kappa=0$, a coercivity result for $\ca B_\kappa$ was proved in Lemma 4.3 of \cite{HerauTononTristani20}. We extend this to a uniform coercivity result in $\kappa$.
	
	\begin{lemma}
		\label{semigroup_extension_coercivity}
		Let $k\ge4$ and $a<0$, and let $\kappa_0>0$. Then there exist $M,R>0$ independent of $\kappa\in[0,\kappa_0]$ such that $\ca B_\kappa$ satisfies the coercivity estimate
		\begin{align*}
			\br{\ca B_\kappa f,f}_{L^2(m)}\le a\norm{f}_{L^2(\br\xi^{\ga/2}m)}^2.
		\end{align*}
	\end{lemma}
	\begin{proof}
		Using Lemma \ref{Lemma4.2_HTT20_extended} and the definition of $\Lref_\kappa$, we can estimate
		\begin{align*}
			\br{\Lref_\kappa h,h}_{L^2(m)}
			&\le -c_0\de^{2-2s}\left(\norm{h}_{L^2(\br\xi^{\ga/2}m)}^2 + \kappa\norm{h}_{L^2(\br\xi^{1-s})}^2\right) + C_\de'\norm{h}_{L^2}^2
		\end{align*}
		and therefore we can express
		\begin{align*}
			\br{\ca B_\kappa h,h}_{L^2(m)}\le
			\int_{\RR^3}\!\left( -c_0\de^{2-2s}\left( \br\xi^\ga + \kappa\br\xi^{2-2s} \right) + C_\de - M\chi_R(\xi)\right)\abs{h}^2\dd\xi,
		\end{align*}
		and by taking $M,R>0$ sufficiently large, the statement is immediate.
	\end{proof}
	
	We can now prove the uniform invertibility lemma on $L^2(\br\xi^k)$.
	
	\begin{proof}[Proof of Lemma \ref{HTT20_Lemma}]
		We first prove uniform invertibility estimates $\Lref_\kappa$, and extend to the uncentred operators by a scaling argument. The operator $\Lref_\kappa$ is symmetric and negative definite on $L^2(\Mref^{-1/2})$ away from its kernel, and therefore admits a self-adjoint Friedrichs extension (cf. Theorem 10.17 of \cite{Schmudgen_book}) and a bounded semigroup $e^{t \Lref_\kappa}$ on $L^2(\Mref^{-1/2})$. Since $\ca A: L^2(\br\xi^k)\to L^2(\Mref^{-1/2})$ is a bounded operator for any $k$, using the coercivity estimate in Lemma \ref{semigroup_extension_coercivity}, we can apply Theorem 2.13 of \cite{GualdaniMischlerMouhot17}, and we derive the uniform bound
		\begin{align*}
			\norm{\Lref_\kappa^{-1} f}_{L^2(\br\xi^k)}\le C_k\norm{f}_{L^2(\br\xi^k)}
		\end{align*} 
		for any microscopic $f\in L^2(\br\xi^k)$. 
		
		To prove invertibility of $L_{\mathfrak{u},\kappa}$ for $\abs{\mathfrak{u}-\uref}\le\epsilon_0$, we express $(\rho,\bfu,T) = \mathsf{f}_0^{-1}\circ I[\mathfrak u]$ in hydrodynamic variables, and change the velocity coordinates $\mathsf{V} = T^{-1/2}(\xi-\bfu e_1)$, but we see that this induces an invertible continuous operator on $L^2(\br\xi^k)$, and by the scaling identity \eqref{LUkappa_temperature_dependence}, this proves the claim.
	\end{proof}
	
	Finally, we prove the following technical lemma which we have used in the proof of Lemma \ref{Lemma4.2_HTT20_extended}, which also provides an elementary connection between the non-cutoff Boltzmann operator and fractional Sobolev spaces.
	
	\begin{lemma}
		\label{Hs_bound}
		For $g\in L^\infty$, we have the bound
		\begin{align*}
			\int_{\RR^3\times\RR^3\times\SSS^2}\! b_s(\cos\theta) 
			1_{\abs{\xi-\xi_*}\le R}
			g_*( h' - h )^2 \dd\sigma\dd\xi_*\dd\xi
			\le C\norm{h}_{H^s}^2.
		\end{align*}
		where $C$ depends on $\norm{g}_{L^\infty}$ and $R>0$.
	\end{lemma}
	This was proved using Littlewood-Paley theory in the proof of Theorem 3.1 of \cite{He14}. We provide here an elementary proof.
	\begin{proof}
		We define $k = (\xi-\xi_*)/\abs{\xi-\xi_*}$ and $r = \abs{\xi-\xi_*}$, which gives
		\begin{align*}
			\xi' = \frac{r}{2}(\sigma - k).
		\end{align*}
		By replacing $g$ with $\norm{g}_{L^\infty}$, we see that we can take $g=1$.
		We can then rewrite the integral as
		\begin{align*}
			\int_{\RR^3\times\SSS^2\times\SSS^2}\!\int_0^R\!
			b_s(\sigma\cdot k) 
			\left[ h\left(\xi - \frac{r}{2}(\sigma-k)\right)- h(\xi) \right]^2
			r^2\dd r\dd k\dd\sigma\dd\xi.
		\end{align*}
		
		We have the identity
		\begin{align*}
			\int_{\SSS^2\times\SSS^2}\! \phi\left( \abs{\frac12(\sigma-k)}\right)\dd\sigma\dd k
			&=
			\int_{\SSS^2\times\SSS^2}\! \phi\left( \sqrt{\frac{1-\sigma\cdot k}{2}}\right)\dd\sigma\dd k
			\\
			&=
			8\pi^2\int_0^\pi \phi\left( \sqrt{\frac{1-\cos\theta}{2}}\right)\sin\theta\dd\theta
			\\
			&=
			32\pi^2\int_0^1\!\phi(z)z\dd z
		\end{align*}
		for $z = \sin(\theta/2)$, so we can rewrite the previous integral as
		\begin{align*}
			32\pi^2\int_{\RR^3\times\SSS^2}\!\int_0^1\!\int_0^R\!
			b_s(1-2z^2) [ h(\xi - rz\sigma) - h(\xi)] r^2\dd r\dd z\dd\sigma\dd\xi
		\end{align*}
		but using the change of variables $y=rz$, we rewrite this as
		\begin{align*}
			&
			32\pi^2\int_{\RR^3\times\SSS^2}\!\iint_{0\le R^{-1} y\le z\le 1}\!
			b_s(1-2z^2) [ h(\xi - y\sigma) - h(\xi)] y z^{-1}\dd y\dd z\dd\sigma\dd\xi
			\\
			&\qquad\lesssim
			\int_{\RR^3\times\SSS^2}\!\iint_{0\le R^{-1} y\le z\le 1}\!
			z^{-3-2s}[h(\xi-y\sigma)-h(\xi)] y\dd y\dd z\dd\sigma\dd\xi
			\\
			&\qquad=
			C\int_{\RR^3\times\SSS^2}\!\int_0^R\!
			(R^{2+2s}y^{-2-2s}-1)[h(\xi-y\sigma)-h(\xi)] y\dd y\dd\sigma\dd\xi
			\\
			&\qquad\le
			C_R'
			\int_{\RR^3\times\RR^3}\! 1_{\abs{\xi-\xi'}\le R}
			\frac{(h(\xi')-h(\xi))^2}{\abs{\xi-\xi'}^{3+2s}}\dd\xi'\dd\xi
			\\
			&\qquad\lesssim \norm{h}_{H^s}^2
		\end{align*}
		where we have used the definition of the fractional Sobolev space in \cite{Hitchhikers_guide}.
	\end{proof}

	\bibliographystyle{acm}
	\bibliography{references}

\end{document}